\documentclass[11pt]{article}

\usepackage{geometry}
\usepackage{latexsym}
\usepackage{mathrsfs}
\usepackage{amsfonts}
\usepackage{amssymb}
\usepackage{subfigure}    
\usepackage{graphicx}
\usepackage{epstopdf}
\usepackage{float}
\usepackage{multirow}
\usepackage{bm}
\usepackage{amsmath}
\usepackage{amsthm}
\usepackage{color}
\usepackage[subnum]{cases}
\usepackage{rotating}
\usepackage{enumitem}
\usepackage{accents}
\usepackage{algorithm}
\usepackage{algorithmic}
\usepackage[title]{appendix}
\usepackage{mathtools}
\usepackage{caption}
\usepackage{url}
\usepackage{booktabs}

\makeatletter
\@addtoreset{equation}{section}
\makeatother

\makeatletter
\def \wideubar{\underaccent{{\cc@style\underline{\mskip15mu}}}}
\def \widebar{\accentset{{\cc@style\underline{\mskip10mu}}}}
\makeatother

\definecolor{blue}{rgb}{0,0,0.9}
\definecolor{red}{rgb}{0.9,0,0}
\definecolor{green}{rgb}{0,0.9,0}
\definecolor{brown}{rgb}{0.6,0.1,0.1}
\definecolor{lightgreen}{rgb}{0.1,0.5,0.1}

\newcommand{\red}[1]{\begin{color}{red}#1\end{color}}

\usepackage[colorlinks=true,
breaklinks=true,
bookmarks=true,
urlcolor=blue,
citecolor=lightgreen,
linkcolor=lightgreen,
bookmarksopen=false,
draft=false]{hyperref}

\graphicspath{{images/}}

\begin{document}

\newtheorem{property}{Property}[section]
\newtheorem{proposition}{Proposition}[section]
\newtheorem{append}{Appendix}[section]
\newtheorem{definition}{Definition}[section]
\newtheorem{lemma}{Lemma}[section]
\newtheorem{corollary}{Corollary}[section]
\newtheorem{theorem}{Theorem}[section]
\newtheorem{remark}{Remark}[section]
\newtheorem{problem}{Problem}[section]
\newtheorem{example}{Example}
\newtheorem{assumption}{Assumption}
\renewcommand*{\theassumption}{\Alph{assumption}}

\renewcommand{\vartheta}{\widehat{\theta}} 
\def\myfbox#1{\fbox{\begin{minipage}{15cm}\red{#1}\end{minipage}}\\ }

\title{Inexact Bregman Proximal Gradient Method and its Inertial Variant with Absolute and Partial Relative Stopping Criteria}

\author{
Lei Yang\thanks{School of Computer Science and Engineering, and Guangdong Province Key Laboratory of Computational Science, Sun Yat-sen University ({\tt yanglei39@mail.sysu.edu.cn}). The research of this author is supported in part by the National Natural Science Foundation of China under grant 12301411, the Natural Science Foundation of Guangdong under grant 2023A1515012026, and the Basic and Applied Basic Research Foundation of Guangzhou under grant 2024A04J4184.
},
\quad
Kim-Chuan Toh\thanks{Department of Mathematics, and Institute of Operations Research and Analytics, National University of Singapore, Singapore 119076 ({\tt mattohkc@nus.edu.sg}). The research of this author is supported in part by the Ministry of Education, Singapore, under Academic Research Fund Tier 1 Grant A00084930000.
}
}	

\maketitle

\begin{abstract}
The Bregman proximal gradient  method (BPGM), which uses the Bregman distance as a proximity measure in the iterative scheme, has recently been re-developed for minimizing convex composite problems \textit{without} the global Lipschitz gradient continuity assumption. This makes the BPGM appealing for a wide range of applications, and hence it has received growing attention in recent years. However, most existing convergence results are only obtained under the assumption that the involved subproblems are solved \textit{exactly}, which is unrealistic in many applications and limits the applicability of the BPGM. To make the BPGM implementable and practical, in this paper, we develop inexact versions of the BPGM (denoted by iBPGM) by employing either an absolute-type stopping criterion or a partial relative-type stopping criterion for solving the subproblems. The $\mathcal{O}(1/k)$ convergence rate and the convergence of the sequence are also established for our iBPGM under some conditions. Moreover, we develop an inertial variant of our iBPGM (denoted by v-iBPGM) and establish the $\mathcal{O}(1/k^{\gamma})$ convergence rate, where $\gamma\geq1$ is a restricted relative smoothness exponent depending on the smooth function in the objective and the kernel function. Specially, when the smooth function in the objective has a Lipschitz continuous gradient and the kernel function is strongly convex, we have $\gamma=2$ and thus the v-iBPGM improves the convergence rate of the iBPGM from $\mathcal{O}(1/k)$ to $\mathcal{O}(1/k^2)$, in accordance with the existing results on the exact accelerated BPGM. Finally, some preliminary numerical experiments for solving the discrete quadratic regularized optimal transport problem are conducted to illustrate the convergence behaviors of our iBPGM and v-iBPGM under different inexactness settings.

\vspace{5mm}


\end{abstract}

\section{Introduction}

In this paper, we consider the following convex composite problem:
\begin{equation}\label{mainpro}
\min\limits_{\bm{x}\in\mathcal{Q}}~F(\bm{x}) := P(\bm{x}) + f(\bm{x}),
\end{equation}
where $\mathcal{Q}\subseteq\mathbb{E}$ is a closed convex set with nonempty interior denoted by $\mathrm{int}\,\mathcal{Q}$, and $\mathbb{E}$ is a real finite dimensional Euclidean space equipped with an inner product $\langle\cdot,\cdot\rangle$ and its induced norm $\|\cdot\|$. The functions $f,\,P: \mathbb{E} \rightarrow (-\infty,\infty]$ are proper closed convex functions and $f$ is differentiable on $\mathrm{int}\,\mathcal{Q}$. Problem \eqref{mainpro} is a generic optimization problem that arises in many application areas such as machine learning, data science, image/signal processing. It has been extensively studied in the literature; see, for example, \cite{bbt2017descent,bt2009a,cd2015convergence, hrx2018accelerated,jst2012inexact,n2013gradient,srb2011convergence,t2018simplified, t2010approximation}.

A classical method for solving problem \eqref{mainpro} is the proximal gradient  method (PGM) \cite{fm1981a,lm1979splitting} (also known as the forward-backward splitting method), whose basic iterative step is given by
\begin{equation*}
\bm{x}^{k+1}=\mathop{\mathrm{argmin}}\limits_{\bm{x}\in\mathcal{Q}} \left\{P(\bm{x}) + \langle \nabla f(\bm{x}^k), \,\bm{x}-\bm{x}^k\rangle + \frac{L}{2}\|\bm{x}-\bm{x}^k\|^2\right\},
\end{equation*}
where $L>0$ is a constant depending on the Lipschitz constant of $\nabla f$. This scheme is developed based on the construction of a quadratic upper approximation of the smooth part $f$. Therefore, a central assumption required in the development and analysis of the PGM is that $\nabla f$ is Lipschitz continuous on $Q$. Moreover, one can consider the Bregman distance $\mathcal{D}_{\phi}$ associated with a kernel function $\phi$ as a proximity measure (see next section for definition), and naturally generalize the PGM to the Bregman proximal gradient method (BPGM)  (see, for example, \cite{at2006interior,bbt2017descent,lu2018relatively,t2008on,t2010approximation}), whose basic iterative step reads as follows:
\begin{equation}\label{bpgscheme}
\bm{x}^{k+1} = \mathop{\mathrm{argmin}}\limits_{\bm{x}\in\mathcal{Q}} \left\{P(\bm{x}) + \langle \nabla f(\bm{x}^k), \,\bm{x}-\bm{x}^k\rangle + L\mathcal{D}_{\phi}(\bm{x},\,\bm{x}^k)\right\}.
\end{equation}
Compared to the PGM, the iterative scheme of the BPGM is more flexible. A suitable choice of the kernel function $\phi$ can help explore the underlying geometry or structure of the problem and potentially lead to a more tractable subproblem \eqref{bpgscheme} (see, for example, \cite[Section 5]{bt2003mirror} and
\cite[Lemma 4]{n2005smooth}). More importantly, as investigated in recent insightful works \cite{bbt2017descent,lu2018relatively}, the
BPGM can be developed based on the notion of \textit{relative smoothness}, which is weaker than the longstanding global Lipschitz gradient continuity assumption. This makes the BPGM applicable for a wider range of problems, such as the Poisson inverse problem \cite{bbt2017descent}, D-optimal design \cite{lu2018relatively}, symmetric nonnegative matrix factorization \cite{ddb2021quartic}, orthogonal nonnegative matrix factorization \cite{ahgp2021multi}, symmetric nonnegative matrix tri-factorization \cite{ahgp2021block}, deep linear neural networks \cite{mwlco2019bregman}, and sparse phase retrieval with Poisson loss \cite{ltap2022bregman}. Additionally, Ahookhosh and Nesterov \cite{an2024high} recently extended the results of \cite{n2023inexact} and shown that employing an (accelerated) inexact high-order proximal-point method at the upper level, while solving the proximal subproblem by the BPGM at the lower level, can result in a superfast method for solving problem \eqref{mainpro}. Due to this appealing potential, the BPGM and its variants have garnered increasing attention in recent years; see, for example, \cite{atp2021bregman,bbctw2019linear,bbt2017descent,bstv2018first,dtdb2019optimal,gp2018perturbed,lu2018relatively,t2018simplified,v2017forward,zls2019simple}.

Unfortunately, although the underlying geometry/structure can possibly be captured by the kernel function, the subproblem \eqref{bpgscheme} still has
no closed-form solution in many applications and its computation may still be demanding numerically; see examples in \cite{cv2018entropic,jscb2020efficient,rd2020bregman} and the one in our numerical section. Thus, to make the BPGM truly implementable and practical, it must allow one to solve the subproblem \textit{approximately} with progressive accuracy and the corresponding stopping criterion should be practically \textit{verifiable}. However, unlike the inexact PGM which has been widely studied under different inexact criteria (see, for example, \cite{cw2005signal,srb2011convergence,vsbv2013accelerated}), the study on the inexact BPGM is still at an early stage. Rebegoldi et al. \cite{rbp2018bregman} considered a more general problem of the form \eqref{mainpro} by allowing a nonconvex $f$ and proposed a Bregman inexact linesearch–based forward–backward algorithm, where an $\varepsilon$-minimizer for the subproblem \eqref{bpgscheme} is permitted at each iteration. However, they only establish the subsequential convergence for the proposed inexact algorithm.
Moreover, their analysis requires the assumption that the kernel function $\phi$ satisfies the condition $0\in\mathrm{int}\,\mathrm{dom}\,\phi$. This assumption may exclude several commonly used kernel functions, such as the Boltzmann-Shannon entropy kernel function $\phi(\bm{x}):=\sum_{i=1}^nx_{i}(\log x_{i}-1)$ and the Burg's entropy kernel function $\phi(\bm{x}):=-\sum_{i=1}^n\log x_{i}$. Subsequently, Kabbadj \cite{k2020inexact} developed an inexact BPGM that also permits an $\varepsilon$-minimizer for each subproblem \eqref{bpgscheme} and established stronger convergence results under the relative smoothness assumption. However, since the optimal objective value of subproblem \eqref{bpgscheme} is generally unknown, verifying an $\varepsilon$-minimizer may not be straightforward in practice. Recently, Stonyakin et al. \cite{stonyakin2021inexact} proposed a general inexact algorithmic framework for first-order methods containing the BPGM as a special case, where an approximate solution of the subproblem \eqref{bpgscheme} is accepted when a certain inexact condition
\cite[Definition 2.12]{stonyakin2021inexact} holds. In our context, their inexact condition can be covered by our absolute-type stopping criterion (AbSC) with the additional requirements that $\delta_k\equiv0$ and $\bm{x}^{k+1}\equiv\widetilde{\bm{x}}^{k+1}$ (see \textbf{Step 1} in Algorithm \ref{algo-iBPGM}). However, this condition might be difficult to verify in some practical cases; see Remark \ref{rem-feas} for more discussions. In addition, all the aforementioned inexact algorithms in \cite{k2020inexact,rbp2018bregman,stonyakin2021inexact} do not accommodate the use of a relative-type stopping criterion, which limits their flexibility in practical applications.

Over the last few decades, Nesterov's works \cite{n1983a,n1988on,n2003introductory,n2005smooth} on accelerated gradient methods have inspired numerous accelerated variants of the PGM (see, e.g.,  \cite{bt2009a,bpr2016variable,n2013gradient,t2008on,t2010approximation}) as well as their inexact counterparts (see, e.g., \cite{ad2015stability,bgk2023fista,brr2018inertial,jst2012inexact,srb2011convergence,vsbv2013accelerated}). Naturally, those developments also motivated many researchers to explore whether and how the BPGM can be accelerated. A successful attempt was made by Auslender and Teboulle \cite[Section 5]{at2006interior}, who proposed an improved interior gradient algorithm (a Bregman gradient scheme) for solving a special case of problem \eqref{mainpro} with $P\equiv0$ and showed that this algorithm can achieve a faster convergence rate of $\mathcal{O}(1/k^2)$. Later on, it was extended by Tseng \cite{t2008on,t2010approximation} to handle problem \eqref{mainpro} in a more general composite setting. Other accelerated variants of the BPGM can also be found in \cite{llm2011primal,t2008on,t2010approximation}. Note that all these works \cite{at2006interior,llm2011primal,t2008on,t2010approximation} just mentioned require the assumptions that $f$ has a global Lipschitz gradient and that
the kernel function of the Bregman distance is strongly convex. Until recently, based on the relative smoothness assumption and a so-called triangle scaling property of the Bregman distance, Hanzely, Richt\'{a}rik and Xiao \cite{hrx2018accelerated} developed an accelerated BPGM that attains a convergence rate of $\mathcal{O}(1/k^{\gamma})$, where $\gamma\in(0,2]$ is the triangle scaling exponent. Later, a similar result was also derived under a slightly broader smoothness condition by Gutman and Pe{\~n}a \cite[Section 3.3]{gp2018perturbed}. Very recently, Dragomir et al. \cite{dtdb2019optimal} established a crucial result that the convergence rate of $\mathcal{O}(1/k)$ is indeed \textit{optimal} for the BPGM to solve the class of problems merely satisfying the relative smoothness assumption, and thus the BPGM cannot be accelerated if no other regularity condition is imposed. But empirically improving the BPGM via an inertial step and a certain adaptive backtracking strategy is still possible, as discussed in \cite{hrx2018accelerated,mops2020convex}.

However, among the aforementioned accelerated works that have investigated the possibility of improving the BPGM, none of them take into account of the possible errors incurred when one solves the involved subproblem \textit{approximately}. Indeed, as far as we know, the rigorous study on the inexact inertial BPGM is very limited. Stonyakin et al. \cite{stonyakin2021inexact} have developed an inexact inertial variant of the BPGM in their work. However, as we have mentioned before, the inexact stopping criterion used in \cite{stonyakin2021inexact} can be covered by our absolute-type stopping criterion and may still require improvements regarding its practical verification. Moreover, we also note that Hien et al. \cite{hpgap2022block} recently proposed a block alternating Bregman majorization-minimization framework with extrapolation, which gives an inertial BPGM if only one block of variables is considered, and employed a surrogate function for $P$ to obtain a more tractable auxiliary subproblem which possibly admits a closed-form solution. But such a surrogate function could be nontrivial to find when $P$ is complex, and the auxiliary subproblem needs to be solved \textit{exactly}. Thus, the kind of inexactness considered in \cite{hpgap2022block} is very different from what we consider in this paper. Moreover, replacing $P$ by its surrogate might also potentially slow down the BPGM, especially when the surrogate is not a good approximation of $P$.

In this paper, to facilitate the practical implementations of the BPGM and its inertial variant, we attempt to develop their inexact counterparts and provide theoretical insights on how the error incurred in the inexact minimization of the subproblem would affect the convergence rate in terms of the objective function value. As we shall see later, our inexact framework is developed based on \textit{either} an absolute-type stopping criterion \textit{or} a partial relative-type stopping criterion, and both of them also have a two-point feature. Thus, the resulting inexact framework is rather broad and can handle different types of errors that may occur when solving the subproblem, and it is also helpful to circumvent the underlying feasibility difficulty in evaluating the Bregman distance when dealing with the problem that has a complex feasible set; see Remark \ref{rem-feas} for more details.

\subsection{Contributions}

The contributions of this paper are summarized as follows.
\begin{itemize}[leftmargin=1.15cm]
\item[{\bf 1.}] We develop an inexact Bregman proximal gradient method (iBPGM) based on \textit{either} an absolute-type stopping criterion \textit{or} a partial relative-type stopping criterion, both of which are distinct from the existing ones in \cite{k2020inexact,rbp2018bregman,stonyakin2021inexact} and are practically verifiable. The $\mathcal{O}(1/k)$ convergence rate and the convergence of the sequence are also established for our iBPGM under some proper conditions.

\item[{\bf 2.}] We develop an inertial variant of our iBPGM (denoted by v-iBPGM) and establish the $\mathcal{O}(1/k^{\gamma})$ ($\gamma\geq1$) convergence rate under an additional restricted relative smoothness assumption (see Assumption \ref{assumC}). Our result can subsume the related results in \cite{at2006interior,hrx2018accelerated,t2008on,t2010approximation} wherein the subproblem is solved \textit{exactly}, and as a byproduct, our analysis indeed provides a unified treatment of these existing results which are developed under different conditions. In particular, when the smooth part in the objective has a Lipschitz continuous gradient and the kernel function is strongly convex, we have $\gamma=2$, and thus the v-iBPGM improves the convergence rate of the iBPGM from $\mathcal{O}(1/k)$ to $\mathcal{O}(1/k^2)$.

\item[{\bf 3.}] We conduct preliminary numerical experiments to evaluate the performances of our iBPGM and v-iBPGM under different inexactness settings. The computational results empirically validate the necessity and significance of developing inexact versions for the BPGM and its inertial variant. Furthermore, they reveal that different methods with different types of stopping criteria have different inherent inexactness tolerance requirements, aligning with our theoretical findings.

\end{itemize}

\subsection{Related works}

The two-point-type inexact framework developed in this paper is motivated by our recent works \cite{clty2020an,yt2020bregman} on a new inexact Bregman proximal point algorithm (iBPPA), which solves problem \eqref{mainpro} via the following iterative scheme: at the $k$-th iteration, find a pair $(\bm{x}^{k+1},\,\widetilde{\bm{x}}^{k+1})$ and an error pair $(\Delta^k, \delta_k)$ by approximately solving
\begin{equation}\label{BPPAsubpro-gen}
\min\limits_{\bm{x}}~
P(\bm{x}) + f(\bm{x}) + \gamma_k\mathcal{D}_{\psi}(\bm{x}, \,\bm{x}^k),
\end{equation}
such that $\bm{x}^{k+1}\in\mathrm{int}\,\mathcal{Q}$, $\widetilde{\bm{x}}^{k+1}\in\mathrm{dom}\,(P+f)\cap\mathcal{Q}$, and
\begin{equation}\label{BPPAcond-gen}
\begin{aligned}
&\Delta^{k} \in \partial_{\delta_k}(P+f)(\widetilde{\bm{x}}^{k+1})
+ \gamma_k\big(\nabla\psi(\bm{x}^{k+1}) - \nabla\psi(\bm{x}^{k})\big) \\
&~~\mathrm{with}~~\|\Delta^{k}\| \leq \eta_k,
~~\delta_k\leq\nu_k, ~~\mathcal{D}_{\psi}(\widetilde{\bm{x}}^{k+1}, \,\bm{x}^{k+1}) \leq \mu_k,
\end{aligned}
\end{equation}
where $\{\gamma_k\}$ is a given proximal parameter sequence, and $\{\eta_k\}$, $\{\nu_k\}$, $\{\mu_k\}$ are three given tolerance parameter sequences. By the definition of the Bregman distance, it can be easily verified that
\begin{equation*}
P(\bm{x}) + f(\bm{x}) + \mathcal{D}_{\psi}(\bm{x}, \,\bm{x}^k)
= P(\bm{x}) + f(\bm{x}^k) + \langle \nabla f(\bm{x}^k), \,\bm{x}-\bm{x}^k\rangle + \lambda\mathcal{D}_{\phi}(\bm{x},\,\bm{x}^k),
\end{equation*}
when $\psi:=\lambda\phi-f$ for any $\lambda>0$. Therefore, the subproblem \eqref{subpro-iBPGM} in the iBPGM developed later in Algorithm \ref{algo-iBPGM} is equivalent to the subproblem \eqref{BPPAsubpro-gen} with $\psi:=\lambda\phi-f$ and $\gamma_k\equiv1$ in the above iBPPA. This demonstrates that, when the \textit{exact} minimization of the subproblem \eqref{subpro-iBPGM} is achievable, the exact BPGM can be viewed as an exact BPPA with a special kernel $\psi$ that depends on $f$. Such an equivalence has also been pointed out in
\cite[Theorem 4.1]{atp2021bregman}. However, when considering the \textit{inexact} minimization of each subproblem, this equivalence may no longer hold. Indeed, with $\psi:=\lambda\phi-f$ and $\gamma_k\equiv1$, the inexact condition \eqref{BPPAcond-gen} used in the above iBPPA becomes
\begin{equation*}
\begin{aligned}
&\Delta^{k} \in \partial_{\delta_k}(P+f)(\widetilde{\bm{x}}^{k+1})
+ \lambda\big(\nabla \phi(\bm{x}^{k+1})-\nabla \phi(\bm{x}^{k})\big)
- \big(\nabla f(\bm{x}^{k+1}) - \nabla f(\bm{x}^{k})\big)  \\
&~~\mathrm{with}~~\|\Delta^{k}\| \leq \eta_k,
~~\delta_k\leq\nu_k, ~~\mathcal{D}_{\lambda\phi-f}(\widetilde{\bm{x}}^{k+1}, \,\bm{x}^{k+1}) \leq \mu_k,
\end{aligned}
\end{equation*}
which is not equivalent to, nor implied by, condition \eqref{inexcond-iBPGM} satisfying an absolute-type stopping criterion employed in our iBPGM, when $\widetilde{\bm{x}}^{k+1}\neq\bm{x}^{k+1}$. Consequently, the derivation and analysis in \cite{clty2020an,yt2020bregman} are not directly applicable to our iBPGM. In addition, we introduce a relative-type stopping criterion in this paper, which is not addressed in \cite{clty2020an,yt2020bregman}. As a result, by setting $f\equiv0$, our iBPGM with this new relative-type stopping criterion also provides a new inexact version of the BPPA.

\subsection{Organization}

The rest of this paper is organized as follows. We present notation and preliminaries in Section \ref{secnot}. We then describe our iBPGM for solving problem \eqref{mainpro} and establish the convergence results in Section \ref{sec-iBPGM}. A possibly faster inertial variant of our iBPGM is developed and analyzed in Section \ref{sec-v-iBPGM}. Some preliminary numerical results are reported in Section \ref{secnum}, with some concluding remarks given in Section \ref{seccon}.

\section{Notation and preliminaries}\label{secnot}

Assume that $f: \mathbb{E} \rightarrow (-\infty, \infty]$ is a proper closed convex function. For a given $\varepsilon \geq 0$, the $\varepsilon$-subdifferential of $f$ at $\bm{x}\in{\rm dom}\,f$ is defined by $\partial_{\varepsilon} f(\bm{x}):=\{\bm{d}\in\mathbb{E}: f(\bm{y}) \geq f(\bm{x}) + \langle \bm{d}, \,\bm{y}-\bm{x} \rangle - \varepsilon, ~\forall\,\bm{y}\in\mathbb{E}\}$, and when $\varepsilon=0$, $\partial_{\varepsilon} f$ is simply denoted by $\partial f$, which is referred to as the subdifferential of $f$. The conjugate function of $f$ is the function $f^*: \mathbb{E} \rightarrow (-\infty,\infty]$ defined by $f^*(\bm{y}):=\sup\left\{\langle \bm{y},\,\bm{x}\rangle-f(\bm{x}): \bm{x}\in\mathbb{E}\right\}$. A proper closed convex function $f$ is \textit{essentially smooth} if (i) $\mathrm{int}\,\mathrm{dom}\,f$ is not empty; (ii) $f$ is differentiable on $\mathrm{int}\,\mathrm{dom}\,f$; (iii) $\|\nabla f(x_k)\|\to\infty$ for every sequence $\{x_k\}$ in $\mathrm{int}\,\mathrm{dom}\,f$ converging to a boundary point of $\mathrm{int}\,\mathrm{dom}\,f$; see \cite[page 251]{r1970convex}.

For a vector $\bm{x}\in\mathbb{R}^n$, $x_i$ denotes its $i$-th entry, $\mathrm{Diag}(\bm{x})$ denotes the diagonal matrix whose $i$th diagonal entry is $x_i$, $\|\bm{x}\|$ denotes its $\ell_2$ (Euclidean) norm. For a matrix $A\in\mathbb{R}^{m\times n}$, $a_{ij}$ denotes its $(i,j)$th entry, $A_{:j}$ denotes its $j$th column, $\|A\|_F$ denotes its Frobenius norm. For a closed convex set $\mathcal{X}\subseteq\mathbb{E}$, its indicator function $\iota_{\mathcal{X}}$ is defined by $\iota_{\mathcal{X}}(\bm{x})=0$ if $\bm{x}\in\mathcal{X}$ and $\iota_{\mathcal{X}}(\bm{x})=+\infty$ otherwise.

Given a proper closed strictly convex function $\phi: \mathbb{E} \rightarrow (-\infty,\infty]$, finite at $\bm{x}$, $\bm{y}$ and differentiable at $\bm{y}$ but not necessarily at $\bm{x}$, the \textit{Bregman distance} \cite{b1967relaxation} between $\bm{x}$ and $\bm{y}$ associated with the kernel function $\phi$ is defined as
\begin{equation*}
\mathcal{D}_{\phi}(\bm{x}, \,\bm{y}) := \phi(\bm{x}) - \phi(\bm{y}) - \langle \nabla\phi(\bm{y}), \,\bm{x} - \bm{y} \rangle.
\end{equation*}
It is easy to see that $D_{\phi}(\bm{x}, \,\bm{y})\geq0$ and  equality holds if and only if $\bm{x}=\bm{y}$ due to the strict convexity of $\phi$. When $\mathbb{E}:=\mathbb{R}^n$ and $\phi(\cdot):=\|\cdot\|^2$, $\mathcal{D}_{\phi}(\cdot,\cdot)$ readily recovers the classical squared Euclidean distance. Based on the Bregman distance, we then define the \textit{restricted relative smoothness} as follows.

\begin{definition}[{\bf Restricted relative smoothness on $\mathcal{X}$}]\label{defRe2Sm}
Let $f, \,\phi: \mathbb{E}\to(-\infty, \infty]$ be proper closed convex functions with $\mathrm{dom}\,f\supseteq\mathrm{dom}\,\phi$ and $f,\,\phi$ be differentiable on $\mathrm{int}\,\mathrm{dom}\,\phi$. Given a closed convex set $\mathcal{X}\subseteq\mathbb{E}$ with $\mathcal{X}\cap\mathrm{int}\,\mathrm{dom}\,\phi\neq\emptyset$, we say that $f$ is $L$-smooth relative to $\phi$ restricted on $\mathcal{X}$ if there exists $L\geq0$ such that
\begin{equation}\label{relsmoothcond}
f(\bm{y}) \leq f(\bm{x}) + \langle \nabla f(\bm{x}), \,\bm{y}-\bm{x}\rangle + L\mathcal{D}_{\phi}(\bm{y}, \,\bm{x}), \quad \forall\,
\bm{x}\in\mathcal{X}\cap \mathrm{int}\, \mathrm{dom}\,\phi,
~\bm{y}\in\mathcal{X}\cap \mathrm{dom}\,\phi.
\end{equation}
\end{definition}

The above restricted relative smoothness modifies the original relative smoothness (or the Lipschitz-like/convexity condition) introduced in \cite{bbt2017descent,lu2018relatively} by imposing a restricted set $\mathcal{X}$, and it readily reduces to the original notion when $\mathcal{X}=\mathbb{E}$. Such a restriction would help to extend the notion of the relative smoothness to more choices of $(f, \,\phi)$ with a proper $\mathcal{X}$. For example, when $\nabla f$ is $L_f$-Lipschitz continuous on $\mathbb{R}^n$ and $\phi$ is $\mu_{\phi}$-strongly convex on $\mathcal{X}\subseteq\mathbb{R}^n$, then it can be verified that $f$ is $\frac{L_f}{\mu_{\phi}}$-smooth relative to $\phi$ restricted on $\mathcal{X}$, but $f$ may not be $L$-smooth relative to $\phi$ according to the original definition in \cite{bbt2017descent,lu2018relatively} because $\phi$ may not be strongly convex on its domain. For example, the entropy function $\phi(x)=\sum^n_{i=1}x_i(\log x_i - 1)$ is $\frac{1}{\alpha}$-strongly convex on $[0,\,\alpha]^n$ with any $\alpha>0$, but it is not strongly convex on $\mathrm{dom}\,\phi=\mathbb{R}^n_{+}$. Therefore, employing the notion of restricted relative smoothness in Definition \ref{defRe2Sm} could broaden the possible applications of the
BPGM and its inertial variants. More discussions on the relative smoothness can be found in \cite{bbt2017descent,lu2018relatively}.

In order to establish the rigorous analysis under the relatively smooth setting, we make the following blanket technical assumptions.

\begin{assumption}\label{assumA}
Problem \eqref{mainpro} and the kernel function $\phi$ satisfy the following assumptions.
\begin{itemize}

\item[{\bf A1.}] $\phi:\mathbb{E}\to(-\infty, \infty]$ is essentially smooth and strictly convex on $\mathrm{int}\,\mathrm{dom}\,\phi$. Moreover, $\overline{\mathrm{dom}}\,\phi=\mathcal{Q}$, where $\overline{\mathrm{dom}}\,\phi$ denotes the closure of $\mathrm{dom}\,\phi$.

\item[{\bf A2.}] $P:\mathbb{E}\to(-\infty, \infty]$ is a proper closed convex function with $\mathrm{dom}\,P\cap\mathrm{int}\,\mathcal{Q}\neq\emptyset$.

\item[{\bf A3.}] $f:\mathbb{E}\to(-\infty, \infty]$ is a proper closed convex function with $\mathrm{dom}\,f\supseteq\mathrm{dom}\,\phi$ and $f$ is differentiable on $\mathrm{int}\,\mathrm{dom}\,\phi$. Moreover, there exists a closed convex set $\mathcal{X}\supseteq\mathrm{dom}\,P\cap\mathrm{dom}\,\phi$ such that $f$ is $L$-smooth relative to $\phi$ restricted on $\mathcal{X}$.

\item[{\bf A4.}] $F^*:=\inf\left\{F(\bm{x}) : \bm{x}\in\mathcal{Q} \right\}>-\infty$, i.e., problem \eqref{mainpro} is bounded from below.

\item[{\bf A5.}] Each subproblem in the iBPGM and its inertial variant is well-defined in the sense that the subproblem has a unique solution, which lies in  $\mathrm{int}\,\mathrm{dom}\,\phi$.

\end{itemize}
\end{assumption}

Assumption \ref{assumA}1 implies that the kernel function $\phi$ is of Legendre type; see \cite[Page 258]{r1970convex}. Moreover, $\overline{\mathrm{dom}}\,\phi=\mathcal{Q}$ implies $\mathrm{int}\,\mathrm{dom}\,\phi=\mathrm{int}\,\mathcal{Q}$ due to the convexity of $\mathrm{dom}\,\phi$ and \cite[Proposition 3.36(iii)]{bc2011convex}. Then, one can see from Assumptions \ref{assumA}2\&3 that
\begin{equation*}
\mathrm{dom}\,(P+f)\cap\mathrm{int}\,\mathrm{dom}\,\phi
=\mathrm{dom}\,P\cap\mathrm{int}\,\mathrm{dom}\,\phi=\mathrm{dom}\,P\cap\mathrm{int}\,\mathcal{Q} \not=\emptyset.
\end{equation*}
This, together with $\overline{\mathrm{dom}}\,\phi=\mathcal{Q}$ and
\cite[Proposition 11.1(iv)]{bc2011convex} implies that
\begin{equation}\label{infeq}
F^*=\inf\left\{F(\bm{x}) : \bm{x}\in\mathcal{Q} \right\}
=\inf\left\{F(\bm{x}) : \bm{x}\in\mathrm{dom}\,\phi \right\}.
\end{equation}
Assumption \ref{assumA}5 is typically made to ensure the well-posedness of the BPGM and its inertial variant; see, for example, \cite[Section 3]{bbt2017descent} and \cite[Assumption A5]{hrx2018accelerated}. Indeed, the existence and uniqueness of the solution are commonly used for proximal algorithms in the convex setting. Moreover, for an essentially smooth kernel function $\phi$, its gradient $\nabla\phi$ is defined only on $\mathrm{int}\,\mathrm{dom}\,\phi$. Thus, in a Bregman-type method, it is essential to ensure that the solution at the $k$-th iteration remains in $\mathrm{int}\,\mathrm{dom}\,\phi$ to guarantee the well-definedness of the Bregman proximal term in the subsequent iteration. Assumption \ref{assumA}5 can be satisfied when, for example, $\phi$ is strongly convex over the entire space. Other sufficient conditions are provided in \cite[Lemma 2]{bbt2017descent}.

Finally, we give four supporting lemmas that will be used in the subsequent analysis. The identity in first lemma is routine to verify and the proofs of last two lemmas are relegated to Appendix \ref{apd-lemmas}.

\begin{lemma}[{\bf Four points identity}]
Suppose that a proper closed strictly convex function $\phi: \mathbb{E} \rightarrow (-\infty,\infty]$ is finite at $\bm{a},\,\bm{b},\,\bm{c},\,\bm{d}$ and differentiable at $\bm{a},\,\bm{b}$. Then,
\begin{equation}\label{fourId}
\langle \nabla\phi(\bm{a})-\nabla\phi(\bm{b}),\,\bm{c}-\bm{d} \rangle = \mathcal{D}_{\phi}(\bm{c},\,\bm{b}) + \mathcal{D}_{\phi}(\bm{d},\,\bm{a}) - \mathcal{D}_{\phi}(\bm{c},\,\bm{a}) - \mathcal{D}_{\phi}(\bm{d},\,\bm{b}).
\end{equation}
\end{lemma}

\begin{lemma}[{\cite[Section 2.2]{p1987introduction}}]\label{lemseqcon}
Suppose that $\{\alpha_k\}_{k=0}^{\infty}\subseteq\mathbb{R}$ and $\{\beta_k\}_{k=0}^{\infty}\subseteq\mathbb{R}_+$ are two sequences such that $\{\alpha_k\}$ is bounded from below, $\sum_{k=0}^{\infty} \beta_k < \infty$, and $\alpha_{k+1} \leq \alpha_{k} + \beta_k$ holds for all $k$. Then, $\{\alpha_k\}$ is convergent.
\end{lemma}

\begin{lemma}\label{lemseqcon2}
Let $\{\alpha_k\}_{k=0}^{\infty}$ be a nonnegative sequence. If $\sum^{\infty}_{k=0}\alpha_k<\infty$, then $\frac{1}{k}\sum^{k-1}_{i=0}i\alpha_i\to0$.
\end{lemma}

\begin{lemma}\label{supplem}
Suppose that Assumption \ref{assumA} holds and $\mathcal{X}$ is the closed convex set in Assumption \ref{assumA}3. Given any $\bm{y}\in \mathcal{X}\cap
\mathrm{int}\,\mathrm{dom}\,\phi$ and $\lambda \geq 0$, let $(\bm{x}^*,\,\widetilde{\bm{x}}^*)$ be a pair of approximate solutions of
\begin{equation*}
\min\limits_{\bm{x}}\,P(\bm{x}) + \langle \nabla f(\bm{y}), \,\bm{x}-\bm{y}\rangle + \lambda\mathcal{D}_{\phi}(\bm{x},\,\bm{y})
\end{equation*}
such that $\bm{x}^* \in \mathcal{X} \cap\mathrm{int}\,\mathrm{dom}\,\phi$, $\widetilde{\bm{x}}^* \in \mathrm{dom}\,P\cap\mathrm{dom}\,\phi$, and
\begin{equation}\label{inexcond-gen}
\Delta \in \partial_{\delta} P(\widetilde{\bm{x}}^*) + \nabla f(\bm{y}) + \lambda(\nabla \phi(\bm{x}^*)-\nabla \phi(\bm{y})).
\end{equation}
Then, for any $\bm{x}\in\mathrm{dom}\,P\cap\mathrm{dom}\,\phi$, we have
\begin{equation*}
\begin{aligned}
&\quad F(\widetilde{\bm{x}}^*) - F(\bm{x}) \\
&\leq \lambda\mathcal{D}_{\phi}(\bm{x},\,\bm{y})
- \lambda\mathcal{D}_{\phi}(\bm{x},\,\bm{x}^*)
- (\lambda-L)\mathcal{D}_{\phi}(\widetilde{\bm{x}}^*,\,\bm{y})
+ \lambda\mathcal{D}_{\phi}(\widetilde{\bm{x}}^*,\,\bm{x}^*) + \delta
+ |\langle \Delta, \,\widetilde{\bm{x}}^* - \bm{x} \rangle|.
\end{aligned}
\end{equation*}
\end{lemma}

\section{An inexact Bregman proximal gradient method}\label{sec-iBPGM}

In this section, we develop an inexact Bregman proximal gradient method (iBPGM) based on two types of inexact stopping criteria for solving problem \eqref{mainpro} and study its convergence properties. The complete iterative framework is presented as Algorithm \ref{algo-iBPGM}.

\begin{algorithm}[ht]
\caption{An inexact Bregman proximal gradient  method (iBPGM) for problem \eqref{mainpro}}\label{algo-iBPGM}
\textbf{Input:} Follow Assumption \ref{assumA} to choose a kernel function $\phi$ with $L\geq0$ and $\mathcal{X}\supseteq\mathrm{dom}\,P\cap\mathrm{dom}\,\phi$. Choose
$\bm{x}^0\in \mathcal{X}\cap\mathrm{int}\,\mathrm{dom}\,\phi$ and $\widetilde{\bm{x}}^{0}\in\mathrm{dom}\,P\cap\mathrm{dom}\,\phi$. For (AbSC), choose $\lambda \geq L$ and a sequence of nonnegative scalars $\{\varepsilon_k\}_{k=0}^{\infty}$. For (ReSC), choose $\lambda>L$ and $0\leq\sigma<1-\lambda^{-1}L$. Set $k=0$.   \\
\textbf{while} a termination criterion is not met, \textbf{do} \vspace{-2mm}
\begin{itemize}[leftmargin=2cm]
\item[\textbf{Step 1}.] Find a pair $(\bm{x}^{k+1}, \,\widetilde{\bm{x}}^{k+1})$ and an error pair $(\Delta^k, \delta_k)$ by approximately solving
    \begin{equation}\label{subpro-iBPGM}
    \min\limits_{\bm{x}}~P(\bm{x}) + \langle \nabla f(\bm{x}^k), \,\bm{x}-\bm{x}^k\rangle + \lambda\mathcal{D}_{\phi}(\bm{x},\,\bm{x}^k),
    \end{equation}
    such that $\bm{x}^{k+1} \in \,\mathcal{X}\cap\mathrm{int}\,\mathrm{dom}\,\phi$, $\widetilde{\bm{x}}^{k+1} \in \mathrm{dom}\,P\cap\mathrm{dom}\,\phi$, and
    \begin{equation}\label{inexcond-iBPGM}
    \Delta^k \in \partial_{\delta_k} P(\widetilde{\bm{x}}^{k+1}) + \nabla f(\bm{x}^k) + \lambda\big(\nabla \phi(\bm{x}^{k+1})-\nabla \phi(\bm{x}^{k})\big),
    \end{equation}
    satisfying \textit{either} an absolute-type stopping criterion (AbSC) \textit{or} a partial relative-type stopping criterion (ReSC) as follows:
    \begin{equation*}
    \begin{aligned}
    &\textbf{(AbSC)} \qquad \|\Delta^k\| + \lambda^{-1}\delta_k + \mathcal{D}_{\phi}(\widetilde{\bm{x}}^{k+1}, \,\bm{x}^{k+1}) \leq \varepsilon_{k}, \\[2pt]
    &\textbf{(ReSC)} \qquad \|\Delta^k\| = 0 ~~\mbox{and}~~ \lambda^{-1}\delta_k + \mathcal{D}_{\phi}(\widetilde{\bm{x}}^{k+1}, \,\bm{x}^{k+1}) \leq \sigma\mathcal{D}_{\phi}(\widetilde{\bm{x}}^{k+1}, \,\bm{x}^{k}).
    \end{aligned}
    \end{equation*}

\item [\textbf{Step 2}.] Set $k = k+1$ and go to \textbf{Step 1}. \vspace{-1.5mm}
\end{itemize}
\textbf{end while}  \\
\textbf{Output}: $(\bm{x}^{k}, \,\widetilde{\bm{x}}^{k})$ \vspace{0.5mm}
\end{algorithm}

One can see from Algorithm \ref{algo-iBPGM} that, at each iteration, our inexact framework allows one to \textit{approximately} solve the subproblem \eqref{subpro-iBPGM} under condition \eqref{inexcond-iBPGM} satisfying \textit{either} an absolute-type stopping criterion (AbSC) \textit{or} a partial\footnote{Here, ``partial" means that, unlike (AbSC), a non-zero error term $\Delta^k$ is not allowed in (ReSC).} relative-type stopping criterion (ReSC). Since the subproblem \eqref{subpro-iBPGM} has an unique solution $\bm{x}^{k,*}\in\mathrm{dom}\,P\cap\mathrm{int}\,\mathrm{dom}\,\phi$ (by Assumption \ref{assumA}5), then condition \eqref{inexcond-iBPGM} always holds at $\bm{x}^{k+1}=\widetilde{\bm{x}}^{k+1}=\bm{x}^{k,*}$ and hence it is achievable. When setting $\varepsilon_k\equiv0$ in (AbSC) or setting $\sigma=0$ in (ReSC), it means that $\bm{x}^{k+1}$ (equals to $\widetilde{\bm{x}}^{k+1}$) should be an optimal solution of the subproblem, and thus, our iBPGM readily reduces to the exact BPGM
studied in \cite{bbt2017descent,lu2018relatively,t2010approximation,zls2019simple}.
Moreover, when the conditions $\delta_k\equiv0$ and $\bm{x}^{k+1}\equiv\widetilde{\bm{x}}^{k+1}$ are imposed in (AbSC), condition \eqref{inexcond-iBPGM} with (AbSC) is equivalent to the inexact condition considered in \cite[Definition 2.12]{stonyakin2021inexact} for the subproblem \eqref{subpro-iBPGM}, and hence our iBPGM subsumes the inexact BPGM  studied in \cite{stonyakin2021inexact} as a special case.

Both inexact stopping criteria (AbSC) and (ReSC) are of two-point type (meaning that two points $\bm{x}^{k+1}$ and $\widetilde{\bm{x}}^{k+1}$ are used in the criterion), and they are inspired by the inexact stopping criterion proposed in our recent works \cite{clty2020an,yt2020bregman} for developing a new practical inexact Bregman proximal point algorithm. They may look unusual at the first glance, but actually provide a rather broad inexact framework in which the approximate solutions are handled through the error term $\Delta^{k}$ appearing on the left-hand-side of the optimality conditions, $\partial_{\delta_k}P$ (an approximation of $\partial P$), and the deviation $\mathcal{D}_{\phi}(\widetilde{\bm{x}}^{k+1}, \,\bm{x}^{k+1})$. Thus, our inexact framework is adaptable to various types of errors that may incur when solving the subproblem inexactly. In particular, the admissible deviation allows $\partial_{\delta_k}P$ and $\nabla \phi$ to be computed at two slightly \textit{different} points, respectively. Such a simple strategy would help to circumvent the potential feasibility difficulty of requiring $\bm{x}^{k+1}\in\mathrm{dom}\,P\cap\mathrm{dom}\,\nabla\phi$ by the exact BPGM studied in \cite{bbt2017descent,lu2018relatively,t2010approximation,zls2019simple} and by the inexact BPGM studied in \cite{stonyakin2021inexact}, as exemplified and discussed in Remark \ref{rem-feas}. We have also noticed the recent works by Rebegoldi et al. \cite{rbp2018bregman} and Kabbadj \cite{k2020inexact}.
In their inexact Bregman-type algorithms, both approaches allow for an $\varepsilon$-minimizer for the subproblem \eqref{subpro-iBPGM} at each iteration. However, verifying an $\varepsilon$-minimizer may not be straightforward in practice, as the optimal objective value of the subproblem is generally unknown a priori. In addition, all these existing inexact algorithms in \cite{k2020inexact,rbp2018bregman,stonyakin2021inexact} do not accommodate the use of a relative-type stopping criterion, which limits their flexibility in practical applications. Finally, we would like to point out that both the pair $(\bm{x}^{k+1}, \,\widetilde{\bm{x}}^{k+1})$ and the error pair $(\Delta^k, \delta_k)$ may not be unique. Any candidate satisfying condition \eqref{inexcond-iBPGM} under \textit{either} (AbSC) \textit{or} (ReSC) is acceptable in our inexact framework.

One can also observe that the main difference between (AbSC) and (ReSC) is the strategy to control the error incurred in the inexact minimization of the subproblem \eqref{subpro-iBPGM}. In (AbSC), at each iteration, the error $\|\Delta^k\| + \lambda^{-1}\delta_k + \mathcal{D}_{\phi}(\widetilde{\bm{x}}^{k+1}, \,\bm{x}^{k+1})$ is simply controlled by a prespecified tolerance parameter $\varepsilon_k$. This is a natural and common way to design an inexact stopping criterion. To guarantee the progressive accuracy and convergence under this type of criterion, certain summable conditions (see, for example, Theorem \ref{thm-fval-iBPGM-IC1}) are also required on the error-tolerance sequence $\{\varepsilon_k\}_{k=0}^{\infty}$, which may need careful tuning for the iBPGM to achieve good practical efficiency, as observed from Tables \ref{ResTable-nu1} and \ref{ResTable-nu001} in the numerical section. In contrast, in (ReSC), the error $\lambda^{-1}\delta_k + \mathcal{D}_{\phi}(\widetilde{\bm{x}}^{k+1}, \,\bm{x}^{k+1})$ is controlled by the difference $\sigma\mathcal{D}_{\phi}(\widetilde{\bm{x}}^{k+1}, \,\bm{x}^{k})$, where only a \textit{single} tolerance parameter $\sigma$ is involved and thus the corresponding parameter tuning is typically easier. However, the computation of $\mathcal{D}_{\phi}(\widetilde{\bm{x}}^{k+1}, \,\bm{x}^{k})$ may incur non-negligible extra cost, especially when the problem size is large. In our current convergence analysis, we also require $\|\Delta^k\|=0$ for (ReSC) to eliminate the error term $|\langle \Delta^k, \,\widetilde{\bm{x}}^{k+1} - \bm{x} \rangle|$ in \eqref{suffdes-iBPGM}, enabling the derivation of \eqref{sumFbd-tmp-IC23} in the subsequent analysis for Theorem \ref{thm-fval-iBPGM-IC2}. Although this requirement can be satisfied in many cases (see two examples in the following subsection), it may limit the applicability of (ReSC). Note that Bello-Cruz et al. \cite{bgk2023fista} recently incorporated a \textit{full} relative-type error rule in the FISTA, which can accommodate a nonzero error $\Delta^k$. Thus, it would be interesting to explore whether the requirement $\|\Delta^k\|=0$ can be removed in our methods, which we will leave for future research. The preliminary numerical study on the performances of our iBPGM under these two types of stopping criteria can be found in Section \ref{secnum}.

\subsection{Examples of verifying the stopping criteria}\label{sec:examples}

The practical implementations for verifying the stopping criteria (AbSC) and (ReSC) may strongly depend on the specific subproblem and the subsolver used. Below, we provide two examples to illustrate the basic idea of verifying these stopping criteria. The second one is similar to the test problem in our numerical experiments. For simplicity, we assume that the relative smoothness holds without a restricted set $\mathcal{X}$.

\paragraph{Example 1.} Consider that $P(\bm{x}):=h(B\bm{x})$, where $B\in\mathbb{R}^{m\times n}$ and $h:\mathbb{R}^m\to\mathbb{R}$ is a continuous convex function with $\mathrm{dom}\,h=\mathbb{R}^m$. In this case, the subproblem \eqref{subpro-iBPGM} can be equivalently expressed as
\begin{equation}\label{ex1-pro}
\min\limits_{\bm{x}}~H_k(\bm{x}) := h(B\bm{x}) + \langle \nabla f(\bm{x}^k), \,\bm{x}\rangle + \lambda\mathcal{D}_{\phi}(\bm{x},\,\bm{x}^k),
\end{equation}
which can be further equivalently rewritten as
\begin{equation*}
\begin{aligned}
\min\limits_{\bm{x},\,\bm{y}}~\,h(\bm{y}) + \langle \nabla f(\bm{x}^k), \,\bm{x}\rangle + \lambda\mathcal{D}_{\phi}(\bm{x},\,\bm{x}^k), \quad
~\,\mathrm{s.t.}~~ B\bm{x} = \bm{y}.
\end{aligned}
\end{equation*}
Then, the associated Lagrangian function is given by
\begin{equation*}
\mathcal{L}(\bm{x},\bm{y},\bm{z})
:= h(\bm{y}) + \langle \nabla f(\bm{x}^k), \,\bm{x}\rangle
+ \lambda\mathcal{D}_{\phi}(\bm{x},\,\bm{x}^k)
+ \langle \bm{z}, \,B\bm{x}-\bm{y}\rangle,
\end{equation*}
where $\bm{z}\in\mathbb{R}^m$ is the Lagrangian multiplier. Thus, the dual problem of \eqref{ex1-pro} can be derived as
\begin{equation}\label{ex1-dualpro}
\max\limits_{\bm{z}} ~G_k(\bm{z}) := -h^*(\bm{z})
+ \inf\limits_{\bm{x}}\left\{\langle\nabla f(\bm{x}^k) + B^{\top}\bm{z}, \,\bm{x}\rangle + \lambda\mathcal{D}_{\phi}(\bm{x},\,\bm{x}^k)\right\}.
\end{equation}
As discussed in \cite[Section 3.1]{bbt2017descent}, when $\phi$ is of Legendre type (see \cite[Page 258]{r1970convex}) and $\nabla\phi(\bm{x}^k) - \lambda^{-1}\left(\nabla f(\bm{x}^k) + B^{\top}\bm{z}\right)\in\mathrm{dom}\nabla\phi^*$, the above infimum is attainable at $p^k_{\lambda}\left(\nabla f(\bm{x}^k) + B^{\top}\bm{z}\right)$, where
\begin{equation*}
p^k_{\lambda}(\bm{u})
:=\nabla\phi^*\left(\nabla\phi(\bm{x}^k)-\lambda^{-1}\bm{u}\right)
=\arg\min\limits_{\bm{x}}\left\{\langle\bm{u}, \,\bm{x}\rangle + \lambda\mathcal{D}_{\phi}(\bm{x},\,\bm{x}^k)\right\}
\end{equation*}
is computable for many commonly used kernel functions, e.g, the quadratic kernel function $\phi(\bm{x})=\frac{1}{2}\|\bm{x}\|^2$ and the entropy kernel function $\phi(\bm{x})=\sum_{i=1}^nx_{i}(\log x_{i}-1)$. Therefore, when both problems \eqref{ex1-pro} and \eqref{ex1-dualpro} are attainable and have zero duality gap, one could consider solving the subproblem \eqref{ex1-pro} by solving its dual problem \eqref{ex1-dualpro}. Indeed, if we have a maxmizing sequence $\{\bm{z}^{k,t}\}$ for \eqref{ex1-dualpro}, then
\begin{equation*}
\bm{x}^{k,t}
:=p^k_{\lambda}\big(\nabla f(\bm{x}^k) + B^{\top}\bm{z}^{k,t}\big)
=\nabla\phi^*\left(\nabla\phi(\bm{x}^k)-\lambda^{-1}\big(\nabla f(\bm{x}^k) + B^{\top}\bm{z}^{k,t}\big)\right)
\end{equation*}
will converge to the solution of problem \eqref{ex1-pro} and the duality gap $H_k(\bm{x}^{k,t})-G_k(\bm{z}^{k,t})$ will also converge to 0. Moreover, we see that
\begin{equation*}
\begin{aligned}
&\quad H_k(\bm{x}^{k,t})-G_k(\bm{z}^{k,t}) \\
&= h(B\bm{x}^{k,t}) + \langle \nabla f(\bm{x}^k), \,\bm{x}^{k,t}\rangle + \lambda\mathcal{D}_{\phi}(\bm{x}^{k,t},\,\bm{x}^k)
-\left(-h^*(\bm{z}^{k,t})
+ \langle\nabla f(\bm{x}^k) + B^{\top}\bm{z}^{k,t}, \,\bm{x}^{k,t}\rangle + \lambda\mathcal{D}_{\phi}(\bm{x}^{k,t},\,\bm{x}^k)\right) \\
&= h(B\bm{x}^{k,t}) + h^*(\bm{z}^{k,t}) - \langle\bm{z}^{k,t}, \,B\bm{x}^{k,t}\rangle.
\end{aligned}
\end{equation*}
This together with \cite[Proposition 1.2.1]{hl1993convex} implies that  $\bm{z}^{k,t}$ is a $\delta_{k,t}$-subgradient of $h$ at $B\bm{x}^{k,t}$, i.e., $\bm{z}^{k,t}\in\partial_{\delta_{k,t}}h(B\bm{x}^{k,t})$, with  $\delta_{k,t}:=H_k(\bm{x}^{k,t})-G_k(\bm{z}^{k,t})\to0$, and hence
\begin{equation*}
B^{\top}\bm{z}^{k,t}\in\partial_{\delta_{k,t}}P(\bm{x}^{k,t}).
\end{equation*}
On the other hand, from the construction of $\bm{x}^{k,t}$, we have that
\begin{equation*}
-B^{\top}\bm{z}^{k,t}=\nabla f(\bm{x}^k) + \lambda\big(\nabla \phi(\bm{x}^{k,t})-\nabla \phi(\bm{x}^{k})\big).
\end{equation*}
Consequently, combining the above two relations induces that
\begin{equation*}
\Delta^{k,t}:=0 \in \partial_{\delta_{k,t}} P(\widetilde{\bm{x}}^{k,t}) + \nabla f(\bm{x}^k) + \lambda\big(\nabla \phi(\bm{x}^{k,t}) - \nabla \phi(\bm{x}^{k})\big) ~~\mbox{with}~~\widetilde{\bm{x}}^{k,t}=\bm{x}^{k,t},
\end{equation*}
which falls into the form of condition \eqref{inexcond-iBPGM}. Thus, our stopping criteria (AbSC) and (ReSC) are checkable at the pair $(\bm{x}^{k,t},\bm{x}^{k,t})$ associated with the error pair $(0,\delta_{k,t})$, and they reduce to the verification of $\lambda^{-1}\delta_{k,t}\leq\varepsilon_k$ and $\lambda^{-1}\delta_{k,t}\leq\sigma\mathcal{D}_{\phi}(\bm{x}^{k,t}, \,\bm{x}^{k})$, respectively. When $\varepsilon_k\neq0$ and $\bm{x}^k$ is not the solution of the subproblem \eqref{ex1-pro} (hence $\mathcal{D}_{\phi}(\bm{x}^{k,t}, \,\bm{x}^{k})$ cannot approach zero), these two inequalities must hold after finitely
many $t$ iterations.

\paragraph{Example 2.} Consider that $P$ is an indicator function $\iota_{\mathcal{C}}$ on a convex polyhedral set $\mathcal{C}:=\{\bm{x}\in\mathbb{R}^n:A\bm{x}=\bm{b}, \,\bm{x}\geq0\}$ with $A\in\mathbb{R}^{m\times n}$ and $\bm{b}\in\mathbb{R}^m$, and consider the entropy function $\phi(\bm{x})=\sum_{i=1}^nx_{i}(\log x_{i}-1)$ as the kernel function with {\rm dom}\,$\phi  =\mathbb{R}^n_+$. In this case, the subproblem \eqref{subpro-iBPGM} can be equivalently expressed as
\begin{equation}\label{ex2-pro}
\begin{aligned}
\min\limits_{\bm{x}}~\langle \nabla f(\bm{x}^k), \,\bm{x}\rangle
+ \lambda\mathcal{D}_{\phi}(\bm{x},\,\bm{x}^k) \quad \mathrm{s.t.}
\quad A\bm{x}=\bm{b}.
\end{aligned}
\end{equation}
Then, the associated Lagrangian function is given by
\begin{equation*}
\mathcal{L}(\bm{x},\bm{z})
:= \langle \nabla f(\bm{x}^k), \,\bm{x}\rangle
+ \lambda\mathcal{D}_{\phi}(\bm{x},\,\bm{x}^k)
+ \langle \bm{z}, \,A\bm{x}-\bm{b}\rangle,
\end{equation*}
where $\bm{z}\in\mathbb{R}^m$ is the Lagrangian multiplier. Thus, the dual problem of \eqref{ex2-pro} can be derived as
\begin{equation}\label{ex2-dualpro}
\max\limits_{\bm{z}} ~
-\langle \bm{z}, \,\bm{b}\rangle
+ \inf\limits_{\bm{x}}\left\{
\langle\nabla f(\bm{x}^k) + A^{\top}\bm{z}, \,\bm{x}\rangle + \lambda\mathcal{D}_{\phi}(\bm{x},\,\bm{x}^k)\right\}.
\end{equation}
Since the entropy function $\phi$ is of Legendre type and $\mathrm{dom}\nabla\phi^*=\mathbb{R}^n$, the above infimum is always attainable at $\nabla\phi^*\left(\nabla\phi(\bm{x}^k)-\lambda^{-1}\big(\nabla f(\bm{x}^k)
+ A^{\top}\bm{z}\big)\right)$ for any $\bm{z}$. Moreover, when the feasible set $\mathcal{C}$ is bounded and the intersection $\mathcal{C}\cap\mathbb{R}^{n}_{++}$ is nonempty (as is the case in our experiments), using similar arguments for proving \cite[Proposition 2]{clty2020an}, one can show that both problems \eqref{ex2-pro} and \eqref{ex2-dualpro} have solutions, and for any optimal solution $\bm{z}^{k,*}$ of problem \eqref{ex2-dualpro}, $\bm{x}^{k,*}:=\nabla\phi^*\left(\nabla\phi(\bm{x}^k)-\lambda^{-1}\big(\nabla f(\bm{x}^k) + A^{\top}\bm{z}^{k,*}\big)\right)\in\mathrm{int}\,\mathrm{dom}\,\phi$ is the optimal solution of problem \eqref{ex2-pro}. Thus, we can solve the subproblem \eqref{ex2-pro} by solving its dual problem \eqref{ex2-dualpro}. Indeed, if we have a maxmizing sequence $\{\bm{z}^{k,t}\}$ for \eqref{ex2-dualpro}, then
\begin{equation*}
\bm{x}^{k,t}:=\nabla\phi^*\left(\nabla\phi(\bm{x}^k)-\lambda^{-1}\big(\nabla f(\bm{x}^k) + A^{\top}\bm{z}^{k,t}\big)\right) ~~\to~~ \bm{x}^{k,*}.
\end{equation*}
But one should note that $\bm{x}^{k,t}$ may not be exactly feasible for problem \eqref{ex2-pro}, namely, $\bm{x}^{k,t}\notin\mathcal{C}$. Hence, let $\mathcal{G}_{\mathcal{C}}(\cdot)$ be the projection or rounding operator on $\mathcal{C}$ such that $\mathcal{G}_{\mathcal{C}}(\bm{u})\in\mathcal{C}$ for any $\bm{u}\in\mathbb{R}^n$, and set $\widetilde{\bm{x}}^{k,t}:=\mathcal{G}_{\mathcal{C}}(\bm{x}^{k,t})$. Then, for any $\bm{u}\in\mathcal{C}$,
\begin{equation*}
\begin{aligned}
&\quad \langle - \nabla f(\bm{x}^k) - \lambda\big(\nabla \phi(\bm{x}^{k,t}) - \nabla \phi(\bm{x}^{k})\big), \,\bm{u}-\widetilde{\bm{x}}^{k,t}\rangle \\
&= \langle - \nabla f(\bm{x}^k) - A^{\top}\bm{z}^{k,t} - \lambda\big(\nabla \phi(\bm{x}^{k,t}) - \nabla \phi(\bm{x}^{k})\big), \,\bm{u}-\widetilde{\bm{x}}^{k,t}\rangle = 0,
\end{aligned}
\end{equation*}
which the first equality follows from $\langle A^{\top}\bm{z}^{k,t},\,\bm{u}-\widetilde{\bm{x}}^{k,t}\rangle=\langle \bm{z}^{k,t},\,A\bm{u}-A\widetilde{\bm{x}}^{k,t}\rangle=0$ and the second equality follows from the construction of $\bm{x}^{k,t}$. The above relation readily implies that
\begin{equation*}
\Delta^{k,t}:=0 \in \partial_{\delta_{k,t}:=0}\iota_{\mathcal{C}}(\widetilde{\bm{x}}^{k,t}) + \nabla f(\bm{x}^k) + \lambda\big(\nabla \phi(\bm{x}^{k,t}) - \nabla \phi(\bm{x}^{k})\big),
\end{equation*}
which falls into the form of condition \eqref{inexcond-iBPGM}. Thus, our stopping criteria (AbSC) and (ReSC) are checkable at the pair $(\widetilde{\bm{x}}^{k,t},\bm{x}^{k,t})$ associated with the error pair $(0,0)$, and they reduce to the verification of $\mathcal{D}_{\phi}(\widetilde{\bm{x}}^{k,t}, \,\bm{x}^{k,t})\leq\varepsilon_k$ and $\mathcal{D}_{\phi}(\widetilde{\bm{x}}^{k,t}, \,\bm{x}^{k,t})\leq\sigma\mathcal{D}_{\phi}(\bm{x}^{k,t}, \,\bm{x}^{k})$, respectively. Note from $\bm{x}^{k,t}\to\bm{x}^{k,*}\in\mathrm{int}\,\mathrm{dom}\,\phi$ that $\mathcal{D}_{\phi}(\widetilde{\bm{x}}^{k,t}, \,\bm{x}^{k,t})$ goes to zero as long as $\widetilde{\bm{x}}^{k,t}-\bm{x}^{k,t}\to0$. Therefore, when $\widetilde{\bm{x}}^{k,t}-\bm{x}^{k,t}\to0$, $\varepsilon_k\neq0$ and $\bm{x}^k$ is not the solution of the subproblem \eqref{ex2-pro} (hence $\mathcal{D}_{\phi}(\bm{x}^{k,t}, \,\bm{x}^{k})$ cannot approach zero), these two inequalities must hold after finitely many $t$ iterations.

\begin{remark}[\textbf{Comments on the underlying feasibility difficulty}]\label{rem-feas}
We would like to highlight the underlying feasibility difficulty when solving the subproblem in the BPGM, using \textbf{Example 2}. Indeed, when employing the exact BPGM in \cite{bbt2017descent,lu2018relatively,t2010approximation,zls2019simple} or the inexact BPGM in \cite{stonyakin2021inexact}, one needs to find the exact solution or an inexact solution $\bar{\bm{x}}^{k+1}$ that should be strictly contained in $\mathrm{dom}\,P\cap\mathrm{dom}\,\nabla\phi = \{\bm{x}\in\mathbb{R}^n:A\bm{x}=\bm{b}, \,\bm{x}>0\} =: \mathrm{relin}\,\mathcal{C}$ (namely, the relative interior of $\mathcal{C}$). However, for most choices of $(A,\bm{b})$, a subsolver for solving \eqref{ex2-pro} or its dual \eqref{ex2-dualpro} may only return a candidate approximate solution $\bm{x}^{k,t}$ that is located outside $\mathcal{C}$. Hence, one has to further perform a proper projection/rounding step at $\bm{x}^{k,t}$ to compute an approximate solution $\bar{\bm{x}}^{k+1}$ in $\mathrm{relin}\,\mathcal{C}$, and then check the inexact condition at $\bar{\bm{x}}^{k,t}$. Due to the non-closedness of $\mathrm{relin}\,\mathcal{C}$, this step is indeed difficult to implement especially when some entries of the exact solution are very close to zero. In contrast, as shown in \textbf{Example 2}, our two-point-type inexact condition requires separately $\bm{x}^{k,t} \in
\mathcal{X}\cap\mathrm{int}\,\mathrm{dom}\,\phi$ and $\widetilde{\bm{x}}^{k,t} \in \mathrm{dom}\,P\cap\mathrm{dom}\,\phi=\mathcal{C}$, where $\mathcal{X}\supseteq\mathrm{dom}\,P\cap\mathrm{dom}\,\phi$ can be any closed convex set satisfying Assumption \ref{assumA}3. The point $\bm{x}^{k,t}$ (which is generally located outside $\mathcal{C}$) can be easily obtained from the subsolver and the point $\widetilde{\bm{x}}^{k,t}$ can be computed by performing a proper projection/rounding step at $\bm{x}^{k,t}$ over $\mathcal{C}$, which is clearly easier and more practical than computing a point in $\mathrm{relin}\,\mathcal{C}$. In this regard, our iBPGM is more favorable than the exact BPGM in \cite{bbt2017descent,lu2018relatively,t2010approximation,zls2019simple} and the inexact BPGM  in \cite{stonyakin2021inexact}. More discussions on the potential advantages of such a two-point-type inexact condition over existing ones can be found in \cite[Section 4]{yt2020bregman}.
\end{remark}

\subsection{Convergence analysis for the iBPGM}

We next establish the convergence of our iBPGM under the (AbSC) or (ReSC) when solving the subproblems. The analysis is inspired by \cite{bbt2017descent,lu2018relatively}, but is more involved due to the approximate minimization of the subproblem under (AbSC) or (ReSC). We first give the following sufficient-descent-like property.

\begin{lemma}\label{lem-suffdes-iBPGM}
Suppose that Assumption \ref{assumA} holds. Let $\{\bm{x}^k\}$ and $\{\widetilde{\bm{x}}^k\}$ be the sequences generated by the iBPGM in Algorithm \ref{algo-iBPGM}. Then, for any $k\geq0$ and any $\bm{x}\in\mathrm{dom}\,P\cap\mathrm{dom}\,\phi$,
\begin{equation}\label{suffdes-iBPGM}
\begin{aligned}
F(\widetilde{\bm{x}}^{k+1}) - F(\bm{x})
&\leq \lambda\mathcal{D}_{\phi}(\bm{x},\,\bm{x}^{k})
- \lambda\mathcal{D}_{\phi}(\bm{x},\,\bm{x}^{k+1})
- (\lambda-L)\mathcal{D}_{\phi}(\widetilde{\bm{x}}^{k+1},\,\bm{x}^{k}) \\
&\qquad
+ \lambda\big(\mathcal{D}_{\phi}(\widetilde{\bm{x}}^{k+1},\,\bm{x}^{k+1})
+ \lambda^{-1}\delta_k\big)
+ |\langle \Delta^k, \,\widetilde{\bm{x}}^{k+1} - \bm{x} \rangle|.
\end{aligned}
\end{equation}
\end{lemma}
\begin{proof}
The desired result can be easily obtained by applying Lemma \ref{supplem} with $\widetilde{\bm{x}}^*=\widetilde{\bm{x}}^{k+1}$, $\bm{x}^*=\bm{x}^{k+1}$ and $\bm{y}=\bm{x}^k$.
\end{proof}

We then have the following results concerning the convergence of the function value.

\begin{theorem}[\textbf{Convergence rate of the iBPGM with (AbSC)}]\label{thm-fval-iBPGM-IC1}
Suppose that Assumption \ref{assumA} holds and $\lambda \geq L$. Let $\{\bm{x}^k\}$ and $\{\widetilde{\bm{x}}^k\}$ be the sequences generated by the iBPGM with (AbSC) in Algorithm \ref{algo-iBPGM}, and let $e(\bm{x}):=F(\bm{x})-F^*$, where $F^*:=\inf\left\{F(\bm{x}) : \bm{x}\in\mathcal{Q} \right\}>-\infty$. The following statements hold.
\begin{itemize}
\item[{\rm (i)}] (\textbf{Summability}) For any $\bm{x}\in\mathrm{dom}\,P\cap\mathrm{dom}\,\phi$, we have that
    \begin{equation*}
    \begin{aligned}
    (\lambda-L){\textstyle\sum^k_{i=0}}\mathcal{D}_{\phi}(\widetilde{\bm{x}}^{i+1},\,\bm{x}^{i})
    \leq
    \lambda\mathcal{D}_{\phi}(\bm{x},\,\bm{x}^{0})
    + (k+1)e(\bm{x})
    + (\|\bm{x}\|+\lambda){\textstyle\sum^k_{i=0}}\varepsilon_{i}
    + {\textstyle\sum^k_{i=0}}|\langle \Delta^i, \,\widetilde{\bm{x}}^{i+1} \rangle|.
    \end{aligned}
    \end{equation*}
    Moreover, if problem \eqref{mainpro} has an optimal solution $\bm{x}^*\in\mathrm{dom}\,\phi$, $\lambda>L$, $\sum\varepsilon_{k}<\infty$ and $\sum|\langle \Delta^k, \,\widetilde{\bm{x}}^{k+1} \rangle|<\infty$, then $\sum^{\infty}_{k=0}\mathcal{D}_{\phi}(\widetilde{\bm{x}}^{k+1},\,\bm{x}^{k})<\infty$.

\item[{\rm (ii)}] (\textbf{At the averaged iterate}) For any $k\geq0$ and any $\bm{x}\in\mathrm{dom}\,P\cap\mathrm{dom}\,\phi$, we have that
                 \begin{equation}\label{compavgk-iBPGM}
                 \begin{aligned}
                 &\quad F\left(\frac{1}{k+1}\sum^k_{i=0}\widetilde{\bm{x}}^{i+1}\right) - F(\bm{x})  \\
                 &\leq \frac{1}{k+1}\left( \lambda\mathcal{D}_{\phi}(\bm{x},\,\bm{x}^{0})
                 + (\|\bm{x}\|+\lambda)\sum^k_{i=0}\varepsilon_{i}
                 + \sum^k_{i=0}|\langle \Delta^i, \,\widetilde{\bm{x}}^{i+1} \rangle|\right).
                 \end{aligned}
                 \end{equation}
                 Moreover, if $\frac{1}{k}\sum^{k-1}_{i=0}\varepsilon_i\to0$ and $\frac{1}{k}\sum^{k-1}_{i=0}|\langle \Delta^i, \,\widetilde{\bm{x}}^{i+1} \rangle|\to0$, we have that $F\big({\textstyle\frac{1}{k}\sum^{k-1}_{i=0}}\,\widetilde{\bm{x}}^{i+1}\big) \to F^*$. In addition, if problem \eqref{mainpro} has an optimal solution $\bm{x}^*\in\mathrm{dom}\,\phi$, $\sum\varepsilon_k<\infty$ and $\sum|\langle \Delta^k, \,\widetilde{\bm{x}}^{k+1} \rangle|<\infty$, we have that
                 \begin{equation*}
                 F\left(\frac{1}{k}\sum^{k-1}_{i=0}\,\widetilde{\bm{x}}^{i+1}\right) - F^* \leq \mathcal{O}\left(\frac{1}{k}\right).
                 \end{equation*}

\item[{\rm (iii)}] (\textbf{At the last iterate}) For any $k\geq0$ and any $\bm{x}\in\mathrm{dom}\,P\cap\mathrm{dom}\,\phi$, we have that
                 \begin{equation}\label{compk-iBPGM}
                 \begin{aligned}
                 &\quad F(\widetilde{\bm{x}}^{k+1}) - F(\bm{x})  \\
                 &\leq \frac{1}{k+1}\left( \lambda\mathcal{D}_{\phi}(\bm{x},\,\bm{x}^{0})
                 + (\|\bm{x}\|+2\lambda)\sum^k_{i=0}\big(\varepsilon_i + i\varepsilon_{i}\big)
                 + \sum^k_{i=0}\big(|\langle \Delta^i, \,\widetilde{\bm{x}}^{i+1} \rangle| + i|\langle \Delta^i, \,\widetilde{\bm{x}}^{i+1} - \widetilde{\bm{x}}^i \rangle|\big) \right).
                 \end{aligned}
                 \end{equation}
                 Moreover, if $\frac{1}{k}\sum^{k-1}_{i=0}i\varepsilon_i\to0$ and $\frac{1}{k}\sum^{k-1}_{i=0}\big(|\langle \Delta^i, \,\widetilde{\bm{x}}^{i+1} \rangle| + i|\langle \Delta^i, \,\widetilde{\bm{x}}^{i+1} - \widetilde{\bm{x}}^i \rangle|\big)\to0$, we have that $F(\widetilde{\bm{x}}^k) \to F^*$. In addition, if problem \eqref{mainpro} has an optimal solution $\bm{x}^*\in\mathrm{dom}\,\phi$, $\sum k\varepsilon_k<\infty$ and $\sum\big(|\langle \Delta^k, \,\widetilde{\bm{x}}^{k+1} \rangle| + k|\langle \Delta^k, \,\widetilde{\bm{x}}^{k+1} - \widetilde{\bm{x}}^k \rangle|\big)<\infty$, we have that
                 \begin{equation*}
                 F(\widetilde{\bm{x}}^k) - F^* \leq \mathcal{O}\left(\frac{1}{k}\right).
                 \end{equation*}

\end{itemize}
\end{theorem}
\begin{proof}
\textit{Statement (i)}.
First, it follows from \eqref{suffdes-iBPGM} and (AbSC) that, for any $i\geq0$ and any $\bm{x}\in\mathrm{dom}\,P\cap\mathrm{dom}\,\phi$,
\begin{equation}\label{recuineqistar}
\begin{aligned}
&\quad F(\widetilde{\bm{x}}^{i+1}) - F(\bm{x})  \;\leq \;
 \lambda\mathcal{D}_{\phi}(\bm{x},\,\bm{x}^{i})
- \lambda\mathcal{D}_{\phi}(\bm{x},\,\bm{x}^{i+1})
- (\lambda-L)\mathcal{D}_{\phi}(\widetilde{\bm{x}}^{i+1},\,\bm{x}^{i})  \\
&\hspace{37mm} + |\langle \Delta^i, \,\widetilde{\bm{x}}^{i+1} - \bm{x} \rangle|
+ \lambda\big(\mathcal{D}_{\phi}(\widetilde{\bm{x}}^{i+1},\,\bm{x}^{i+1}) + \lambda^{-1}\delta_i\big) \\
&\leq \lambda\mathcal{D}_{\phi}(\bm{x},\,\bm{x}^{i})
- \lambda\mathcal{D}_{\phi}(\bm{x},\,\bm{x}^{i+1})
- (\lambda-L)\mathcal{D}_{\phi}(\widetilde{\bm{x}}^{i+1},\,\bm{x}^{i})
+ (\|\bm{x}\|+\lambda)\varepsilon_{i}
+ |\langle \Delta^i, \,\widetilde{\bm{x}}^{i+1} \rangle|,
\end{aligned}
\end{equation}
which, together with $F(\widetilde{\bm{x}}^{i+1}) \geq F^*$ for all $i\geq0$, implies that
\begin{equation*}
(\lambda-L)\mathcal{D}_{\phi}(\widetilde{\bm{x}}^{i+1},\,\bm{x}^{i})
\leq \lambda\mathcal{D}_{\phi}(\bm{x},\,\bm{x}^{i})
- \lambda\mathcal{D}_{\phi}(\bm{x},\,\bm{x}^{i+1})
+ e(\bm{x})
+ (\|\bm{x}\|+\lambda)\varepsilon_{i}
+ |\langle \Delta^i, \,\widetilde{\bm{x}}^{i+1} \rangle|.
\end{equation*}
Summing this inequality from $i=0$ to $i=k$ yields
\begin{equation*}
\begin{aligned}
&\quad (\lambda-L){\textstyle\sum^k_{i=0}}\mathcal{D}_{\phi}(\widetilde{\bm{x}}^{i+1},\,\bm{x}^{i}) \\
&\leq
\lambda\mathcal{D}_{\phi}(\bm{x},\,\bm{x}^{0})
- \lambda\mathcal{D}_{\phi}(\bm{x},\,\bm{x}^{k+1})
+ (k+1)e(\bm{x})
+ (\|\bm{x}\|+\lambda){\textstyle\sum^k_{i=0}}\varepsilon_{i}
+ {\textstyle\sum^k_{i=0}}|\langle \Delta^i, \,\widetilde{\bm{x}}^{i+1} \rangle| \\
&\leq \lambda\mathcal{D}_{\phi}(\bm{x},\,\bm{x}^{0})
+ (k+1)e(\bm{x})
+ (\|\bm{x}\|+\lambda){\textstyle\sum^k_{i=0}}\varepsilon_{i}
+ {\textstyle\sum^k_{i=0}}|\langle \Delta^i, \,\widetilde{\bm{x}}^{i+1} \rangle|.
\end{aligned}
\end{equation*}
Moreover, if problem \eqref{mainpro} has an optimal solution $\bm{x}^*\in\mathrm{dom}\,\phi$, we can substitute $\bm{x}^*$ in the above relation and obtain from $e(\bm{x}^*)=0$ that
\begin{equation*}
(\lambda-L){\textstyle\sum^k_{i=0}}\mathcal{D}_{\phi}(\widetilde{\bm{x}}^{i+1},\,\bm{x}^{i})
\leq \lambda\mathcal{D}_{\phi}(\bm{x}^*,\,\bm{x}^{0})
+ (\|\bm{x}^*\|+\lambda){\textstyle\sum^k_{i=0}}\varepsilon_{i}
+ {\textstyle\sum^k_{i=0}}|\langle \Delta^i, \,\widetilde{\bm{x}}^{i+1} \rangle|.
\end{equation*}
Thus, if, in addition, $\lambda>L$, $\sum\varepsilon_{k}<\infty$ and $\sum|\langle \Delta^k, \,\widetilde{\bm{x}}^{k+1} \rangle|<\infty$, one can conclude that $\sum^{\infty}_{k=0}\mathcal{D}_{\phi}(\widetilde{\bm{x}}^{k+1},\,\bm{x}^{k})<\infty$.

\textit{Statement (ii)}.
From \eqref{recuineqistar} and $\lambda \geq L$, we see that, for any $i\geq0$,
\begin{equation*}
\begin{aligned}
F(\widetilde{\bm{x}}^{i+1}) - F(\bm{x})
\leq \lambda\mathcal{D}_{\phi}(\bm{x},\,\bm{x}^{i})
- \lambda\mathcal{D}_{\phi}(\bm{x},\,\bm{x}^{i+1})
+ (\|\bm{x}\|+\lambda)\varepsilon_{i}
+ |\langle \Delta^i, \,\widetilde{\bm{x}}^{i+1} \rangle|.
\end{aligned}
\end{equation*}
Summing this inequality from $i=0$ to $i=k$ gives
\begin{equation}\label{sumFbd}
\begin{aligned}
&\quad {\textstyle\sum^k_{i=0}}F(\widetilde{\bm{x}}^{i+1}) - (k+1)F(\bm{x})  \\
&\leq \lambda\mathcal{D}_{\phi}(\bm{x},\,\bm{x}^{0})
- \lambda\mathcal{D}_{\phi}(\bm{x},\,\bm{x}^{k+1})
+ (\|\bm{x}\|+\lambda){\textstyle\sum^k_{i=0}}\varepsilon_{i}
+ {\textstyle\sum^k_{i=0}}|\langle \Delta^i, \,\widetilde{\bm{x}}^{i+1} \rangle|.
\end{aligned}
\end{equation}
Using this inequality and $F\big(\frac{1}{k+1}\sum^k_{i=0}\widetilde{\bm{x}}^{i+1}\big)\leq\frac{1}{k+1}\sum^k_{i=0}F(\widetilde{\bm{x}}^{i+1})$ (by the convexity of $F$), we obtain \eqref{compavgk-iBPGM}.

Moreover, if $\frac{1}{k}\sum^{k-1}_{i=0}\varepsilon_i\to0$ and $\frac{1}{k}\sum^{k-1}_{i=0}|\langle \Delta^i, \,\widetilde{\bm{x}}^{i+1} \rangle|\to0$, one can see from \eqref{compavgk-iBPGM} that
\begin{equation*}
\limsup\limits_{k\to\infty}\,
F\big({\textstyle\frac{1}{k}\sum^{k-1}_{i=0}}\,\widetilde{\bm{x}}^{i+1}\big) \leq F(\bm{x}), \quad \forall\,\bm{x}\in\mathrm{dom}\,P\cap\mathrm{dom}\,\phi.
\end{equation*}
This together with $\frac{1}{k}\sum^{k-1}_{i=0}\widetilde{\bm{x}}^{i+1}\in\mathrm{dom}\,P\cap\mathrm{dom}\,\phi\subseteq
\mathrm{dom}\,P\cap\mathrm{dom}\,f\cap Q$ and \eqref{infeq} implies that
\begin{equation*}
F^*
\leq\liminf\limits_{k\to\infty}\,
F\big({\textstyle\frac{1}{k}\sum^{k-1}_{i=0}}\,\widetilde{\bm{x}}^{i+1}\big)  \leq\limsup\limits_{k\to\infty}\,
F\big({\textstyle\frac{1}{k}\sum^{k-1}_{i=0}}\,\widetilde{\bm{x}}^{i+1}\big)
\leq F^*,
\end{equation*}
from which, we can conclude that $F\big({\textstyle\frac{1}{k}\sum^{k-1}_{i=0}}\widetilde{\bm{x}}^{i+1}\big) \to F^*$.

In addition, if problem \eqref{mainpro} has an optimal solution $\bm{x}^*\in\mathrm{dom}\,\phi$, $\sum\varepsilon_k<\infty$ and $\sum|\langle \Delta^k, \,\widetilde{\bm{x}}^{k+1} \rangle|<\infty$, one can see from \eqref{compavgk-iBPGM} with $\bm{x}^*$ in place of $\bm{x}$ that $F\big(\frac{1}{k}\sum^{k-1}_{i=0}\widetilde{\bm{x}}^{i+1}\big) - F^* \leq \mathcal{O}(k^{-1})$.

\textit{Statement (iii)}.
For any $i\geq0$, by setting $\bm{x}=\widetilde{\bm{x}}^i$ in \eqref{suffdes-iBPGM}, we obtain
\begin{equation*}
\begin{aligned}
F(\widetilde{\bm{x}}^{i+1}) - F(\widetilde{\bm{x}}^i)
&\leq \lambda\mathcal{D}_{\phi}(\widetilde{\bm{x}}^i,\,\bm{x}^{i})
- \lambda\mathcal{D}_{\phi}(\widetilde{\bm{x}}^i,\,\bm{x}^{i+1})
- (\lambda-L)\mathcal{D}_{\phi}(\widetilde{\bm{x}}^{i+1},\,\bm{x}^{i})  \\
&\qquad + \lambda\big(\mathcal{D}_{\phi}(\widetilde{\bm{x}}^{i+1},\,\bm{x}^{i+1}) + \lambda^{-1}\delta_i\big)
+ |\langle \Delta^i, \,\widetilde{\bm{x}}^{i+1} - \widetilde{\bm{x}}^i \rangle| \\
&\leq \lambda(\varepsilon_{i}+\varepsilon_{i-1})
+ |\langle \Delta^i, \,\widetilde{\bm{x}}^{i+1} - \widetilde{\bm{x}}^i \rangle|,
\end{aligned}
\end{equation*}
where the last inequality follows from $\lambda \geq L$ and (AbSC). Then, we see that, for any $i\geq0$,
\begin{equation*}
\begin{aligned}
iF(\widetilde{\bm{x}}^{i+1})
&\leq iF(\widetilde{\bm{x}}^{i})
+ \lambda i(\varepsilon_{i}+\varepsilon_{i-1})
+ i|\langle \Delta^i, \,\widetilde{\bm{x}}^{i+1} - \widetilde{\bm{x}}^i \rangle| \\
~~\Longleftrightarrow~~
F(\widetilde{\bm{x}}^{i+1})
&\geq (i+1)F(\widetilde{\bm{x}}^{i+1}) - iF(\widetilde{\bm{x}}^{i})
- \lambda i(\varepsilon_{i}+\varepsilon_{i-1})
- i|\langle \Delta^i, \,\widetilde{\bm{x}}^{i+1} - \widetilde{\bm{x}}^i \rangle|,
\end{aligned}
\end{equation*}
which can induce that
\begin{equation*}
{\textstyle\sum^{k}_{i=0}}F(\widetilde{\bm{x}}^{i+1})
\geq (k+1)F(\widetilde{\bm{x}}^{k+1})
- \lambda{\textstyle\sum^{k}_{i=0}}\,i(\varepsilon_{i}+\varepsilon_{i-1})
- {\textstyle\sum^{k}_{i=0}}\,i|\langle \Delta^i, \,\widetilde{\bm{x}}^{i+1} - \widetilde{\bm{x}}^i \rangle|.
\end{equation*}
This together with \eqref{sumFbd} implies that
\begin{equation*}
\begin{aligned}
&\quad (k+1)\big(F(\widetilde{\bm{x}}^{k+1})-F(\bm{x})\big)  \\
&\leq {\textstyle\sum^{k}_{i=0}}F(\widetilde{\bm{x}}^{i+1})
- (k+1)F(\bm{x}) + \lambda{\textstyle\sum^{k}_{i=0}}\,i(\varepsilon_{i}+\varepsilon_{i-1})
+ {\textstyle\sum^{k}_{i=0}}\,i|\langle \Delta^i, \,\widetilde{\bm{x}}^{i+1} - \widetilde{\bm{x}}^i \rangle|   \\
&\leq \lambda\mathcal{D}_{\phi}(\bm{x},\,\bm{x}^{0})
- \lambda\mathcal{D}_{\phi}(\bm{x},\,\bm{x}^{k+1})
+ (\|\bm{x}\|+\lambda){\textstyle\sum^k_{i=0}}\,\varepsilon_i
+ \lambda{\textstyle\sum^{k}_{i=0}}\,i(\varepsilon_{i}+\varepsilon_{i-1}) \\
&\qquad
+ {\textstyle\sum^k_{i=0}}|\langle \Delta^i, \,\widetilde{\bm{x}}^{i+1} \rangle|
+ {\textstyle\sum^{k}_{i=0}}\,i|\langle \Delta^i, \,\widetilde{\bm{x}}^{i+1} - \widetilde{\bm{x}}^i \rangle| \\
&\leq \lambda\mathcal{D}_{\phi}(\bm{x},\,\bm{x}^{0})
- \lambda\mathcal{D}_{\phi}(\bm{x},\,\bm{x}^{k+1})
+ (\|\bm{x}\|+2\lambda){\textstyle\sum^k_{i=0}}\,\big(\varepsilon_i + i\varepsilon_{i}\big)  \\
&\qquad
+ {\textstyle\sum^k_{i=0}}\,\big(|\langle \Delta^i, \,\widetilde{\bm{x}}^{i+1} \rangle|
+ i|\langle \Delta^i, \,\widetilde{\bm{x}}^{i+1} - \widetilde{\bm{x}}^i \rangle|\big).
\end{aligned}
\end{equation*}
Dividing the above inequality by $k+1$, we can get \eqref{compk-iBPGM}.

Moreover, if $\frac{1}{k}\sum^{k-1}_{i=0}i\varepsilon_i\to0$ (which also yields $\frac{1}{k}\sum^{k-1}_{i=0}\varepsilon_i\to0$) and $\frac{1}{k}\sum^{k-1}_{i=0}\big(|\langle \Delta^i, \,\widetilde{\bm{x}}^{i+1} \rangle|
+ i|\langle \Delta^i, \,\widetilde{\bm{x}}^{i+1} - \widetilde{\bm{x}}^i \rangle|\big)\to0$, it is easy to see from \eqref{compk-iBPGM} that
\begin{equation*}
\limsup\limits_{k\to\infty}\,F(\widetilde{\bm{x}}^{k}) \leq F(\bm{x}), \quad \forall\,\bm{x}\in\mathrm{dom}\,P\cap\mathrm{dom}\,\phi.
\end{equation*}
This, together with $\widetilde{\bm{x}}^{k}\in\mathrm{dom}\,P\cap\mathrm{dom}\,\phi\subseteq
\mathrm{dom}\,P\cap\mathrm{dom}\,f\cap Q$ and \eqref{infeq}, implies that
\begin{equation*}
F^*
\leq\liminf\limits_{k\to\infty}\,F(\widetilde{\bm{x}}^{k})
\leq\limsup\limits_{k\to\infty}\,F(\widetilde{\bm{x}}^{k})
\leq F^*,
\end{equation*}
from which we can conclude that $F(\widetilde{\bm{x}}^{k}) \to F^*$.

In addition, $\sum k\varepsilon_k<\infty$ yields
$\sum\varepsilon_k<\infty$. Thus, if problem \eqref{mainpro} has an optimal solution $\bm{x}^*\in\mathrm{dom}\,\phi$, $\sum k\varepsilon_k<\infty$ and $\sum\big(|\langle \Delta^k, \,\widetilde{\bm{x}}^{k+1} \rangle|
+ k|\langle \Delta^k, \,\widetilde{\bm{x}}^{k+1} - \widetilde{\bm{x}}^k \rangle|\big)<\infty$, we can obtain from \eqref{compk-iBPGM} with $\bm{x}^*$ in place of $\bm{x}$ that $F(\widetilde{\bm{x}}^k) - F^* \leq \mathcal{O}(k^{-1})$. This completes the proof.
\end{proof}

\begin{remark}[\textbf{Comments on convergence rate of the iBPGM with (AbSC)}]\label{rek-comp-iBPGM}
We see from Theorem \ref{thm-fval-iBPGM-IC1} that, under proper summable-error conditions, the sequence of the function value achieves the $\mathcal{O}(1/k)$ convergence rate at both the averaged iterate ${\textstyle\frac{1}{k}\sum^{k-1}_{i=0}}\,\widetilde{\bm{x}}^{i+1}$ and the last iterate $\widetilde{\bm{x}}^k$. The latter result requires (unsurprisingly) stronger summability conditions and covers the related convergence results in \cite{bbt2017descent,lu2018relatively,zls2019simple} wherein each subproblem is solved exactly. Next we comment on how the condition $\sum|\langle \Delta^k, \,\widetilde{\bm{x}}^{k+1} \rangle|<\infty$ (or $\sum\big(|\langle \Delta^k, \,\widetilde{\bm{x}}^{k+1} \rangle| + k|\langle \Delta^k, \,\widetilde{\bm{x}}^{k+1} - \widetilde{\bm{x}}^k \rangle|\big)<\infty$) can be ensured. First, as is the case in our experiments, if no error occurs on the left-hand-side of the optimality condition \eqref{inexcond-iBPGM} (namely, $\Delta^k=0$)
when applying a certain subsolver for solving the subproblem, such summable conditions can be readily satisfied. Moreover, if one knows a priori that $\{\widetilde{\bm{x}}^{k}\}$ will be bounded (for example, when $\mathrm{dom}\,P\cap\mathrm{dom}\,\phi$ is bounded as the test problem in our numerical experiments), then it readily reduces to $\sum\|\Delta^k\|<\infty$ (or $\sum k\|\Delta^k\|<\infty$), which can be guaranteed by $\sum \varepsilon_k<\infty$ (or $\sum k\varepsilon_k<\infty$). In addition, one could also set an arbitrarily
summable nonnegative sequence $\{\eta_k\}$, and during the iterations, check that $|\langle \Delta^k, \,\widetilde{\bm{x}}^{k+1} \rangle|\leq\eta_k$ (or $|\langle \Delta^k, \,\widetilde{\bm{x}}^{k+1} \rangle| + k|\langle \Delta^k, \,\widetilde{\bm{x}}^{k+1} - \widetilde{\bm{x}}^k \rangle|\leq\eta_k$).
This then ensures that $\sum|\langle \Delta^k, \,\widetilde{\bm{x}}^{k+1} \rangle|<\infty$ (or $\sum\big(|\langle \Delta^k, \,\widetilde{\bm{x}}^{k+1} \rangle| + k|\langle \Delta^k, \,\widetilde{\bm{x}}^{k+1} - \widetilde{\bm{x}}^k \rangle|\big)<\infty$). The last summability requirement also indicates that if $\{\widetilde{\bm{x}}^{k}\}$ appears to be unbounded, one may need to drive $\Delta^{k}$ to zero more quickly.
\end{remark}

We next study the convergence rate of the iBPGM with (ReSC).

\begin{theorem}[\textbf{Convergence rate of the iBPGM with (ReSC)}]\label{thm-fval-iBPGM-IC2}
Suppose that Assumption \ref{assumA} holds, $\lambda>L$, and $0<\sigma<\frac{\lambda-L}{\lambda}$. Let $\{\bm{x}^k\}$ and $\{\widetilde{\bm{x}}^k\}$ be the sequences generated by the iBPGM with (ReSC) in Algorithm \ref{algo-iBPGM}, and let $e(\bm{x}):=F(\bm{x})-F^*$, where $F^*:=\inf\left\{F(\bm{x}): \bm{x}\in\mathcal{Q} \right\}>-\infty$. Then, the following statements hold.
\begin{itemize}
\item[{\rm (i)}] (\textbf{Summability}) For any $\bm{x}\in\mathrm{dom}\,P\cap\mathrm{dom}\,\phi$, we have that
    \begin{equation*}
    {\textstyle\sum^k_{i=0}}\mathcal{D}_{\phi}(\widetilde{\bm{x}}^{i+1},\,\bm{x}^{i})
    \leq (\lambda-L-\lambda\sigma)^{-1}\big(\lambda\mathcal{D}_{\phi}(\bm{x},\,\bm{x}^{0})
    + (k+1)e(\bm{x})\big).
    \end{equation*}
    Moreover, if problem \eqref{mainpro} has an optimal solution $\bm{x}^*\in\mathrm{dom}\,\phi$, then $\sum^{\infty}_{k=0}\mathcal{D}_{\phi}(\widetilde{\bm{x}}^{k+1},\,\bm{x}^{k})<\infty$ and  $\sum^{\infty}_{k=0}\mathcal{D}_{\phi}(\widetilde{\bm{x}}^{k+1},\,\bm{x}^{k+1})<\infty$.

\item[{\rm (ii)}] (\textbf{At the averaged iterate}) For any $k\geq0$ and any $\bm{x}\in\mathrm{dom}\,P\cap\mathrm{dom}\,\phi$, we have that
                 \begin{equation}\label{compavgk-iBPGM-IC23}
                 \begin{aligned}
                 F\left(\frac{1}{k+1}\sum^k_{i=0}\,\widetilde{\bm{x}}^{i+1}\right) - F(\bm{x})
                 \leq \frac{\lambda}{k+1}\mathcal{D}_{\phi}(\bm{x},\,\bm{x}^{0})
                 \end{aligned}
                 \end{equation}
                 and $F\big({\textstyle\frac{1}{k}\sum^{k-1}_{i=0}}\,\widetilde{\bm{x}}^{i+1}\big) \to F^*$. Moreover, if problem \eqref{mainpro} has an optimal solution $\bm{x}^*\in\mathrm{dom}\,\phi$, we have that
                 \begin{equation*}
                 F\left(\frac{1}{k}\sum^{k-1}_{i=0}\,\widetilde{\bm{x}}^{i+1}\right) - F^* \leq \mathcal{O}\left(\frac{1}{k}\right).
                 \end{equation*}

\item[{\rm (iii)}] (\textbf{At the last iterate}) For any $k\geq0$ and any $\bm{x}\in\mathrm{dom}\,P\cap\mathrm{dom}\,\phi$, we have that
                 \begin{equation}\label{compk-iBPGM-IC2}
                 \begin{aligned}
                 F(\widetilde{\bm{x}}^{k+1}) - F(\bm{x})
                 \leq \frac{1}{k+1}\left(
                 a\mathcal{D}_{\phi}(\bm{x},\bm{x}^{0})
                 + b(k+1)e(\bm{x})
                 + c\sum^{k}_{i=0}i\mathcal{D}_{\phi}(\widetilde{\bm{x}}^{i+1},\bm{x}^{i+1})\right).
                 \end{aligned}
                 \end{equation}
                 where $a:=(1+b)\lambda$, $b:=(\lambda-L-\lambda\sigma)^{-1}\sigma\lambda$, and $c:=2\lambda-(\lambda-L)\sigma^{-1}$. Moreover, if problem \eqref{mainpro} has an optimal solution $\bm{x}^*\in\mathrm{dom}\,\phi$; or $0<\sigma\leq\frac{\lambda-L}{2\lambda}$; or $\frac{1}{k}\sum^{k-1}_{i=0}i\mathcal{D}_{\phi}(\widetilde{\bm{x}}^{i+1},\bm{x}^{i+1})\to0$,        we have that $F(\widetilde{\bm{x}}^k) \to F^*$. In addition, if problem \eqref{mainpro} has an optimal solution $\bm{x}^*\in\mathrm{dom}\,\phi$ together with $0<\sigma\leq\frac{\lambda-L}{2\lambda}$ or $\sum^{\infty}_{k=0}k\mathcal{D}_{\phi}(\widetilde{\bm{x}}^{k+1},\bm{x}^{k+1})<\infty$, we have that
                 \begin{equation*}
                 F(\widetilde{\bm{x}}^k) - F^* \leq \mathcal{O}\left(\frac{1}{k}\right).
                 \end{equation*}

\end{itemize}
\end{theorem}
\begin{proof}
\textit{Statement (i)}.
First, it follows from \eqref{suffdes-iBPGM} and (ReSC) that, for any $i\geq0$ and any $\bm{x}\in\mathrm{dom}\,P\cap\mathrm{dom}\,\phi$,
\begin{equation}\label{sumFbd-tmp-IC23}
\begin{aligned}
&\quad F(\widetilde{\bm{x}}^{i+1}) - F(\bm{x})  \\
&\leq \lambda\mathcal{D}_{\phi}(\bm{x},\,\bm{x}^{i})
- \lambda\mathcal{D}_{\phi}(\bm{x},\,\bm{x}^{i+1})
- (\lambda-L)\mathcal{D}_{\phi}(\widetilde{\bm{x}}^{i+1},\,\bm{x}^{i})
+ \lambda\big(\mathcal{D}_{\phi}(\widetilde{\bm{x}}^{i+1},\,\bm{x}^{i+1}) + \lambda^{-1}\delta_i\big) \\
&\leq \lambda\mathcal{D}_{\phi}(\bm{x},\,\bm{x}^{i})
- \lambda\mathcal{D}_{\phi}(\bm{x},\,\bm{x}^{i+1})
- (\lambda-L-\lambda\sigma)\mathcal{D}_{\phi}(\widetilde{\bm{x}}^{i+1},\,\bm{x}^{i}),
\end{aligned}
\end{equation}
which, together with $F(\widetilde{\bm{x}}^{i+1}) \geq F^*$ for all $i\geq0$, implies that
\begin{equation*}
(\lambda-L-\lambda\sigma)\mathcal{D}_{\phi}(\widetilde{\bm{x}}^{i+1},\,\bm{x}^{i})
\leq \lambda\mathcal{D}_{\phi}(\bm{x},\,\bm{x}^{i})
- \lambda\mathcal{D}_{\phi}(\bm{x},\,\bm{x}^{i+1})
+ e(\bm{x}).
\end{equation*}
Summing this inequality from $i=0$ to $i=k$, we have that
\begin{equation*}
\begin{aligned}
(\lambda-L-\lambda\sigma){\textstyle\sum^k_{i=0}}\mathcal{D}_{\phi}(\widetilde{\bm{x}}^{i+1},\,\bm{x}^{i})
&\leq \lambda\mathcal{D}_{\phi}(\bm{x},\,\bm{x}^{0})
- \lambda\mathcal{D}_{\phi}(\bm{x},\,\bm{x}^{k+1})
+ (k+1)e(\bm{x}) \\
&\leq \lambda\mathcal{D}_{\phi}(\bm{x},\,\bm{x}^{0})
+ (k+1)e(\bm{x}),
\end{aligned}
\end{equation*}
which, together with $\lambda>L$ and $0<\sigma<\frac{\lambda-L}{\lambda}$, yields
\begin{equation*}
{\textstyle\sum^k_{i=0}}\mathcal{D}_{\phi}(\widetilde{\bm{x}}^{i+1},\,\bm{x}^{i})
\leq (\lambda-L-\lambda\sigma)^{-1}\big(\lambda\mathcal{D}_{\phi}(\bm{x},\,\bm{x}^{0})
+ (k+1)e(\bm{x})\big).
\end{equation*}
Moreover, by (ReSC), we have $\mathcal{D}_{\phi}(\widetilde{\bm{x}}^{i+1},\,\bm{x}^{i+1})\leq\sigma\mathcal{D}_{\phi}(\widetilde{\bm{x}}^{i+1},\,\bm{x}^{i})$ for all $i\geq0$. Thus, we obtain that
\begin{equation}\label{sumablediff-IC23}
\begin{aligned}
{\textstyle\sum^{k}_{i=0}}\mathcal{D}_{\phi}(\widetilde{\bm{x}}^{i+1},\,\bm{x}^{i+1})
&\leq
\sigma {\textstyle\sum^{k}_{i=0}}\mathcal{D}_{\phi}(\widetilde{\bm{x}}^{i+1},\,\bm{x}^{i}) \\
&\leq (\lambda-L-\lambda\sigma)^{-1}\sigma\big(\lambda\mathcal{D}_{\phi}(\bm{x},\,\bm{x}^{0})
+ (k+1)e(\bm{x})\big).
\end{aligned}
\end{equation}
Now, if problem \eqref{mainpro} has an optimal solution $\bm{x}^*\in\mathrm{dom}\,\phi$, we can substitute $\bm{x}^*$ in the above relations and obtain from $e(\bm{x}^*)=0$ that
\begin{equation*}
\begin{aligned}
{\textstyle\sum^k_{i=0}}\mathcal{D}_{\phi}(\widetilde{\bm{x}}^{i+1},\,\bm{x}^{i})
&\leq (\lambda-L-\lambda\sigma)^{-1}\lambda\mathcal{D}_{\phi}(\bm{x}^*,\,\bm{x}^{0}), \\
{\textstyle\sum^{k}_{i=0}}\mathcal{D}_{\phi}(\widetilde{\bm{x}}^{i+1},\,\bm{x}^{i+1})
&\leq (\lambda-L-\lambda\sigma)^{-1}\sigma\lambda\mathcal{D}_{\phi}(\bm{x}^*,\,\bm{x}^{0}).
\end{aligned}
\end{equation*}
Then, we have that $\sum^{\infty}_{k=0}\mathcal{D}_{\phi}(\widetilde{\bm{x}}^{k+1},\,\bm{x}^{k})<\infty$ and $\sum^{\infty}_{k=0}\mathcal{D}_{\phi}(\widetilde{\bm{x}}^{k+1},\,\bm{x}^{k+1})<\infty$.

\textit{Statement (ii)}.
From \eqref{sumFbd-tmp-IC23} with $\lambda>L$ and $0<\sigma<\frac{\lambda-L}{\lambda}$, we also have that, for any $i\geq0$,
\begin{equation*}
F(\widetilde{\bm{x}}^{i+1}) - F(\bm{x})
\leq \lambda\mathcal{D}_{\phi}(\bm{x},\,\bm{x}^{i})
- \lambda\mathcal{D}_{\phi}(\bm{x},\,\bm{x}^{i+1}).
\end{equation*}
Summing this inequality from $i=0$ to $i=k$ gives
\begin{equation}\label{sumFbd-IC2}
{\textstyle\sum^k_{i=0}}F(\widetilde{\bm{x}}^{i+1})
- (k+1)F(\bm{x})
\leq \lambda\mathcal{D}_{\phi}(\bm{x},\,\bm{x}^{0})
- \lambda\mathcal{D}_{\phi}(\bm{x},\,\bm{x}^{k+1})
\leq \lambda\mathcal{D}_{\phi}(\bm{x},\,\bm{x}^{0}),
\end{equation}
which, together with $F\big(\frac{1}{k+1}\sum^k_{i=0}\widetilde{\bm{x}}^{i+1}\big)\leq\frac{1}{k+1}\sum^k_{i=0}F(\widetilde{\bm{x}}^{i+1})$ (by the convexity of $F$), yields \eqref{compavgk-iBPGM-IC23}. Moreover, one can see from \eqref{compavgk-iBPGM-IC23} that
\begin{equation*}
\limsup\limits_{k\to\infty}\,
F\big({\textstyle\frac{1}{k}\sum^{k-1}_{i=0}}\,\widetilde{\bm{x}}^{i+1}\big) \leq F(\bm{x}), \quad \forall\,\bm{x}\in\mathrm{dom}\,P\cap\mathrm{dom}\,\phi.
\end{equation*}
This together with $\frac{1}{k}\sum^{k-1}_{i=0}\widetilde{\bm{x}}^{i+1}\in\mathrm{dom}\,P\cap\mathrm{dom}\,\phi\subseteq
\mathrm{dom}\,P\cap\mathrm{dom}\,f\cap Q$ and \eqref{infeq} implies that
\begin{equation*}
F^*
\leq\liminf\limits_{k\to\infty}\,
F\big({\textstyle\frac{1}{k}\sum^{k-1}_{i=0}}\,\widetilde{\bm{x}}^{i+1}\big)  \leq\limsup\limits_{k\to\infty}\,
F\big({\textstyle\frac{1}{k}\sum^{k-1}_{i=0}}\,\widetilde{\bm{x}}^{i+1}\big)
\leq F^*,
\end{equation*}
from which we can conclude that $F\big({\textstyle\frac{1}{k}\sum^{k-1}_{i=0}}\widetilde{\bm{x}}^{i+1}\big) \to F^*$. In addition, if problem \eqref{mainpro} has an optimal solution $\bm{x}^*\in\mathrm{dom}\,\phi$, one can see from \eqref{compavgk-iBPGM-IC23} with $\bm{x}^*$ in place of $\bm{x}$ that $F\big(\frac{1}{k}\sum^{k-1}_{i=0}\widetilde{\bm{x}}^{i+1}\big) - F^* \leq \mathcal{O}(k^{-1})$.

\textit{Statement (iii)}. For any $i\geq0$, by setting $\bm{x}=\widetilde{\bm{x}}^i$ in \eqref{suffdes-iBPGM}, we have
\begin{equation*}
\begin{aligned}
&\quad F(\widetilde{\bm{x}}^{i+1}) - F(\widetilde{\bm{x}}^i)
\;\leq \;\lambda\mathcal{D}_{\phi}(\widetilde{\bm{x}}^i,\,\bm{x}^{i})
- \lambda\mathcal{D}_{\phi}(\widetilde{\bm{x}}^i,\,\bm{x}^{i+1})
- (\lambda-L-\lambda\sigma)\mathcal{D}_{\phi}(\widetilde{\bm{x}}^{i+1},\,\bm{x}^{i}) \\
&\leq \lambda\mathcal{D}_{\phi}(\widetilde{\bm{x}}^i,\,\bm{x}^{i}) - (\lambda-L-\lambda\sigma)\sigma^{-1}\mathcal{D}_{\phi}(\widetilde{\bm{x}}^{i+1},\,\bm{x}^{i+1}) \\
&= \lambda\mathcal{D}_{\phi}(\widetilde{\bm{x}}^{i},\,\bm{x}^{i})
-
\lambda\mathcal{D}_{\phi}(\widetilde{\bm{x}}^{i+1},\,\bm{x}^{i+1})
+ \big[2\lambda-(\lambda-L)\sigma^{-1}\big]\mathcal{D}_{\phi}(\widetilde{\bm{x}}^{i+1},\,\bm{x}^{i+1}),
\end{aligned}
\end{equation*}
where the second inequality follows from $\sigma>0$ and $\mathcal{D}_{\phi}(\widetilde{\bm{x}}^{i+1},\,\bm{x}^{i+1})\leq\sigma\mathcal{D}_{\phi}(\widetilde{\bm{x}}^{i+1},\,\bm{x}^{i})$ for all $i\geq0$ by (ReSC). Then, it follows from the above relation that, for any $i\geq0$,
\begin{equation*}
i\big[F(\widetilde{\bm{x}}^{i+1})+\lambda\mathcal{D}_{\phi}(\widetilde{\bm{x}}^{i+1},\,\bm{x}^{i+1})\big] \leq i\big[F(\widetilde{\bm{x}}^{i})+\lambda\mathcal{D}_{\phi}(\widetilde{\bm{x}}^{i},\,\bm{x}^{i})\big] +\big[2\lambda-(\lambda-L)\sigma^{-1}\big]i\mathcal{D}_{\phi}(\widetilde{\bm{x}}^{i+1},\,\bm{x}^{i+1}),
\end{equation*}
which implies that
\begin{equation*}
\begin{aligned}
F(\widetilde{\bm{x}}^{i+1})+\lambda\mathcal{D}_{\phi}(\widetilde{\bm{x}}^{i+1},\,\bm{x}^{i+1})
&\geq (i+1)\big[F(\widetilde{\bm{x}}^{i+1})+\lambda\mathcal{D}_{\phi}(\widetilde{\bm{x}}^{i+1},\,\bm{x}^{i+1})\big]
-i\big[F(\widetilde{\bm{x}}^{i})+\lambda\mathcal{D}_{\phi}(\widetilde{\bm{x}}^{i},\,\bm{x}^{i})\big] \\
&\qquad  -\big[2\lambda-(\lambda-L)\sigma^{-1}\big]i\mathcal{D}_{\phi}(\widetilde{\bm{x}}^{i+1},\,\bm{x}^{i+1}).
\end{aligned}
\end{equation*}
Summing this inequality from $i=0$ to $i=k$ yields
\begin{equation*}
\begin{aligned}
&\quad {\textstyle\sum^{k}_{i=0}}\big[F(\widetilde{\bm{x}}^{i+1})+\lambda\mathcal{D}_{\phi}(\widetilde{\bm{x}}^{i+1},\,\bm{x}^{i+1})\big] \\
&\geq (k+1)\big[F(\widetilde{\bm{x}}^{k+1})+\lambda\mathcal{D}_{\phi}(\widetilde{\bm{x}}^{k+1},\,\bm{x}^{k+1})\big]
- \big[2\lambda-(\lambda-L)\sigma^{-1}\big]{\textstyle\sum^{k}_{i=0}}\,i\mathcal{D}_{\phi}(\widetilde{\bm{x}}^{i+1},\,\bm{x}^{i+1}) \\
&\geq (k+1)F(\widetilde{\bm{x}}^{k+1}) - \big[2\lambda-(\lambda-L)\sigma^{-1}\big]{\textstyle\sum^{k}_{i=0}}\,i\mathcal{D}_{\phi}(\widetilde{\bm{x}}^{i+1},\,\bm{x}^{i+1}).
\end{aligned}
\end{equation*}
This together with \eqref{sumFbd-IC2} implies that
\begin{equation*}
\begin{aligned}
&\quad (k+1)\big(F(\widetilde{\bm{x}}^{k+1})-F(\bm{x})\big) \\
&\leq {\textstyle\sum^{k}_{i=0}}F(\widetilde{\bm{x}}^{i+1}) - (k+1)F(\bm{x})
+ \lambda{\textstyle\sum^{k}_{i=0}}\mathcal{D}_{\phi}(\widetilde{\bm{x}}^{i+1},\,\bm{x}^{i+1})
+
\big[2\lambda-(\lambda-L)\sigma^{-1}\big]{\textstyle\sum^{k}_{i=0}}\,i\mathcal{D}_{\phi}(\widetilde{\bm{x}}^{i+1},\,\bm{x}^{i+1})
\\
&\leq \lambda\mathcal{D}_{\phi}(\bm{x},\,\bm{x}^{0})
+ \lambda{\textstyle\sum^{k}_{i=0}}\mathcal{D}_{\phi}(\widetilde{\bm{x}}^{i+1},\,\bm{x}^{i+1})
+ \big[2\lambda-(\lambda-L)\sigma^{-1}\big]{\textstyle\sum^{k}_{i=0}}\,i\mathcal{D}_{\phi}(\widetilde{\bm{x}}^{i+1},\,\bm{x}^{i+1}) \\
&\leq a\mathcal{D}_{\phi}(\bm{x},\,\bm{x}^{0})
+ b(k+1)e(\bm{x}) + c\,{\textstyle\sum^{k}_{i=0}}\,i\mathcal{D}_{\phi}(\widetilde{\bm{x}}^{i+1},\,\bm{x}^{i+1}),
\end{aligned}
\end{equation*}
where the last inequality follows from \eqref{sumablediff-IC23} with $a:=(1+b)\lambda$, $b:=(\lambda-L-\lambda\sigma)^{-1}\sigma\lambda$, and $c:=2\lambda-(\lambda-L)\sigma^{-1}$. Then, dividing the above inequality by $k+1$, we can get \eqref{compk-iBPGM-IC2}. We next consider the following three cases.
\begin{itemize}
\item If problem \eqref{mainpro} has an optimal solution $\bm{x}^*\in\mathrm{dom}\,\phi$, we have from \eqref{compk-iBPGM-IC2} with $\bm{x}^*$ in place of $\bm{x}$ that
    \begin{equation}\label{compkic2star}
    \begin{aligned}
    F(\widetilde{\bm{x}}^{k+1}) - F^*
    \leq {\textstyle\frac{1}{k+1}}\left(
    a\,\mathcal{D}_{\phi}(\bm{x}^*,\bm{x}^{0})
    + c\,{\textstyle\sum^{k}_{i=0}}i\mathcal{D}_{\phi}(\widetilde{\bm{x}}^{i+1},\bm{x}^{i+1})\right).
    \end{aligned}
    \end{equation}
    Since  $\sum^{\infty}_{k=0}\mathcal{D}_{\phi}(\widetilde{\bm{x}}^{k+1},\,\bm{x}^{k+1})<\infty$ by statement (i), it then follows from Lemma \ref{lemseqcon2} that $\frac{1}{k+1}\sum^{k}_{i=0}i\mathcal{D}_{\phi}(\widetilde{\bm{x}}^{i+1},\bm{x}^{i+1})\to0$.
    This fact together with \eqref{compkic2star} implies that $F(\widetilde{\bm{x}}^k) \to F^*$.

\item If $0<\sigma\leq\frac{\lambda-L}{2\lambda}$, we see that $c:=2\lambda-(\lambda-L)\sigma^{-1}\leq0$ and hence
    \begin{equation*}
    \begin{aligned}
    F(\widetilde{\bm{x}}^{k+1}) - F(\bm{x})
    \leq {\textstyle\frac{a}{k+1}}\mathcal{D}_{\phi}(\bm{x},\bm{x}^{0})
    + be(\bm{x}), \quad \forall\,\bm{x}\in\mathrm{dom}\,P\cap\mathrm{dom}\,\phi.
    \end{aligned}
    \end{equation*}
    Then, we have that
    \begin{equation*}
    \limsup\limits_{k\to\infty}\,
    F\big(\widetilde{\bm{x}}^{k}\big) \leq F(\bm{x}) + be(\bm{x}), \quad \forall\,\bm{x}\in\mathrm{dom}\,P\cap\mathrm{dom}\,\phi.
    \end{equation*}
    This together with $\widetilde{\bm{x}}^{k}\in\mathrm{dom}\,P\cap\mathrm{dom}\,\phi\subseteq
    \mathrm{dom}\,P\cap\mathrm{dom}\,f\cap Q$ and \eqref{infeq} implies that
    \begin{equation*}
    F^*\leq\liminf\limits_{k\to\infty}\,
    F\big(\widetilde{\bm{x}}^{k}\big)\leq\limsup\limits_{k\to\infty}\,
    F\big(\widetilde{\bm{x}}^{k}\big)
    \leq F^*,
    \end{equation*}
    from which, we can conclude that $F\big(\widetilde{\bm{x}}^{k}\big) \to F^*$.

\item If $\frac{1}{k}\sum^{k-1}_{i=0}i\mathcal{D}_{\phi}(\widetilde{\bm{x}}^{i+1},\bm{x}^{i+1})\to0$, one can see from \eqref{compk-iBPGM-IC2} that
    \begin{equation*}
    \limsup\limits_{k\to\infty}\,
    F\big(\widetilde{\bm{x}}^{k}\big) \leq F(\bm{x}) + be(\bm{x}), \quad \forall\,\bm{x}\in\mathrm{dom}\,P\cap\mathrm{dom}\,\phi.
    \end{equation*}
    Then, using the same arguments as above, we have that $F\big(\widetilde{\bm{x}}^{k}\big) \to F^*$.

\end{itemize}

In addition, when problem \eqref{mainpro} has an optimal solution $\bm{x}^*\in\mathrm{dom}\,\phi$, we have \eqref{compkic2star}. In this case, if $0<\sigma\leq\frac{\lambda-L}{2\lambda}$ or $\sum^{\infty}_{k=0}k\mathcal{D}_{\phi}(\widetilde{\bm{x}}^{k+1},\bm{x}^{k+1})<\infty$, one can see that the right-hand side of \eqref{compkic2star} is finite and hence $F(\widetilde{\bm{x}}^k) - F^* \leq \mathcal{O}\left(\frac{1}{k}\right)$. This completes the proof.
\end{proof}

Theorems \ref{thm-fval-iBPGM-IC1} and \ref{thm-fval-iBPGM-IC2} give the convergence rate of our iBPGM under (AbSC) and (ReSC), respectively. To establish the convergence result for the sequence of iterates, we need to make the following additional assumptions on the kernel function $\phi$.

\begin{assumption}\label{assumB}
Assume that the kernel function $\phi$ satisfies the following conditions.
\begin{itemize}[leftmargin=1cm]
\item[{\bf B1.}] $\mathrm{dom}\,\phi = \overline{\mathrm{dom}}\,\phi = \mathcal{Q}$, i.e., the domain of $\phi$ is closed.

\item[{\bf B2.}] For any $\bm{x}\in\mathrm{dom}\,\phi$ and $\alpha\in\mathbb{R}$, the level set $\big\{\bm{y}\in\mathrm{int}\,\mathrm{dom}\,\phi:
    \mathcal{D}_{\phi}(\bm{x},\,\bm{y})\leq\alpha\big\}$ is bounded.

\item[{\bf B3.}] If $\{\bm{x}^k\}\subseteq\mathrm{int}\,\mathrm{dom}\,\phi$ converges to some $\bm{x}^{*}\in\mathrm{dom}\,\phi$, then $\mathcal{D}_{\phi}(\bm{x}^{*},\,\bm{x}^k)\to0$.

\item[{\bf B4.}] (Convergence consistency) If $\{\bm{x}^k\}\subseteq\mathrm{dom}\,\phi$ and $\{\bm{y}^k\}\subseteq\mathrm{int}\,\mathrm{dom}\,\phi$ are two sequences such that $\{\bm{x}^k\}$ is bounded, $\bm{y}^k \to \bm{y}^{*}$ and $\mathcal{D}_{\phi}(\bm{x}^k,\,\bm{y}^k)\to0$, then $\bm{x}^k \to \bm{y}^{*}$.
\end{itemize}
\end{assumption}

All or part of the above assumptions are typically utilized in the study of the convergence of the Bregman proximal point algorithm or the Bregman proximal gradient method in the convex setting; see, for example, \cite{bbt2017descent,cl1981iterative,cz1992proximal,e1993nonlinear,e1998approximate,ss2000inexact,t2018simplified,yt2020bregman}. These assumptions can be satisfied by several well-known kernel functions, such as the quadratic kernel function $\phi(\bm{x}):=\frac{1}{2}\|\bm{x}\|^2$, the entropy kernel function $\phi(\bm{x}):=\sum_{i=1}^nx_{i}(\log x_{i}-1)$, and the Hellinger kernel function $\phi(\bm{x}):=-\sum_{i=1}^n\sqrt{1-x_i^2}$. Indeed, Bauschke and Borwein have demonstrated in \cite[Proposition 3.3 and Remark 3.4]{bb1997legendre} that Assumption \ref{assumB}3 can be satisfied by all separable essentially smooth functions. Moreover, Solodov and Svaiter have shown in \cite[Theorem 2.4]{ss2000inexact} that Assumption \ref{assumB}4 (convergence consistency) holds under the following conditions: (i) $\phi$ is strictly convex and continuous on $\mathrm{dom}\,\phi$; (ii) $\phi$ is continuously differentiable on $\mathrm{int}\,\mathrm{dom}\,\phi$; (iii) Assumptions \ref{assumB}1, \ref{assumB}2 and \ref{assumB}3 are satisfied. More in-depth discussions and explanations on functions with these properties can be found in \cite{bb1997legendre,bbc2003redundant} and references therein. With Assumption \ref{assumB}, we are now ready to present the convergence results for the iterates generated by Algorithm \ref{algo-iBPGM}.

\begin{theorem}\label{thm-conv-iBPGM}
Suppose that Assumptions \ref{assumA} and \ref{assumB} hold. Moreover, suppose that $\lambda \geq L$ and $\sum\varepsilon_k<\infty$ for (AbSC), and that $\lambda>L$ and $0<\sigma<\frac{\lambda-L}{\lambda}$ for (ReSC). Let $\{\bm{x}^k\}$ and $\{\widetilde{\bm{x}}^k\}$ be the sequences generated by the iBPGM in Algorithm \ref{algo-iBPGM}. If $\{\widetilde{\bm{x}}^k\}$ is bounded and the optimal solution set of problem \eqref{mainpro} is nonempty, then the following statements hold.
\begin{itemize}
\item[{\rm (i)}] Any cluster point of $\{\widetilde{\bm{x}}^k\}$ is an optimal solution of problem \eqref{mainpro}.

\item[{\rm (ii)}] The sequences $\{\bm{x}^{k}\}$ and $\{\widetilde{\bm{x}}^{k}\}$ converge to the same limit that is an optimal solution of problem \eqref{mainpro}.
\end{itemize}
\end{theorem}
\begin{proof}
\textit{Statement (i)}.
For (AbSC), since $\sum\varepsilon_k<\infty$, it follows from Lemma \ref{lemseqcon2} that $\frac{1}{k}\sum^{k-1}_{i=0}i\varepsilon_i\to0$. This together with Theorem \ref{thm-fval-iBPGM-IC1}(iii) implies that $F(\widetilde{\bm{x}}^k) \to F^*$. On the other hand, for (ReSC), we also have from (iii) and (iv) in Theorem \ref{thm-fval-iBPGM-IC2} that $F(\widetilde{\bm{x}}^k) \to F^*$. Now, suppose that $\widetilde{\bm{x}}^{\infty}$ is a cluster point (which must exist since $\{\widetilde{\bm{x}}^k\}$ is bounded) and $\{\widetilde{\bm{x}}^{k_i}\}$ is a convergent subsequence such that $\lim\limits_{i\to\infty}\widetilde{\bm{x}}^{k_i} = \widetilde{\bm{x}}^{\infty}$. Then,
\begin{equation*}
\min\left\{F(\bm{x}) : \bm{x}\in\mathcal{Q} \right\}
= F^* = \lim\limits_{k\to\infty}\,F(\widetilde{\bm{x}}^{k}) = \lim\limits_{k_i\to\infty}\,F(\widetilde{\bm{x}}^{k_i}) \geq F(\widetilde{\bm{x}}^{\infty}),
\end{equation*}
where the last inequality follows from the lower semicontinuity of $F$ (since $P$ and $f$ are closed by Assumptions \ref{assumA}2\&3). This implies that $F(\widetilde{\bm{x}}^{\infty})$ is finite and hence  $\widetilde{\bm{x}}^{\infty}\in\mathrm{dom}\,F$. Moreover, since $\widetilde{\bm{x}}^k\in\mathrm{dom}\,P\cap\mathrm{dom}\,\phi\subseteq\mathrm{dom}\,\phi$, then $\widetilde{\bm{x}}^{\infty}\in\overline{\mathrm{dom}}\,\phi=\mathcal{Q}$. Therefore, $\widetilde{\bm{x}}^{\infty}$ is an optimal solution of problem \eqref{mainpro}. This proves statement (i).

\textit{Statement (ii)}.
Let $\bm{x}^*$ be an arbitrary optimal solution of problem \eqref{mainpro}. From Assumptions \ref{assumA}3 and \ref{assumB}1, we see that $\bm{x}^*\in\mathrm{dom}\,P\cap\mathrm{dom}\,f\cap\mathcal{Q} = \mathrm{dom}\,P\cap\mathrm{dom}\,\phi$. Then, we set $\bm{x}=\bm{x}^*$ in \eqref{suffdes-iBPGM} and obtain after rearranging the resulting inequality that
\begin{equation}\label{suffdes-iBPGM-xstar}
\begin{aligned}
\mathcal{D}_{\phi}(\bm{x}^*,\,\bm{x}^{k+1})
&\leq \mathcal{D}_{\phi}(\bm{x}^*,\,\bm{x}^{k})
- \lambda^{-1}(F(\widetilde{\bm{x}}^{k+1})-F(\bm{x}^*))
- \lambda^{-1}(\lambda-L)\mathcal{D}_{\phi}(\widetilde{\bm{x}}^{k+1},\,\bm{x}^{k}) \\
&\qquad + \big(\mathcal{D}_{\phi}(\widetilde{\bm{x}}^{k+1},\,\bm{x}^{k+1}) + \lambda^{-1}\delta_k\big)
+ \lambda^{-1}|\langle \Delta^k, \,\widetilde{\bm{x}}^{k+1} - \bm{x}^* \rangle|.
\end{aligned}
\end{equation}
We next show that $\{\mathcal{D}_{\phi}(\bm{x}^*,\,\bm{x}^{k})\}$ is convergent.
\begin{itemize}
\item For (AbSC), we have from \eqref{suffdes-iBPGM-xstar}, $\lambda \geq L$ and $F(\widetilde{\bm{x}}^{k+1}) \geq F(\bm{x}^*)$ for all $k\geq0$ that $\mathcal{D}_{\phi}(\bm{x}^*,\,\bm{x}^{k+1})
    \leq \mathcal{D}_{\phi}(\bm{x}^*,\,\bm{x}^{k})
    + \big(\lambda^{-1}\|\widetilde{\bm{x}}^{k+1}-\bm{x}^*\|+1\big)\varepsilon_{k}$. Then, using this fact, the nonnegativity of $\mathcal{D}_{\phi}(\bm{x}^*,\,\bm{x}^{k})$, $\sum\varepsilon_k<\infty$, the boundedness of $\{\widetilde{\bm{x}}^k\}$, and Lemma \ref{lemseqcon}, we obtain that $\{\mathcal{D}_{\phi}(\bm{x}^*,\,\bm{x}^{k})\}$ is convergent.

\item For (ReSC), we have from \eqref{suffdes-iBPGM-xstar} that
    \begin{equation*}
    \begin{aligned}
    \mathcal{D}_{\phi}(\bm{x}^*,\,\bm{x}^{k+1})
    &\leq \mathcal{D}_{\phi}(\bm{x}^*,\,\bm{x}^{k})
    - \lambda^{-1}(F(\widetilde{\bm{x}}^{k+1})-F(\bm{x}^*))
    - \lambda^{-1}(\lambda-L-\lambda\sigma)\mathcal{D}_{\phi}(\widetilde{\bm{x}}^{k+1},\,\bm{x}^{k}) \\
    &\leq \mathcal{D}_{\phi}(\bm{x}^*,\,\bm{x}^{k}),
    \end{aligned}
    \end{equation*}
    where the last inequality follows from $\lambda>L$, $0<\sigma<\frac{\lambda-L}{\lambda}$, and $F(\widetilde{\bm{x}}^{k+1}) \geq F(\bm{x}^*)$ for all $k\geq0$. Thus, $\{\mathcal{D}_{\phi}(\bm{x}^*,\,\bm{x}^{k})\}$ is nonincreasing. This together with the nonnegativity of $\mathcal{D}_{\phi}(\bm{x}^*,\,\bm{x}^{k})$ implies that $\{\mathcal{D}_{\phi}(\bm{x}^*,\,\bm{x}^{k})\}$ is convergent.

\end{itemize}

Since $\{\mathcal{D}_{\phi}(\bm{x}^*,\,\bm{x}^{k})\}$ is convergent, it then follows from Assumption \ref{assumB}2
that the sequence $\{\bm{x}^{k}\}$ is bounded and hence has at least one cluster point. Suppose that $\bm{x}^{\infty}$ is a cluster point and $\{\bm{x}^{k_j}\}$ is a convergent subsequence such that $\lim_{j\to\infty} \bm{x}^{k_j} = \bm{x}^{\infty}$.
Then, from $\lim_{j\to\infty}\mathcal{D}_{\phi}(\widetilde{\bm{x}}^{k_j}, \,\bm{x}^{k_j}) = 0$ (by Theorem \ref{thm-fval-iBPGM-IC2}(i)),
the boundedness of $\{\widetilde{\bm{x}}^{k_j}\}$ and
Assumption \ref{assumB}4, we have that $\lim_{j\to\infty}\widetilde{\bm{x}}^{k_j} = \bm{x}^{\infty}$. This together with statement (i) implies that $\bm{x}^{\infty}$ is an optimal solution of problem \eqref{mainpro}. Moreover, by using \eqref{suffdes-iBPGM-xstar} with $\bm{x}^*$ replaced by $\bm{x}^{\infty}$, we can conclude that $\{\mathcal{D}_{\phi}(\bm{x}^{\infty},\,\bm{x}^k)\}$ is convergent. On the other hand, it follows from $\lim_{j\to\infty} \bm{x}^{k_j} = \bm{x}^{\infty}$ and Assumption \ref{assumB}3 that $\mathcal{D}_{\phi}(\bm{x}^{\infty}, \,\bm{x}^{k_j})\to0$. Consequently, we must have that $\mathcal{D}_{\phi}(\bm{x}^{\infty},\,\bm{x}^k)\to0$. Now, let $\bm{z}$ be arbitrary cluster point of $\{\bm{x}^k\}$ with a convergent subsequence $\{\bm{x}^{k'_j}\}$ such that $\bm{x}^{k'_j}\to\bm{z}$. Since $\mathcal{D}_{\phi}(\bm{x}^{\infty},\,\bm{x}^k)\to0$, then we have $\mathcal{D}_{\phi}(\bm{x}^{\infty},\,\bm{x}^{k'_j})\to0$. From this and Assumption \ref{assumB}4, we see that $\bm{x}^{\infty}=\bm{z}$. Since $\bm{z}$ is arbitrary, we can conclude that $\lim_{k\to\infty}\bm{x}^k=\bm{x}^{\infty}$. Finally, using this together with the boundedness of $\{\widetilde{\bm{x}}^k\}$, $\mathcal{D}_{\phi}(\widetilde{\bm{x}}^{k}, \,\bm{x}^{k})\to0$ and Assumption \ref{assumB}4, we deduce that $\{\widetilde{\bm{x}}^k\}$ also converges to $\bm{x}^{\infty}$. This completes the proof.
\end{proof}

\section{An inertial variant of the iBPGM}\label{sec-v-iBPGM}

In this section, we shall develop an inertial variant of our iBPGM based on two types of inexact stopping criteria to obtain a possibly faster convergence rate. Before proceeding, we introduce an additional restricted relative smoothness assumption, which modifies the one proposed by Gutman and Pe{\~n}a in
\cite[Section 3.3]{gp2018perturbed}. This assumption is crucial for developing the convergence rate of our inertial variant and can subsume the related conditions used in \cite{at2006interior,hrx2018accelerated,t2008on,t2010approximation} for developing inertial methods.

\begin{assumption}\label{assumC}
In addition to Assumption \ref{assumA}3 with a closed convex set $\mathcal{X}\supseteq\mathrm{dom}\,P\cap\mathrm{dom}\,\phi$, $f$ also satisfies the following smoothness condition relative to $\phi$ restricted on $\mathcal{X}$: there exist two constants $\tau>0$ and $\gamma\geq1$ such that, for any $\bm{x},\,\widetilde{\bm{z}}\in\mathrm{dom}\,P\cap\mathrm{dom}\,\phi$ and $\bm{z}\in\mathrm{int}\,\mathrm{dom}\,\phi\cap\mathcal{X}$,
\begin{equation}\label{ineq-smooth}
\mathcal{D}_{f}\big((1-\theta)\bm{x}+\theta\widetilde{\bm{z}}, \,(1-\theta)\bm{x}+\theta\bm{z}\big) \leq \tau L\,\theta^{\gamma}\,\mathcal{D}_{\phi}(\widetilde{\bm{z}},\,\bm{z}), \quad \forall\,\theta\in(0,\,1].
\end{equation}
Here, $\gamma$ is called the restricted relative smoothness exponent of $f$ relative to $\phi$ restricted on $\mathcal{X}$.
\end{assumption}

We provide below two representative examples of $(f,\,\phi)$ satisfying Assumption \ref{assumC}.

\begin{example}\label{example1}
If $\nabla f$ is $L_f$-Lipschitz continuous on $\mathcal{X}$ and $\phi$ is $\mu_{\phi}$-strongly convex on $\mathcal{X}$, one can verify that
\begin{equation*}
\mathcal{D}_{f}\big((1-\theta)\bm{x}+\theta\widetilde{\bm{z}}, \,(1-\theta)\bm{x}+\theta\bm{z}\big) \leq \frac{L_f\theta^2}{2}\|\widetilde{\bm{z}}-\bm{z}\|^2 \leq \mu_{\phi}^{-1} L_f\,\theta^2\,\mathcal{D}_{\phi}(\widetilde{\bm{z}},\,\bm{z}).
\end{equation*}
Thus, \eqref{ineq-smooth} holds with $\tau=\mu_{\phi}^{-1}$, $L=L_f$ and $\gamma=2$. Here, we would like to point out that, without the restriction on the set $\mathcal{X}$, the above inequality may not hold for the entropy kernel function $\phi(\bm{x})=\sum_{i}x_{i}(\log x_{i}-1)$ even when $\nabla f$ is Lipschitz continuous on $\mathbb{R}^n$, since the entropy kernel is not strongly convex on its whole domain $\mathbb{R}^n_+$. In contrast, when taking into consideration of the restriction on the set $\mathcal{X}$ and when $\mathcal{X}$ is a bounded and closed subset of $\mathbb{R}^n_+$, one can have the above inequality with the exponent $\gamma=2$ for the entropy kernel function. As we shall see later from Theorem \ref{thm-iABPGM-IC1}(iii), this would lead to a faster convergence rate of $\mathcal{O}(1/k^2)$. Therefore, introducing the set $\mathcal{X}$ is important to broaden the choices of $(f,\,\phi)$.
\end{example}

\begin{example}
If $f$ is $L$-smooth relative to $\phi$ restricted on $\mathcal{X}$ and the Bregman distance $\mathcal{D}_{\phi}$ has the so-called triangle scaling property (TSP) \cite[Definition 2]{hrx2018accelerated}) with a triangle scaling exponent $\gamma\geq1$, one can verify that
\begin{equation*}
\mathcal{D}_{f}\big((1-\theta)\bm{x}+\theta\widetilde{\bm{z}}, \,(1-\theta)\bm{x}+\theta\bm{z}\big)
\leq L\mathcal{D}_{\phi}\big((1-\theta)\bm{x}+\theta\widetilde{\bm{z}}, \,(1-\theta)\bm{x}+\theta\bm{z}\big)
\leq L\theta^{\gamma}\mathcal{D}_{\phi}(\widetilde{\bm{z}},\,\bm{z}).
\end{equation*}
Thus, \eqref{ineq-smooth} holds with $\tau=1$ and some $\gamma\geq1$ determined by $\mathcal{D}_{\phi}$. For example, when considering the entropy function $\phi(\bm{x})=\sum_{i}x_{i}(\log x_{i}-1)$ as a kernel function, $\mathcal{D}_{\phi}(\cdot, \,\cdot)$ is jointly convex\footnote{More examples on the joint convexity of the Bregman distance can be found in \cite{bb2001joint}.} and hence $\mathcal{D}_{\phi}$ has the TSP with $\gamma=1$.
\end{example}

We also need to specify the conditions on the choice of the parameters sequence $\{\theta_k\}_{k=-1}^{\infty}$, which will be used for developing the inertial method in the sequel. Specifically, the sequence of the parameters $\{\theta_k\}_{k=-1}^{\infty}\subseteq(0,1]$ with $\theta_{-1}=\theta_0=1$ is chosen such that, for all $k\geq1$,
\begin{numcases}{}
\theta_k\leq\frac{\alpha-1}{k+\alpha-1}, ~~\alpha\geq\gamma+1, \label{condtheta1} \\
\vartheta_k := \frac{1}{\theta_{k-1}^{\gamma}} -  \frac{1-\theta_{k}}{\theta_{k}^{\gamma}} \geq 0, \label{condtheta}
\end{numcases}
where $\gamma\geq1$ is the restricted relative smoothness exponent specified in Assumption \ref{assumC}. These conditions are inspired by several works on  accelerated methods (see, for example, \cite{cd2015convergence,hrx2018accelerated,t2008on}), and they actually provide a unified and broader framework to choose $\{\theta_k\}$. Using similar arguments as in \cite[Lemma 3]{hrx2018accelerated}, one can show that, for any $\alpha\geq\gamma+1$, $\theta_k=\frac{\alpha-1}{k+\alpha-1}$, $\forall k\geq1$, satisfy conditions \eqref{condtheta1} and \eqref{condtheta}. One can also obtain a sequence $\{\theta_k\}$ by iteratively solving the equality form of \eqref{condtheta} via an appropriate root-finding procedure. It should be noted that these conditions allow $\{\theta_k\}$ to decrease but not too fast. Moreover, we have the following lemma whose proof can be found in Appendix \ref{apd-lemmas}.

\begin{lemma}\label{lem-thetasumbd}
For any sequence $\{\theta_k\}_{k=-1}^{\infty}\subseteq(0,1]$ with $\theta_{-1}=\theta_0=1$ satisfying conditions \eqref{condtheta1} and \eqref{condtheta}, we have that $\theta_k\geq\frac{1}{k+\gamma}$ for all $k\geq0$ and
\begin{equation*}
\sum^k_{i=0} \vartheta_i \leq 2 + \frac{(k+1+\gamma)^{\gamma}}{\gamma}.
\end{equation*}
\end{lemma}

We are now ready to present an inertial variant of our iBPGM (denoted by v-iBPGM for short) with two types of inexact stopping criteria for solving problem \eqref{mainpro}. The complete framework is presented as Algorithm \ref{algo-iABPGM}.

\begin{algorithm}[ht]
\caption{An inertial variant of the iBPGM (v-iBPGM) for solving problem \eqref{mainpro}}\label{algo-iABPGM}
\textbf{Input:} Let $\{\varepsilon_k\}_{k=0}^{\infty}$ be a sequence of nonnegative scalars, $\mathcal{X}\supseteq\mathrm{dom}\,P\cap\mathrm{dom}\,\phi$ be the closed convex set in Assumption \ref{assumA}3, and $\gamma\geq1$ be the restricted relative smoothness exponent specified in Assumption \ref{assumC}. Choose $\bm{x}^{0}\in\mathrm{dom}\,P\cap\mathrm{dom}\,\phi$ and $\bm{z}^0\in\mathcal{X}\cap\mathrm{int}\,\mathrm{dom}\,\phi$. Set $\theta_0=1$ and $k=0$.  \\
\textbf{while} a termination criterion is not met, \textbf{do} \vspace{-2mm}
\begin{itemize}[leftmargin=2cm]
\item[\textbf{Step 1}.] Compute $\bm{y}^k = (1-\theta_k)\bm{x}^k + \theta_k\bm{z}^k$.

\item[\textbf{Step 2}.] Find a pair $(\bm{z}^{k+1}, \,\widetilde{\bm{z}}^{k+1})$ and an error pair $(\Delta^k, \delta_k)$ by approximately solving
    \begin{equation}\label{subpro-iABPGM}
    \min\limits_{\bm{z}}~P(\bm{z}) + \langle \nabla f(\bm{y}^k), \,\bm{z}-\bm{y}^k\rangle
    + \lambda\,\theta_k^{\gamma-1}\mathcal{D}_{\phi}(\bm{z},\,\bm{z}^k),
    \end{equation}
    such that $\bm{z}^{k+1} \in \mathcal{X}\cap \mathrm{int}\,\mathrm{dom}\,\phi$, $\widetilde{\bm{z}}^{k+1} \in \mathrm{dom}\,P\cap\mathrm{dom}\,\phi$ and
    \begin{equation}\label{inexcond-iABPGM}
    \Delta^k \in \partial_{\delta_k} P(\widetilde{\bm{z}}^{k+1}) + \nabla f(\bm{y}^k) + \lambda\,\theta_k^{\gamma-1}(\nabla \phi(\bm{z}^{k+1})-\nabla \phi(\bm{z}^{k})),
    \end{equation}
    satisfying \textit{either} an absolute-type stopping criterion (AbSC') \textit{or} a partial relative-type stopping criterion (ReSC') as follows:
    \begin{equation*}
    \begin{aligned}
    &\textbf{(AbSC')} \qquad \|\Delta^k\| + \lambda^{-1}\theta_k^{1-\gamma}\delta_k
    + \mathcal{D}_{\phi}(\widetilde{\bm{z}}^{k+1}, \,\bm{z}^{k+1})
    \leq \varepsilon_{k}, \\
    &\textbf{(ReSC')} \qquad \|\Delta^k\|=0 ~~\mbox{and}~~\lambda^{-1}\theta_k^{1-\gamma}\delta_k
    + \mathcal{D}_{\phi}(\widetilde{\bm{z}}^{k+1}, \,\bm{z}^{k+1})
    \leq \sigma\mathcal{D}_{\phi}(\widetilde{\bm{z}}^{k+1}, \,\bm{z}^{k}).
    \end{aligned}
    \end{equation*}

\item[\textbf{Step 3}.] Compute $\bm{x}^{k+1} = (1-\theta_k)\bm{x}^k + \theta_k\widetilde{\bm{z}}^{k+1}$.

\item[\textbf{Step 4}.] Choose $\theta_{k+1}\in(0,1]$ satisfying conditions \eqref{condtheta1} and \eqref{condtheta}.

\item[\textbf{Step 5}.] Set $k = k+1$ and go to \textbf{Step 1}. \vspace{-1.5mm}
\end{itemize}
\textbf{end while}  \\
\textbf{Output}: $\bm{x}^{k}$ \vspace{0.5mm}
\end{algorithm}

Our v-iBPGM in Algorithm \ref{algo-iABPGM} is inspired by Hanzely, Richt\'{a}rik and Xiao's accelerated Bregman proximal gradient method \cite{hrx2018accelerated}, which extends Auslender and Teboulle's improved interior gradient algorithm
\cite[Section 5]{at2006interior} and Tseng's extension \cite{t2008on,t2010approximation} to the relatively smooth setting. However, note that none of the methods in \cite{at2006interior,hrx2018accelerated,t2008on,t2010approximation} allow the subproblem to be solved \textit{approximately}. Moreover, the analysis in \cite{hrx2018accelerated} is established based on the relative smoothness condition and a crucial triangle scaling property (TSP) for the Bregman distance. Since TSP may not imply the strong convexity of the kernel function, the convergence results in \cite{hrx2018accelerated} actually cannot recover the related results in \cite{at2006interior,t2008on,t2010approximation} when $\nabla f$ is $L$-Lipschitz continuous. In contrast, our analysis shall use the restricted relative smoothness condition \eqref{relsmoothcond} together with a more
general condition \eqref{ineq-smooth} in Assumption \ref{assumC} that could be  satisfied by the conditions used in \cite{at2006interior,hrx2018accelerated,t2008on,t2010approximation}. Thus, our subsequent convergence results can readily subsume the related results in \cite{at2006interior,hrx2018accelerated,t2008on,t2010approximation} when the subproblem is solved exactly.

\begin{lemma}
Suppose that Assumptions \ref{assumA} and \ref{assumC} hold, $\lambda \geq \tau L$, and $0\leq\sigma\leq\frac{\lambda-\tau L}{\lambda}$. Let $\{\bm{x}^k\}$, $\{\bm{z}^k\}$ and $\{\widetilde{\bm{z}}^k\}$ be the sequences generated by the v-iBPGM in Algorithm \ref{algo-iABPGM}. Then, for any $k\geq0$ and any $\bm{x}\in\mathrm{dom}\,P\cap\mathrm{dom}\,\phi$, the following statements hold.
\begin{itemize}
\item For (AbSC'), we have that
      \begin{equation}\label{suffdes-iABPGM-IC1}
      \begin{aligned}
      &\quad \theta_{k+1}^{-\gamma}(1-\theta_{k+1})\big(F(\bm{x}^{k+1})-F(\bm{x})\big)
      + \lambda\mathcal{D}_{\phi}(\bm{x},\,\bm{z}^{k+1}) \\
      &\leq \theta_k^{-\gamma}(1-\theta_k)\big(F(\bm{x}^{k})-F(\bm{x})\big)
      + \lambda\mathcal{D}_{\phi}(\bm{x},\,\bm{z}^{k})
      + \vartheta_{k+1}e(\bm{x}) \\
      &\qquad
      + \big(\lambda+\|\bm{x}\|\theta_k^{1-\gamma}\big)\varepsilon_k
      + \theta_k^{1-\gamma}|\langle\Delta^k, \,\widetilde{\bm{z}}^{k+1}\rangle|.
      \end{aligned}
      \end{equation}

\item For (ReSC'), we have that
      \begin{equation}\label{suffdes-iABPGM-IC2}
      \begin{aligned}
      &\quad \theta_{k+1}^{-\gamma}(1-\theta_{k+1})\big(F(\bm{x}^{k+1})-F(\bm{x})\big)
      + \lambda\mathcal{D}_{\phi}(\bm{x},\,\bm{z}^{k+1}) \\
      &\leq \theta_k^{-\gamma}(1-\theta_k)\big(F(\bm{x}^{k})-F(\bm{x})\big)
      + \lambda\mathcal{D}_{\phi}(\bm{x},\,\bm{z}^{k})
      +\vartheta_{k+1} e(\bm{x}).
      \end{aligned}
      \end{equation}

\end{itemize}
Here, $\vartheta_{k+1}:=\theta_{k}^{-\gamma}-\theta_{k+1}^{-\gamma}(1-\theta_{k+1})$ and $e(\bm{x}):=F(\bm{x})-F^*$ with $F^*:=\inf\left\{F(\bm{x}) : \bm{x}\in\mathcal{Q} \right\}>-\infty$ (see Assumption \ref{assumA}4).
\end{lemma}
\begin{proof}
First, using the similar arguments for deducing \eqref{ineq1-gen}, we have that, for any $k\geq0$ and any $\bm{x}\in\mathrm{dom}\,P\cap\mathrm{dom}\,\phi$,
\begin{equation*}
\begin{aligned}
P(\widetilde{\bm{z}}^{k+1})
&\leq P(\bm{x}) - \langle \nabla f(\bm{y}^k), \,\widetilde{\bm{z}}^{k+1}
-\bm{x} \rangle + h_k(\bm{x}),
\end{aligned}
\end{equation*}
where
\begin{equation*}
\begin{aligned}
h_k(\bm{x}) := &\; \lambda\theta_k^{\gamma-1}\mathcal{D}_{\phi}(\bm{x},\,\bm{z}^{k})
- \lambda\theta_k^{\gamma-1}\mathcal{D}_{\phi}(\bm{x},\,\bm{z}^{k+1})
- \lambda\theta_k^{\gamma-1}\mathcal{D}_{\phi}(\widetilde{\bm{z}}^{k+1},\,\bm{z}^{k})
\\
&~~ + \lambda\theta_k^{\gamma-1}\big(\mathcal{D}_{\phi}(\widetilde{\bm{z}}^{k+1},\,\bm{z}^{k+1}) + \lambda^{-1}\theta_k^{1-\gamma}\delta_k\big)
+ |\langle\Delta^k, \,\widetilde{\bm{z}}^{k+1}-\bm{x}\rangle|.
\end{aligned}
\end{equation*}
This implies that
\begin{equation}\label{ineq2-acciBPGM}
\begin{aligned}
&\quad P(\widetilde{\bm{z}}^{k+1}) + f(\bm{y}^k) + \langle \nabla f(\bm{y}^k), \,\widetilde{\bm{z}}^{k+1}-\bm{y}^{k}\rangle \\
&\leq P(\bm{x}) + f(\bm{y}^k)
+ \langle \nabla f(\bm{y}^k), \,\bm{x}-\bm{y}^{k}\rangle
+ h_k(\bm{x})
\;\leq\; P(\bm{x}) + f(\bm{x})
+ h_k(\bm{x}),
\end{aligned}
\end{equation}
where the last inequality follows from the convexity of $f$. Next, we see that
\begin{equation*}
\begin{aligned}
&\quad F(\bm{x}^{k+1})
= P(\bm{x}^{k+1}) + f(\bm{x}^{k+1})
= P\big((1-\theta_k)\bm{x}^k + \theta_k\widetilde{\bm{z}}^{k+1}\big) + f(\bm{x}^{k+1}) \\
&\leq (1-\theta_k)P(\bm{x}^k) + \theta_kP(\widetilde{\bm{z}}^{k+1})
+ f(\bm{y}^k) + \langle \nabla f(\bm{y}^k), \,\bm{x}^{k+1}-\bm{y}^{k}\rangle + \mathcal{D}_{f}(\bm{x}^{k+1},\,\bm{y}^{k}) \\
&= (1-\theta_k)\!\left[ P(\bm{x}^k) + f(\bm{y}^k) + \langle \nabla f(\bm{y}^k), \,\bm{x}^{k}-\bm{y}^{k}\rangle \right]
+ \theta_k\!\left[ P(\widetilde{\bm{z}}^{k+1}) + f(\bm{y}^k) + \langle \nabla f(\bm{y}^k), \,\widetilde{\bm{z}}^{k+1}-\bm{y}^{k}\rangle \right] \\
&\qquad + \mathcal{D}_{f}\big((1-\theta_k)\bm{x}^k + \theta_k\widetilde{\bm{z}}^{k+1},\,(1-\theta_k)\bm{x}^k + \theta_k\bm{z}^k\big) \\
&\leq (1-\theta_k)F(\bm{x}^k)
+ \theta_kF(\bm{x})
+ \lambda\theta_k^{\gamma}\mathcal{D}_{\phi}(\bm{x},\bm{z}^{k})
- \lambda\theta_k^{\gamma}\mathcal{D}_{\phi}(\bm{x},\bm{z}^{k+1}) \\
&\qquad - (\lambda-\tau L)
\theta_k^{\gamma}\mathcal{D}_{\phi}(\widetilde{\bm{z}}^{k+1},\bm{z}^{k})
+ \lambda\theta_k^{\gamma}\big(\mathcal{D}_{\phi}(\widetilde{\bm{z}}^{k+1},\,\bm{z}^{k+1})
+ \lambda^{-1}\theta_k^{1-\gamma}\delta_k\big)
+ \theta_k|\langle\Delta^k, \,\widetilde{\bm{z}}^{k+1}-\bm{x}\rangle|,
\end{aligned}
\end{equation*}
where the first inequality follows from the convexity of $P$ and the definition of $\mathcal{D}_f$, and the last inequality follows from \eqref{ineq-smooth} and \eqref{ineq2-acciBPGM}. Now, in the above inequality, subtracting $F(\bm{x})$ from both sides, dividing both sides by $\theta_k^{\gamma}$ and rearranging the resulting relation, we have that for any $k\geq0$,
\begin{equation}\label{suffdes-iABPGM-tmp}
\begin{aligned}
&\quad \theta_k^{-\gamma}\big(F(\bm{x}^{k+1})-F(\bm{x})\big)
+ \lambda\mathcal{D}_{\phi}(\bm{x},\bm{z}^{k+1})  \\
&\leq \theta_k^{-\gamma}(1-\theta_k)\big(F(\bm{x}^{k})-F(\bm{x})\big)
+ \lambda\mathcal{D}_{\phi}(\bm{x},\bm{z}^{k})
- (\lambda-\tau L)\mathcal{D}_{\phi}(\widetilde{\bm{z}}^{k+1},\bm{z}^{k}) \\
&\qquad
+ \lambda\big(\mathcal{D}_{\phi}(\widetilde{\bm{z}}^{k+1},\,\bm{z}^{k+1}) + \lambda^{-1}\theta_k^{1-\gamma}\delta_k\big)
+ \theta_k^{1-\gamma}|\langle\Delta^k, \,\widetilde{\bm{z}}^{k+1}-\bm{x}\rangle|.
\end{aligned}
\end{equation}
Moreover,
\begin{equation*}
\begin{aligned}
\theta_{k}^{-\gamma}\big(F(\bm{x}^{k+1})-F(\bm{x})\big)
&= \theta_{k+1}^{-\gamma}(1-\theta_{k+1})\big(F(\bm{x}^{k+1})-F(\bm{x})\big)
+ \vartheta_{k+1}\big(F(\bm{x}^{k+1})-F(\bm{x})\big) \\
&\geq
\theta_{k+1}^{-\gamma}(1-\theta_{k+1})\big(F(\bm{x}^{k+1})-F(\bm{x})\big)
+ \vartheta_{k+1}\big(F^*-F(\bm{x})\big),
\end{aligned}
\end{equation*}
where the last inequality follows from $\vartheta_{k+1}\geq0$ (by condition \eqref{condtheta}) and $F(\bm{x}^{k+1}) \geq F^*$ for all $k\geq0$. Then, combining the above relations yields
\begin{equation}\label{suffdes-iABPGM}
\begin{aligned}
&\quad \theta_{k+1}^{-\gamma}(1-\theta_{k+1})\big(F(\bm{x}^{k+1})-F(\bm{x})\big)
+ \lambda\mathcal{D}_{\phi}(\bm{x},\,\bm{z}^{k+1}) \\
&\leq \theta_k^{-\gamma}(1-\theta_k)\big(F(\bm{x}^{k})-F(\bm{x})\big)
+ \lambda\mathcal{D}_{\phi}(\bm{x},\,\bm{z}^{k})
+ \vartheta_{k+1} (F(\bm{x})-F^*)   \\[2pt]
&\quad
- (\lambda-\tau L)\mathcal{D}_{\phi}(\widetilde{\bm{z}}^{k+1},\bm{z}^{k})
+ \lambda\big(\mathcal{D}_{\phi}(\widetilde{\bm{z}}^{k+1},\,\bm{z}^{k+1})
+ \lambda^{-1}\theta_k^{1-\gamma}\delta_k\big)
+ \theta_k^{1-\gamma}|\langle\Delta^k, \,\widetilde{\bm{z}}^{k+1}-\bm{x}\rangle|.
\end{aligned}
\end{equation}
Thus, one can have the following results.
\begin{itemize}
\item When using (AbSC'), we have that
      \begin{equation}\label{addineq-iABPGM-IC1}
      \hspace{-0.7cm}
      \begin{aligned}
      &-(\lambda-\tau L)\mathcal{D}_{\phi}(\widetilde{\bm{z}}^{k+1},\bm{z}^{k})
      + \lambda\big(\mathcal{D}_{\phi}(\widetilde{\bm{z}}^{k+1},\,\bm{z}^{k+1}) + \lambda^{-1}\theta_k^{1-\gamma}\delta_k\big)
      + \theta_k^{1-\gamma}|\langle\Delta^k, \,\widetilde{\bm{z}}^{k+1}-\bm{x}\rangle|  \\
      &\leq \lambda\varepsilon_k + \theta_k^{1-\gamma}|\langle\Delta^k, \,\bm{x}\rangle| + \theta_k^{1-\gamma}|\langle\Delta^k, \,\widetilde{\bm{z}}^{k+1}\rangle|
      \leq \big(\lambda+\|\bm{x}\|\theta_k^{1-\gamma}\big)\varepsilon_k
      + \theta_k^{1-\gamma}|\langle\Delta^k, \,\widetilde{\bm{z}}^{k+1}\rangle|,
      \end{aligned}
      \end{equation}
      which, together with \eqref{suffdes-iABPGM}, gives \eqref{suffdes-iABPGM-IC1}.

\item When using (ReSC') with $0\leq\sigma\leq\frac{\lambda-\tau L}{\lambda}$, we have that
    \begin{equation}\label{addineq-iABPGM-IC2}
    \hspace{-0.7cm}
    \begin{aligned}
    &-(\lambda-\tau L)\mathcal{D}_{\phi}(\widetilde{\bm{z}}^{k+1},\bm{z}^{k})
    + \lambda\big(\mathcal{D}_{\phi}(\widetilde{\bm{z}}^{k+1},\,\bm{z}^{k+1}) + \lambda^{-1}\theta_k^{1-\gamma}\delta_k\big)
    + \theta_k^{1-\gamma}|\langle\Delta^k, \,\widetilde{\bm{z}}^{k+1}-\bm{x}\rangle|  \\
    &\leq - (\lambda-\tau L-\lambda\sigma)\mathcal{D}_{\phi}(\widetilde{\bm{z}}^{k+1},\bm{z}^{k})
    \leq 0,
    \end{aligned}
    \end{equation}
    which, together with \eqref{suffdes-iABPGM}, gives \eqref{suffdes-iABPGM-IC2}.

\end{itemize}
We then complete the proof.
\end{proof}

\begin{theorem}[\textbf{Convergence rate of the v-iBPGM with (AbSC')}]\label{thm-iABPGM-IC1}
Suppose that Assumptions \ref{assumA} and \ref{assumC} hold, and $\lambda \geq \tau L$. Let $\{\bm{x}^k\}$, $\{\bm{z}^k\}$ and $\{\widetilde{\bm{z}}^k\}$ be the sequences generated by the v-iBPGM with (AbSC') in Algorithm \ref{algo-iABPGM}. Then, the following statements hold.
\begin{itemize}
\item[{\rm (i)}] For any $k\geq0$ and any $\bm{x}\in\mathrm{dom}\,P\cap\mathrm{dom}\,\phi$,
    \begin{equation*}
    \begin{aligned}
    &\quad F(\bm{x}^{k+1}) - F(\bm{x})  \\
    &\leq \left(\frac{\alpha-1}{k+\alpha-1}\right)^{\gamma}
    \left(\lambda\mathcal{D}_{\phi}(\bm{x},\,\bm{z}^{0})
    + e(\bm{x}){\textstyle\sum^{k}_{i=0}}\vartheta_i
    +  {\textstyle\sum^{k}_{i=0}}\big(\lambda+\|\bm{x}\|\theta_i^{1-\gamma}\big)\varepsilon_i
    + {\textstyle\sum^{k}_{i=0}}\,\theta_i^{1-\gamma}|\langle\Delta^i, \,\widetilde{\bm{z}}^{i+1}\rangle|\right),
    \end{aligned}
    \end{equation*}
    where $e(\bm{x}):=F(\bm{x})-F^*$ with $F^*:=\inf\left\{F(\bm{x}) : \bm{x}\in\mathcal{Q} \right\}>-\infty$, $\vartheta_{k+1}:=\theta_{k}^{-\gamma}-\theta_{k+1}^{-\gamma}(1-\theta_{k+1})$, and $\alpha$ is as in \eqref{condtheta1}.

\item[{\rm (ii)}] If $k^{-\gamma}\sum^{k-1}_{i=0}\theta_i^{1-\gamma}\varepsilon_i\to0$ and $k^{-\gamma}\sum^{k-1}_{i=0}\,\theta_i^{1-\gamma}|\langle\Delta^i, \,\widetilde{\bm{z}}^{i+1}\rangle|\to0$, then $F(\bm{x}^k) \to F^*$.

\item[{\rm (iii)}] If $\sum\theta_k^{1-\gamma}\varepsilon_k<\infty$, $\sum\theta_k^{1-\gamma}|\langle\Delta^k, \,\widetilde{\bm{z}}^{k+1}\rangle|<\infty$ and problem \eqref{mainpro} has an optimal solution $\bm{x}^*$ such that $\bm{x}^*\in\mathrm{dom}\,\phi$, then
    \begin{equation*}
    F(\bm{x}^k) - F^* \leq \mathcal{O}\left(\frac{1}{k^{\gamma}}\right).
    \end{equation*}
\end{itemize}
\end{theorem}
\begin{proof}
\textit{Statement (i)}.
First, we see from \eqref{suffdes-iABPGM-IC1} that, for any $i\geq0$ and any $\bm{x}\in\mathrm{dom}\,P\cap\mathrm{dom}\,\phi$,
\begin{equation*}
\begin{aligned}
&\quad~ \theta_{i+1}^{-\gamma}(1-\theta_{i+1})\big(F(\bm{x}^{i+1})-F(\bm{x})\big)
+ \lambda\mathcal{D}_{\phi}(\bm{x},\,\bm{z}^{i+1}) \\
&\leq \theta_i^{-\gamma}(1-\theta_i)\big(F(\bm{x}^{i})-F(\bm{x})\big)
+ \lambda\mathcal{D}_{\phi}(\bm{x},\,\bm{z}^{i})
+ e(\bm{x})\vartheta_{i+1}   \\
& \qquad
+ \big(\lambda+\|\bm{x}\|\theta_i^{1-\gamma}\big)\varepsilon_i
+ \theta_i^{1-\gamma}|\langle\Delta^i, \,\widetilde{\bm{z}}^{i+1}\rangle|.
\end{aligned}
\end{equation*}
For any $k\geq1$, summing this inequality from $i=0$ to $i=k-1$ and recalling $\theta_0=1$ results in
\begin{equation*}
\begin{aligned}
&\quad \theta_{k}^{-\gamma}(1-\theta_{k})\big(F(\bm{x}^{k})-F(\bm{x})\big)
+ \lambda\mathcal{D}_{\phi}(\bm{x},\,\bm{z}^{k})   \\
&\leq \lambda\mathcal{D}_{\phi}(\bm{x},\,\bm{z}^{0})
+ e(\bm{x}){\textstyle\sum^{k-1}_{i=0}}\vartheta_{i+1}
+ {\textstyle\sum^{k-1}_{i=0}}\big(\lambda+\|\bm{x}\|\theta_i^{1-\gamma}\big)\varepsilon_i
+ {\textstyle\sum^{k-1}_{i=0}}\theta_i^{1-\gamma}|\langle\Delta^i, \,\widetilde{\bm{z}}^{i+1}\rangle|,
\end{aligned}
\end{equation*}
which, together with \eqref{suffdes-iABPGM-tmp}, \eqref{addineq-iABPGM-IC1} and $\sum^{k-1}_{i=0}\vartheta_{i+1}\leq\sum^{k}_{i=0}\vartheta_{i}$, implies that, for any $k\geq1$,
\begin{equation*}
\begin{aligned}
&\quad \theta_{k}^{-\gamma}\big(F(\bm{x}^{k+1})-F(\bm{x})\big)
+ \lambda\mathcal{D}_{\phi}(\bm{x},\,\bm{z}^{k+1})  \\
&\leq \lambda\mathcal{D}_{\phi}(\bm{x},\,\bm{z}^{0})
+ e(\bm{x}){\textstyle\sum^{k}_{i=0}}\,\vartheta_i
+  {\textstyle\sum^{k}_{i=0}}\big(\lambda+\|\bm{x}\|\theta_i^{1-\gamma}\big)\varepsilon_i
+ {\textstyle\sum^{k}_{i=0}}\,\theta_i^{1-\gamma}|\langle\Delta^i, \,\widetilde{\bm{z}}^{i+1}\rangle|.
\end{aligned}
\end{equation*}
Moreover, it is easy to verify that this inequality also holds for $k=0$ (using \eqref{suffdes-iABPGM-tmp} for $k=0$). Thus, multiplying the above inequality by $\theta_{k}^{\gamma}$ and using condition \eqref{condtheta1}, we can obtain the desired result in statement (i).

\textit{Statement (ii)}.
Note from Lemma \ref{lem-thetasumbd} that $\sum^{k}_{i=0}\vartheta_i \leq 2+\frac{(k+1+\gamma)^{\gamma}}{\gamma}$. Thus, one can verify that, for all $k\geq0$,
\begin{equation*}
\begin{aligned}
&\left(\frac{\alpha-1}{k+\alpha-1}\right)^{\gamma}\sum^{k}_{i=0}\vartheta_i
\leq 2\left(\frac{\alpha-1}{k+\alpha-1}\right)^{\gamma} + \left(\frac{\alpha-1}{k+\alpha-1}\right)^{\gamma}\frac{(k+1+\gamma)^{\gamma}}{\gamma} \\
&\leq 2 + \frac{(\alpha-1)^{\gamma}}{\gamma}\left(\frac{k+1+\gamma}{k+\alpha-1}\right)^{\gamma}
\leq 2 + \frac{\max\big\{(\alpha-1)^{\gamma},\,(1+\gamma)^{\gamma}\big\}}{\gamma},
\end{aligned}
\end{equation*}
where the second inequality follows from $\frac{\alpha-1}{k+\alpha-1}\leq1$ and the last inequality follows from $\frac{k+1+\gamma}{k+\alpha-1}\leq\max\big\{1,\, \frac{1+\gamma}{\alpha-1}\big\}$. Using this fact, $k^{-\gamma}\sum^{k-1}_{i=0}\theta_i^{1-\gamma}\varepsilon_i\to0$ (which implies  $k^{-\gamma}\sum^{k-1}_{i=0}\varepsilon_i\to0$ since $\theta_i^{1-\gamma}\geq1$),  and $k^{-\gamma}\sum^{k-1}_{i=0}\,\theta_i^{1-\gamma}|\langle\Delta^i, \,\widetilde{\bm{z}}^{i+1}\rangle|\to0$, one can see from statement (i) that
\begin{equation*}
\limsup\limits_{k\to\infty}\,F(\bm{x}^{k}) \leq F(\bm{x}) + \left(2+\frac{\max\big\{(\alpha-1)^{\gamma},\,(1+\gamma)^{\gamma}\big\}}{\gamma}\right)e(\bm{x}), \quad \forall\,\bm{x}\in\mathrm{dom}\,P\cap\mathrm{dom}\,\phi.
\end{equation*}
This together with \eqref{infeq} implies that
\begin{equation*}
F^*
\leq\liminf\limits_{k\to\infty}\,F(\bm{x}^{k})
\leq\limsup\limits_{k\to\infty}\,F(\bm{x}^{k})
\leq F^*,
\end{equation*}
from which, we can conclude that $F(\bm{x}^{k}) \to F^*$ and prove statement (ii).

\textit{Statement (iii)}.
When problem \eqref{mainpro} has an optimal solution $\bm{x}^*$ such that $\bm{x}^*\in\mathrm{dom}\,\phi$, we have $F^*=F(\bm{x}^*)=\min\left\{F(\bm{x}) : \bm{x}\in\mathcal{Q} \right\}$. Then, applying statement (i) with $\bm{x}=\bm{x}^*$, we get
\begin{equation*}
\begin{aligned}
&\quad F(\bm{x}^{k+1}) - F^*  \\
&\leq \left(\frac{\alpha-1}{k+\alpha-1}\right)^{\gamma}
\left(\lambda\mathcal{D}_{\phi}(\bm{x}^*,\bm{z}^{0})
+  {\textstyle\sum^{k}_{i=0}}\big(\lambda+\|\bm{x}^*\|\theta_i^{1-\gamma}\big)\varepsilon_i
+ {\textstyle\sum^{k}_{i=0}}\,\theta_i^{1-\gamma}|\langle\Delta^i, \,\widetilde{\bm{z}}^{i+1}\rangle|\right).
\end{aligned}
\end{equation*}
Since $\sum\theta_k^{1-\gamma}\varepsilon_k<\infty$ and $\sum\theta_k^{1-\gamma}|\langle\Delta^k, \,\widetilde{\bm{z}}^{k+1}\rangle|<\infty$, it is easy to see that $F(\bm{x}^k)-F^*\leq\mathcal{O}(k^{-\gamma})$. This completes the proof.
\end{proof}

\begin{remark}[\textbf{Comments on convergence rate of the v-iBPGM with (AbSC')}]\label{rek-comp-viBPGM}
Similar to the discussions in Remark \ref{rek-comp-iBPGM}, the condition that $\sum\theta_k^{1-\gamma}|\langle\Delta^k, \,\widetilde{\bm{z}}^{k+1}\rangle|<\infty$
can be readily guaranteed if no error occurs on the left-hand-side of the optimality condition (namely, $\Delta^k=0$) when applying a certain subsolver for solving the subproblem, as is the case in our experiments. It can also reduce to $\sum\theta_k^{1-\gamma}\varepsilon_k<\infty$ if, for example, $\mathrm{dom}\,P\cap\mathrm{dom}\,\phi$ is bounded, which is the case for the test problem in our numerical experiments. Additionally, one could introduce an arbitrarily summable nonnegative sequence $\{\eta_k\}$, and at each iteration, verify that $\theta_k^{1-\gamma}|\langle\Delta^k, \,\widetilde{\bm{z}}^{k+1}\rangle|\leq\eta_k$ along with (AbSC') to meet the summability requirement.
\end{remark}

\begin{theorem}[\textbf{Convergence rate of the v-iBPGM with (ReSC')}]\label{thm-iABPGM-IC2}
Suppose that Assumptions \ref{assumA} and \ref{assumC} hold, $\lambda \geq \tau L$, and $0\leq\sigma\leq\frac{\lambda-\tau L}{\lambda}$. Let $\{\bm{x}^k\}$, $\{\bm{z}^k\}$ and $\{\widetilde{\bm{z}}^k\}$ be the sequences generated by the v-iBPGM with (ReSC') in Algorithm \ref{algo-iABPGM}. Then, the following statements hold.
\begin{itemize}
\item[{\rm (i)}] For any $k\geq0$ and any $\bm{x}\in\mathrm{dom}\,P\cap\mathrm{dom}\,\phi$,
    \begin{equation*}
    F(\bm{x}^{k+1}) - F(\bm{x})
    \leq \left(\frac{\alpha-1}{k+\alpha-1}\right)^{\gamma}
    \left(\lambda\mathcal{D}_{\phi}(\bm{x},\,\bm{z}^{0})
    + e(\bm{x}){\textstyle\sum^{k}_{i=0}}\vartheta_i\right),
    \end{equation*}
    where $e(\bm{x}):=F(\bm{x})-F^*$ with $F^*:=\inf\left\{F(\bm{x}) : \bm{x}\in\mathcal{Q} \right\}>-\infty$, $\vartheta_{k+1}:=\theta_{k}^{-\gamma}-\theta_{k+1}^{-\gamma}(1-\theta_{k+1})$, and $\alpha$ is as in \eqref{condtheta1}.

\item[{\rm (ii)}] It holds that $F(\bm{x}^k) \to F^*$.

\item[{\rm (iii)}] If problem \eqref{mainpro} has an optimal solution $\bm{x}^*$ such that $\bm{x}^*\in\mathrm{dom}\,\phi$, then
    \begin{equation*}
    F(\bm{x}^k) - F^* \leq \mathcal{O}\left(\frac{1}{k^{\gamma}}\right).
    \end{equation*}
\end{itemize}
\end{theorem}
\begin{proof}
\textit{Statement (i)}.
First, we see from \eqref{suffdes-iABPGM-IC2} that, for any $i\geq0$ and any $\bm{x}\in\mathrm{dom}\,P\cap\mathrm{dom}\,\phi$,
\begin{equation*}
\begin{aligned}
&\quad~ \theta_{i+1}^{-\gamma}(1-\theta_{i+1})\big(F(\bm{x}^{i+1})-F(\bm{x})\big)
+ \lambda\mathcal{D}_{\phi}(\bm{x},\,\bm{z}^{i+1}) \\
&\leq \theta_i^{-\gamma}(1-\theta_i)\big(F(\bm{x}^{i})-F(\bm{x})\big)
+ \lambda\mathcal{D}_{\phi}(\bm{x},\,\bm{z}^{i})
+ e(\bm{x})\vartheta_{i+1}.
\end{aligned}
\end{equation*}
For any $k\geq1$, summing this inequality from $i=0$ to $i=k-1$ and recalling $\theta_0=1$ results in
\begin{equation*}
\theta_{k}^{-\gamma}(1-\theta_{k})\big(F(\bm{x}^{k})-F(\bm{x})\big)
+ \lambda\mathcal{D}_{\phi}(\bm{x},\,\bm{z}^{k})
\leq \lambda\mathcal{D}_{\phi}(\bm{x},\,\bm{z}^{0})
+ e(\bm{x}){\textstyle\sum^{k-1}_{i=0}}\vartheta_{i+1},
\end{equation*}
which, together with \eqref{suffdes-iABPGM-tmp}, \eqref{addineq-iABPGM-IC2} and $\sum^{k-1}_{i=0}\vartheta_{i+1}\leq\sum^{k}_{i=0}\vartheta_{i}$, implies that, for any $k\geq1$,
\begin{equation*}
\theta_{k}^{-\gamma}\big(F(\bm{x}^{k+1})-F(\bm{x})\big)
+ \lambda\mathcal{D}_{\phi}(\bm{x},\,\bm{z}^{k+1})
\leq \lambda\mathcal{D}_{\phi}(\bm{x},\,\bm{z}^{0})
+ e(\bm{x}){\textstyle\sum^{k}_{i=0}}\vartheta_i.
\end{equation*}
Moreover, it is easy to verify that this inequality also holds for $k=0$ (using \eqref{suffdes-iABPGM-tmp} for $k=0$). Thus, multiplying the above inequality by $\theta_{k}^{\gamma}$ and using condition \eqref{condtheta1}, we can obtain the desired result in statement (i).

With statement (i), the proofs of statements (ii) and (iii) are similar to those in Theorem \ref{thm-iABPGM-IC1} and thus are omitted to save space.
\end{proof}

One can see from Theorems \ref{thm-iABPGM-IC1} and \ref{thm-iABPGM-IC2} that our v-iBPGM with (AbSC') or (ReSC') achieves a flexible convergence rate of $\mathcal{O}(1/k^{\gamma})$ with $\gamma\geq1$ being a restricted relative smoothness exponent. Thus, when $\gamma>1$, the v-iBPGM indeed improves the $\mathcal{O}(1/k)$ convergence rate of the iBPGM at both the averaged iterate and the last iterate. In particular, when $\gamma=2$, $\tau=1$, and $\mathrm{dom}\,P\cap\mathrm{dom}\,\phi$ is bounded, along with the fact $\theta_k=\mathcal{O}(1/k)$ (see condition \eqref{condtheta1} and Lemma \ref{lem-thetasumbd}), the summable-error conditions required by the v-iBPGM with (AbSC') in Theorem \ref{thm-iABPGM-IC1} readily boil down to $\sum k\varepsilon_k<\infty$, which is the same as the one required by the iBPGM with (AbSC) in Theorem \ref{thm-fval-iBPGM-IC1}. Moreover, in this case, the tolerance condition $0\leq\sigma\leq\frac{\lambda-L}{\lambda}$ required by the v-iBPGM with (ReSC') in Theorem \ref{thm-iABPGM-IC2} is also looser than the condition $0<\sigma\leq\frac{\lambda-L}{2\lambda}$ by the iBPGM with (ReSC) in Theorem \ref{thm-fval-iBPGM-IC2}. Thus, it is interesting to see that the v-iBPGM appears to achieve a faster convergence rate at the last iterate under a same or even weaker level of error controls as the iBPGM. But we should be mindful that, since $\theta_k$ is contained as part of the proximal parameter in the subproblem \eqref{subpro-iABPGM} of the v-iBPGM, and it goes to zero eventually, we may have a more difficult subproblem to solve when $\theta_k$ becomes very small. This is somewhat the price we have to pay to obtain a theoretically faster convergence rate, and hence the true practical performance is indeed problem-dependent, as shown in our numerical section.

Before closing this section, we give the subsequential convergence result for the v-iBPGM. Note that establishing the global sequential convergence for the inertial variant could be more challenging, even when the subproblem is solved \textit{exactly}. Inspired by recent works \cite{ahgp2021block,mops2020convex} on inertial variants of the BPGM, it might be possible to establish similar global convergence results for our v-iBPGM under the widely-used Kurdyka-{\L}ojasiewicz (KL) property (see, e.g, \cite{ab2009on,abrs2010proximal,abs2013convergence,bdl2007the,bst2014proximal} for more details) together with additional assumptions on the kernel function $\phi$. We will leave this topic for future research.

\begin{theorem}[\textbf{Subsequential convergence of the v-iBPGM}]\label{thm-conv-v-iBPGM}
Suppose that Assumptions \ref{assumA} and \ref{assumC} hold. Moreover, for (AbSC'), suppose that $\lambda \geq \tau L$, $k^{-\gamma}\sum^{k-1}_{i=0}\theta_i^{1-\gamma}\varepsilon_i\to0$ and $k^{-\gamma}\sum^{k-1}_{i=0}\,\theta_i^{1-\gamma}|\langle\Delta^i, \,\widetilde{\bm{z}}^{i+1}\rangle|\to0$; for (ReSC'), suppose that $\lambda \geq \tau L$ and $0\leq\sigma\leq\frac{\lambda-\tau L}{\lambda}$. Let $\{\bm{x}^k\}$ and $\{\widetilde{\bm{z}}^k\}$ be the sequences generated by the v-iBPGM in Algorithm \ref{algo-iABPGM}. If the optimal solution set of problem \eqref{mainpro} is nonempty and $\{\bm{x}^k\}$ is bounded, then any cluster point of $\{\bm{x}^k\}$ is an optimal solution of problem \eqref{mainpro}.
\end{theorem}
\begin{proof}
If the optimal solution set of problem \eqref{mainpro} is nonempty and $\{\bm{x}^k\}$ is bounded, we have that $F^*=\min\left\{F(\bm{x}) : \bm{x}\in\mathcal{Q} \right\}$ and $\{\bm{x}^k\}$ has at least one cluster point. Suppose that $\bm{x}^{\infty}$ is a cluster point and $\{\bm{x}^{k_i}\}$ is a convergent subsequence such that $\lim\limits_{i\to\infty}\bm{x}^{k_i} = \bm{x}^{\infty}$. Then, from Theorem \ref{thm-iABPGM-IC1}(ii) or Theorem \ref{thm-iABPGM-IC2}(ii) together with the lower semicontinuity of $F$ (since $P$ and $f$ are closed by Assumptions \ref{assumA}2\&3), we see that
\begin{equation*}
\min\left\{F(\bm{x}) : \bm{x}\in\mathcal{Q} \right\} = F^* = \lim\limits_{k\to\infty}\,F(\bm{x}^{k}) \geq F(\bm{x}^{\infty}).
\end{equation*}
This implies that $F(\bm{x}^{\infty})$ is finite and hence  $\bm{x}^{\infty}\in\mathrm{dom}\,F$. Moreover, since $\bm{x}^k\in\mathrm{dom}\,P\cap\mathrm{dom}\,\phi\subseteq\mathrm{dom}\,\phi$, then $\bm{x}^{\infty}\in\overline{\mathrm{dom}}\,\phi=\mathcal{Q}$. Therefore, $\bm{x}^{\infty}$ is an optimal solution of problem \eqref{mainpro}. This completes the proof.
\end{proof}

\begin{remark}[\textbf{Comments on the boundedness of iterates}]\label{rek-bdness-v-iBPGM}
In Theorem \ref{thm-conv-v-iBPGM}, we assume the boundedness of $\{\bm{x}^k\}$ to ensure the existence of a cluster point. Here, we comment on sufficient conditions that can guarantee this boundedness assumption. Indeed, this can be easily satisfied when $\mathrm{dom}\,P\cap\mathrm{dom}\,\phi$ is bounded, which is the case for the test problems in our numerical experiments. Moreover, from Theorems \ref{thm-iABPGM-IC1}(iii) and \ref{thm-iABPGM-IC2}(iii), we see that the sequence of the objective value $\{F(\bm{x}^k)\}$ achieves a convergence rate of $\mathcal{O}\left(1/k^{\gamma}\right)$ under appropriate conditions. Thus, when $F$ is additionally level-bounded, the boundedness of $\{\bm{x}^k\}$ can be readily guaranteed.
\end{remark}

\section{Numerical experiments}\label{secnum}

In this section, we conduct some numerical experiments to test our iBPGM and v-iBPGM with the absolute-type and relative-type stopping criteria for solving the discrete quadratic regularized optimal transport (QROT) problem. It is worth noting that the goal of our experiments here is not to develop a state-of-the-art method for the QROT problem, but rather to preliminarily demonstrate the influence of different inexact stopping criteria on the convergence behaviors of different methods, and thus providing some insight beyond the theoretical results and validating the necessity of developing practical inexact versions. All experiments in this section are run in {\sc Matlab} R2024a on a PC with Intel Xeon processor E-2176G@3.70GHz (with 6 cores and 12 threads) and 64GB of RAM, equipped with a Windows OS.

The QROT problem, as an important variant of the classic optimal transport problem, has attracted particular attention in recent years; see, e.g., \cite{blondel2018smooth,essid2018quadratically,lorenz2021quadratically} for more details. Mathematically, the discrete QROT problem is given as follows:
\begin{equation}\label{QROTproblem}
\begin{aligned}
&\min\limits_{X}~\langle C, \,X\rangle + \frac{\nu}{2}\|X\|_F^2  \\
&~\mathrm{s.t.} \quad X\in\Omega := \big\{X\in\mathbb{R}^{m \times n}: X\bm{e}_n = \bm{a}, ~X^{\top}\bm{e}_m = \bm{b}, ~X\geq0\big\},
\end{aligned}
\end{equation}
where $\nu>0$ is a given regularization parameter, $C\in\mathbb{R}^{m \times n}_+$ is a given cost matrix, $\bm{a}:=(a_1,\cdots,a_m)^{\top}\in\Sigma_m$ and  $\bm{b}:=(b_1,\cdots,b_n)^{\top}\in\Sigma_n$ are given probability vectors with $\Sigma_{m}$ (resp., $\Sigma_{n}$) denoting the $m$ (resp., $n$)-dimensional unit simplex, and $\bm{e}_{m}$ (resp., $\bm{e}_{n}$) denotes the $m$ (resp., $n$)-dimensional vector of all ones. It is obvious that problem \eqref{QROTproblem} falls into the form of \eqref{mainpro} via some simple reformulations, and thus our iBPGM and v-iBPGM are applicable.

One can also show that the dual problem of \eqref{QROTproblem} is given by
\begin{equation}\label{dualQROT}
\max\limits_{\bm{f},\,\bm{g}}~
- \frac{1}{2\nu}\left\|\left(\bm{f}\bm{e}_{n}^{\top}
+ \bm{e}_{m}\bm{g}^{\top}-C\right)_+\right\|_F^2
+ \bm{a}^{\top}\bm{f}
+ \bm{b}^{\top}\bm{g},
\end{equation}
and the Karush-Kuhn-Tucker (KKT) system for \eqref{QROTproblem} and \eqref{dualQROT} is given by
\begin{equation}\label{kkt-QROT}
\begin{array}{l}
X\bm{e}_{n} = \bm{a}, ~~
X^{\top}\bm{e}_{m} = \bm{b}, ~~
\langle X, \,\mathcal{Z}(X,\bm{f},\bm{g})\rangle = 0, ~~X\geq0,
~~\mathcal{Z}(X,\bm{f},\bm{g})\geq0,
\end{array}
\end{equation}
where $\bm{f}\in\mathbb{R}^m$ and $\bm{g}\in\mathbb{R}^n$ are the Lagrangian multipliers (\textit{or} dual variables), $\mathcal{Z}(X,\bm{f},\bm{g}) := C + \nu X - \bm{f}\bm{e}_{n}^{\top} - \bm{e}_{m}\bm{g}^{\top}$, and $(\cdot)_+$ denotes the projection operator onto $\mathbb{R}^{m\times n}_+$. Moreover, $(X,\bm{f},\bm{g})$ satisfies the KKT system \eqref{kkt-QROT} if and only if $X$ solves \eqref{QROTproblem} and $(\bm{f}, \bm{g})$ solves \eqref{dualQROT}, respectively. Based on \eqref{kkt-QROT}, we define the relative KKT residual for any $(X,\,\bm{f},\,\bm{g})$ as follows:
\begin{equation*}
\Delta_{\rm kkt}(X,\bm{f},\bm{g})
:= \max\left\{\Delta_p(X), \,\Delta_d(X,\bm{f},\bm{g}), \,\Delta_c(X,\bm{f},\bm{g})\right\},
\end{equation*}
where $\Delta_p(X):=\max\left\{\frac{\|X\bm{e}_{n}-\bm{a}\|}{1+\|\bm{a}\|}, \,\frac{\|X^{\top}\bm{e}_{m}-\bm{b}\|}{1+\|\bm{b}\|},
\,\frac{\|\min\{X, \,0\}\|_F}{1+\|X\|_F}\right\}$, $\Delta_d(X,\bm{f},\bm{g}):= \frac{\|\min\{\mathcal{Z}(X,\bm{f},\bm{g}), \,0\}\|_F}{1+\|C\|_F}$, and $\Delta_c(X,\bm{f},\bm{g}):=\frac{\left|\langle X, \,\mathcal{Z}(X,\bm{f},\bm{g})\rangle\right|}{1+\|C\|_F}$. Moreover, we define the relative duality gap as follows:
\begin{equation*}
\Delta_{\rm gap}(X, \bm{f}, \bm{g}) := \cfrac{\left|\mathtt{pobj}(X) - \mathtt{dobj}(\bm{f},\bm{g})\right|}{1 + \left|\mathtt{pobj}(X)\right| + \left|\mathtt{dobj}(\bm{f},\bm{g})\right|},
\end{equation*}
where $\mathtt{pobj}(X) := \langle C, \,X\rangle
+ \frac{\nu}{2}\|X\|_F^2$ and $\mathtt{dobj}(\bm{u},\bm{v})
:= -\frac{1}{2\nu}\big\|\big(\bm{f}\bm{e}_{n}^{\top}
+ \bm{e}_{m}\bm{g}^{\top}-C\big)_+\big\|_F^2
+ \bm{a}^{\top}\bm{f}
+ \bm{b}^{\top}\bm{g}$. These relative residuals will be used later in our termination criterion.

\subsection{Implementation details}

To apply the iBPGM and v-iBPGM, we equivalently reformulate problem \eqref{QROTproblem} as
\begin{equation}\label{QROTproblem-ent}
\min\limits_{X}~\delta_{\Omega^{\circ}}(X) + \langle C, \,X\rangle + \frac{\nu}{2}\|X\|_F^2 \quad \mathrm{s.t.} \quad X\geq0,
\end{equation}
where $\Omega^{\circ}:=\big\{X\in\mathbb{R}^{m \times n}: X\bm{e}_n=\bm{a}, \,X^{\top}\bm{e}_m=\bm{b}\big\}$ is an affine space. Clearly, this problem takes the form of \eqref{mainpro} with $P(X):=\delta_{\Omega^{\circ}}(X)$, $f(X):=\langle C, \,X\rangle + \frac{\nu}{2}\|X\|_F^2$ and $\mathcal{Q}:=\mathbb{R}^{m \times n}_{+}$. Then, we consider the entropy kernel function $\phi(X)=\sum_{ij}x_{ij}(\log x_{ij}-1)$ whose domain is $\mathbb{R}^{m \times n}_{+}$ and choose $\mathcal{X}=[0,1]^{m\times n}$ which contains $\mathrm{dom}\,P\cap\mathrm{dom}\,\phi$. One can verify that $\phi$ is 1-strongly convex on $[0,1]^{m\times n}$, because $\phi$ is entrywise separable and the function $t\mapsto t(\log t-1)$ is 1-strongly convex on $[0,1]$. This, together with the Lipschitz continuity of $\nabla f$, implies that $f$ is $L$-smooth relative to $\phi$ restricted on $\mathcal{X}$ with $L=\nu$. Using these facts, one can further see from Example \ref{example1} in Section \ref{sec-v-iBPGM} that $(f, \,\phi)$ satisfies Assumption \ref{assumC} with $\tau=1$ and $\gamma=2$. Thus, we can readily apply our iBPGM and v-iBPGM with the entropy kernel function to solve problem \eqref{QROTproblem-ent} (and hence problem \eqref{QROTproblem}).

The subproblem at the $k$-th iteration ($k\geq0$) takes the following form
\begin{equation*}
\min_{X}~\delta_{\Omega^{\circ}}(X) + \langle D^k, \,X\rangle + \mu_k\mathcal{D}_{\phi}(X, \,S^k),
\end{equation*}
where $D^k\in\mathbb{R}^{m \times n}$, $S^k\in\mathbb{R}^{m \times n}$, and $\mu_k>0$ are specified in \eqref{subpro-iBPGM} or \eqref{subpro-iABPGM}. This problem is further equivalent to
\begin{equation}\label{proreform2-subpro}
\min\limits_{X}~ \langle M^k, \,X\rangle + \mu_k\,{\textstyle\sum_{ij}}x_{ij}(\log x_{ij}-1), \quad \mathrm{s.t.} \quad X\bm{e}_n=\bm{a}, ~~ X^{\top}\bm{e}_m = \bm{b},
\end{equation}
where $M^k:=D^k-\mu_k \log S^k$. Problem \eqref{proreform2-subpro} has the same form as the entropic regularized optimal transport problem and hence can be readily solved by the popular Sinkhorn's algorithm; see
\cite[Section 4.2]{pc2019computational} for more details. Specifically, let $\Xi^k:=e^{-M^k/\mu_k}$. Then, given an arbitrary initial positive vector $\bm{v}^{k,0}$, the iterative scheme is given by
\begin{equation}\label{sinkalg}
\bm{u}^{k,t} = \bm{a} ./ \big(\Xi^k\bm{v}^{k,t-1}\big), \quad
\bm{v}^{k,t} = \bm{b} ./ \big((\Xi^k)^{\top}\bm{u}^{k,t}\big),
\end{equation}
where `$./$' denotes the entrywise division between two vectors. When a pair $(\bm{u}^{k,t}, \,\bm{v}^{k,t})$ is obtained based on a certain stopping criterion, an approximate solution of \eqref{proreform2-subpro} can be recovered by setting $X^{k,t}:= \mathrm{Diag}(\bm{u}^{k,t})\,\Xi^k\,\mathrm{Diag}(\bm{v}^{k,t})$. Meanwhile, a pair of approximate dual solutions can be recovered by setting $\bm{f}^{k,t}:=\mu_k\log\bm{u}^{k,t}$ and $\bm{g}^{k,t}:=\mu_k\log\bm{v}^{k,t}$. Note that $X^{k,t}$ is in general not exactly \textit{feasible} for \eqref{QROTproblem}, namely, $X^{k,t}\notin\Omega$. Thus, a proper projection \textit{or} rounding procedure is further needed for the verification of the inexact condition. Let $\mathcal{G}_{\Omega}$ be a rounding procedure given in \cite[Algorithm 2]{awr2017near} so that $\widetilde{X}^{k,t} := \mathcal{G}_{\Omega}(X^{k,t})\in\Omega$. Then, similar to the manipulations in \textbf{Example 2} in subsection \ref{sec:examples}, we can show that for any $Y\in\Omega^{\circ}$,
\begin{equation*}
\begin{aligned}
&\quad \langle -D^k-\mu_k\,(\log X^{k,t} - \log S^k), \,Y-\widetilde{X}^{k,t}\rangle
= \langle - M^k - \mu_k \log X^{k,t}, \,Y-\widetilde{X}^{k,t}\rangle \\
&= \langle - M^k - \mu_k \log (\mathrm{Diag}(\bm{u}^{k,t})\,\Xi^k\,\mathrm{Diag}(\bm{v}^{k,t})), \,Y-\widetilde{X}^{k,t}\rangle \\
&= \langle -\mu_k\,(\log \bm{u}^{k,t})\,\bm{e}_n^{\top} - \mu_k\,\bm{e}_m\,(\log \bm{v}^{k,t})^{\top}, \,Y-\widetilde{X}^{k,t}\rangle \\
&= - \mu_k\,\langle\,\log \bm{u}^{k,t}, \,Y\bm{e}_n-\widetilde{X}^{k,t}\bm{e}_n\,\rangle
- \mu_k\,\langle\,\log \bm{v}^{k,t}, \, Y^{\top}\bm{e}_m - (\widetilde{X}^{k,t})^{\top}\bm{e}_m\,\rangle
= 0,
\end{aligned}
\end{equation*}
where the last equality follows from $Y\bm{e}_n=\bm{a}=\widetilde{X}^{k,t}\bm{e}_n$ and $Y^{\top}\bm{e}_m=\bm{b}=(\widetilde{X}^{k,t})^{\top}\bm{e}_m$. This relation implies that
\begin{equation}\label{OTincl-ent}
0\in\partial\delta_{\Omega^{\circ}}(\widetilde{X}^{k,t}) + D^k + \mu_k\big(\log X^{k,t} - \log S^k\big).
\end{equation}
From this relation, we see that our inexact condition \eqref{inexcond-iBPGM} or \eqref{inexcond-iABPGM} with the absolute-type or relative-type stopping criterion is verifiable at the pair $(X^{k,t}, \,\widetilde{X}^{k,t})$, and can be satisfied when the Bregman distance $\mathcal{D}_{\phi}(\widetilde{X}^{k,t}, \,X^{k,t})$ is smaller than a given absolute-type or relative-type tolerance as shown in \eqref{stopcond-ab} or \eqref{stopcond-re}, respectively.

\begin{remark}[\textbf{Comments on errors incurred in computation}]\label{rek-err}
Here, we remind the reader that, as discussed in deriving \eqref{OTincl-ent}, when applying Sinkhorn's algorithm for solving problem \eqref{proreform2-subpro}, no error occurs on the left-hand-side of the optimality condition
\eqref{inexcond-iBPGM} or \eqref{inexcond-iABPGM}
(i.e., $\Delta^k=0$) and the computation of $\partial \delta_{\Omega^{\circ}}$ (i.e., $\delta_k=0$). However, this outcome does depend on the specific problem and the subsolver used to solve the subproblem. In fact, different subsolvers may incur errors at different parts of the optimality condition of the subproblem. We refer the reader to an example in \cite[Section 3.2]{clty2020an} where the $\delta$-subdifferential is incurred when employing a block coordinate descent method to solve the subproblem in the entropic proximal point algorithm. Therefore, a flexible inexact condition in form of \eqref{inexcond-iBPGM} or \eqref{inexcond-iABPGM} that allows different type of errors is useful and necessary to fit different situations in practice.
\end{remark}

In view of the above, we employ Sinkhorn's algorithm as a subroutine in our iBPGM and v-iBPGM, and as guided by our inexact stopping criteria, at the $k$-th iteration ($k\geq0$), we will terminate Sinkhorn's algorithm when the following absolute-type or relative-type stopping criterion is satisfied.
\begin{itemize}
\item The absolute-type criterion:
      \begin{equation}\label{stopcond-ab}
      \mathcal{D}_{\phi}(\widetilde{X}^{k,t}, \,X^{k,t}) \leq \max\left\{\frac{\Upsilon}{(k+1)^p}, \,10^{-10}\right\}.
      \end{equation}
      Here, the coefficient $\Upsilon$ controls the initial accuracy for solving the subproblem and, together with $p$, would determine the tightness of the tolerance requirement. Generally, a small $\Upsilon$ or a large $p$ would give a tight tolerance control, which may cause excessive cost of solving each subproblem, while a large $\Upsilon$ or a small $p$ would give a loose tolerance control, which may greatly increase the number of outer iterations. Therefore, there is a tradeoff on the choices of $\Upsilon>0$ and $p>0$. In our experiments, we will choose $\Upsilon\in\{10, \,1, \,0.1, \,0.01\}$ and $p\in\{1.1, \,2.1, \,3.1\}$ for comparisons.

\item The relative-type criterion:
      \begin{equation}\label{stopcond-re}
      \mathcal{D}_{\phi}(\widetilde{X}^{k,t}, \,X^{k,t}) \leq \sigma\mathcal{D}_{\phi}(\widetilde{X}^{k,t}, \,X^{k}),
      \end{equation}
      where $0<\sigma<1$ is the tolerance parameter and $X^k$ is the iterate obtained from the previous iteration. Note from Theorems \ref{thm-fval-iBPGM-IC2} and \ref{thm-iABPGM-IC2} that the tolerance conditions $0<\sigma\leq\frac{\lambda-L}{2\lambda}$ and $0\leq\sigma\leq\frac{\lambda-\tau L}{\lambda}$ are needed for the iBPGM and v-iBPGM respectively to guarantee the corresponding convergence rates. However, as observed from our experiments, such theoretical requirements are usually too conservative to achieve good empirical performances. In our experiments, we will choose $\sigma\in\{0.999, \,0.99, \,0.9, \,0.7, \,0.5, \,0.3, \,0.1\}$ for comparisons.

\end{itemize}

For ease of future reference, in the following, we use  iBPGM-Ab/v-iBPGM-Ab and iBPGM-Re/v-iBPGM-Re to denote the iBPGM and v-iBPGM with the absolute-type and relative-type stopping criteria, respectively. In all methods, we choose $\lambda=2\nu$. Moreover, in both v-iBPGM-Ab and v-iBPGM-Re, we choose $\theta_k=\frac{\alpha-1}{k+\alpha-1}$ with $\alpha=5$ for all $k\geq1$. We initialize all methods with $X^0:=\bm{a}\bm{b}^{\top}$ and perform a \textit{warm-start} strategy to initialize Sinkhorn's algorithm. Specifically, at each iteration, we initialize Sinkhorn's algorithm with $\bm{v}^{k,0}=\bm{v}^{k-1,t}$, where $\bm{v}^{k-1,t}$ is the iterate obtained by Sinkhorn's algorithm in the previous outer iteration.

We follow \cite[Section 6.1]{yt2020bregman} to randomly generate the simulated data. Specifically, we first generate two discrete probability distributions $\big\{ (a_i, \,\bm{p}_i)\in \mathbb{R}_+\times\mathbb{R}^3 : i = 1,\cdots,m \big\}$ and $\big\{ (b_j, \,\bm{q}_j)\in \mathbb{R}_+\times\mathbb{R}^3 : j = 1,\cdots,n
\big\}$. Here, $\bm{a}:=(a_1, \cdots\!, a_{m})^{\top}$ and $\bm{b}:=(b_1, \cdots\!, b_{n})^{\top}$ are probabilities/weights, which are generated from the uniform distribution on the open interval $(0,\,1)$ and further normalized such that $\sum^{m}_ia_i=\sum^{n}_jb_j=1$. Moreover, $\{\bm{p}_i\}$ and $\{\bm{q}_j\}$ are support points whose entries are drawn from a five-component multivariate Gaussian mixture distribution, with a mean vector $(-20,10,0,10,20)^{\top}$ and a variance vector $(5,5,5,5,5)^{\top}$, using randomly generated weights. Then, the cost matrix $C$ is generated by $c_{ij}=\|\bm{p}_i-\bm{q}_j\|^2$ for $1\leq i\leq m$ and $1\leq j\leq n$ and normalized by dividing (element-wise) by its maximal entry.

Since problem \eqref{QROTproblem} is a quadratic programming (QP) problem, we also apply Gurobi (with default settings) to solve it. It is well known that Gurobi is a powerful commercial solver for QPs and is able to obtain high quality solutions. Therefore, we will use the objective function value obtained by Gurobi as the benchmark to evaluate the qualities of the solutions obtained by our methods.

\subsection{Comparison results}

In the following comparisons, we choose $m=n=200$, and generate 10 test instances with different random seeds. We will terminate the iBPGM and v-iBPGM when
\begin{equation*}
\max\left\{\Delta_{\rm kkt}(X^k,\bm{f}^k,\bm{g}^k), \,\Delta_{\rm gap}(X^k,\bm{f}^k,\bm{g}^k) \right\}
< 10^{-5},
\end{equation*}
where $X^{k}$ and $(\bm{f}^{k}, \bm{g}^{k})$ are approximate optimal solutions of problem \eqref{QROTproblem} and its corresponding dual problem, respectively. Both can be recovered from the Sinkhorn iterate. Moreover, we also terminate all methods when the number of total Sinkhorn iterations exceeds $10^5$.

Tables \ref{ResTable-nu1} and \ref{ResTable-nu001} show the average numerical performance of the iBPGM and v-iBPGM under different inexact stopping criteria for $\nu=1$ and $\nu=0.01$, respectively. In each table, we report ``\texttt{nobj}", which denotes the normalized objective function value $\big|\langle C, \,\mathcal{G}_{\Omega}(X^{k})\rangle + \frac{\nu}{2}\|\mathcal{G}_{\Omega}(X^{k})\|_F^2-f^*\big|\,/\,|f^*|$ with $f^*$ being the highly accurate optimal function value computed by Gurobi, the number of outer iterations (denoted by ``\texttt{out}\#"), the number of total Sinkhorn iterations (denoted by ``\texttt{sink}\#"), and the computational time in seconds (denoted by ``\texttt{time}"). From the results, we have several observations as follows.

\begin{itemize}[leftmargin=0.5cm]
\item \textbf{(Overall performance vs Inexactness)} The overall performance (namely, ``\texttt{nobj}" vs ``\texttt{sink}\#") of each method is unsurprisingly affected by the choice of $(\Upsilon,p)$ in \eqref{stopcond-ab} or $\sigma$ in \eqref{stopcond-re}, which determines the tightness of the inexact tolerance requirement. Generally, a larger $\Upsilon$ with a smaller $p$ or a larger $\sigma$ would lead to a larger number of outer iterations, since the stopping criterion \eqref{stopcond-ab} or \eqref{stopcond-re} is easier to satisfy. However, it does not always yield good overall performances. For example, when $\nu=0.01$, both iBPGM and v-iBPGM with $(\Upsilon,p)=(10, 1.1)$ or $\sigma=0.999$ require more outer iterations and Sinkhorn iterations, while giving a worse \texttt{nobj}. This somewhat aligns with our theoretical results, where the convergence (rate) is only guaranteed under certain requirements on the inexact errors, though the current theoretical requirements appear to be conservative. On the other hand, choosing $p=3.1$ (especially with $\Upsilon=0.01$) or $\sigma=0.1$ results in the fastest tolerance decay, but does not provide the best overall performance for any method, mainly due to the excessive cost of solving each subproblem more accurately. This indicates that a tighter tolerance requirement for solving the subproblem does not necessarily result in better overall efficiency and there is a trade-off between the number of outer and inner iterations.

\item \textbf{(Inertial version vs Non-inertial version)} One can also observe that  v-iBPGM-Ab/v-iBPGM-Re can outperform iBPGM-Ab/iBPGM-Re, requiring fewer outer iterations, particularly for a larger $\nu$ (e.g., $\nu=1$). This empirically verifies the improved convergence rate of v-iBPGM-Ab/v-iBPGM-Re under similar tolerance requirements for iBPGM-Ab/iBPGM-Re, as expected from Theorems \ref{thm-iABPGM-IC1} and \ref{thm-iABPGM-IC2} with $\gamma=2$ for the QROT problem. However, the improvement becomes less significant as $\nu$ decreases. This is likely because, when $\nu$ is smaller, along with the gradual decrease of $\theta_k$, the proximal parameter $\lambda\theta_k$ (with $\lambda=2\nu$ in our setting) in \eqref{subpro-iABPGM} becomes very small, making the subproblem more challenging to solve using Sinkhorn's algorithm and thus necessitating more Sinkhorn iterations for solving each subproblem. This implies that the true performance of the inertial variant may depend on the difficulty of the subproblem and the efficiency of the subsolver for solving the subproblem in different scenarios.

\item \textbf{(Absolute criterion vs Relative criterion)} Finally, one can see that, with proper choices of tolerance parameters,  iBPGM-Ab (resp., v-iBPGM-Ab) and  iBPGM-Re (resp., v-iBPGM-Re) can be comparable in terms of the total number of Sinkhorn iterations. This is indeed reasonable because  iBPGM-Ab (resp., v-iBPGM-Ab) and  iBPGM-Re (resp., v-iBPGM-Re) essentially use the same iBPGM (resp., v-iBPGM) framework but with different stopping criteria for solving the subproblems. Note that  iBPGM-Re and v-iBPGM-Re involve only a \textit{single} tolerance parameter $\sigma$. Therefore, they are typically more user-friendly for parameter tuning compared to  iBPGM-Ab and v-iBPGM-Ab, which require a summable error-tolerance sequence $\{\varepsilon_k\}$ for achieving good practical efficiency. Specifically, the absolute-type criterion in \eqref{stopcond-ab} requires the careful tuning of the two parameters $\Upsilon$ and $p$ for good performance. However, we should also note that  iBPGM-Re and v-iBPGM-Re incur additional overhead (e.g., the computation of $\mathcal{D}_{\phi}(\widetilde{X}^{k,t}, \,X^{k})$ on the right-hand side of \eqref{stopcond-re}), when checking the relative-type stopping criterion. This is why  iBPGM-Re and v-iBPGM-Re may take more computational time than  iBPGM-Ab and v-iBPGM-Ab, despite using a similar number of total Sinkhorn iterations.

\end{itemize}

In summary, allowing the subproblem to be solved approximately is both necessary and important for the BPGM and its inertial variant to be practical and implementable. Different methods with different types of stopping criteria have different inherent inexactness tolerance requirements. The study of such phenomena can deepen our understanding of these methods under different inexact stopping criteria, which can further guide us to choose a proper method with a suitable inexact stopping criterion in practical implementations.

\begin{table}[ht]
\setlength{\belowcaptionskip}{7pt}
\renewcommand\arraystretch{1.15}
\caption{Comparisons between iBPGM and v-iBPGM under different choices of tolerance parameters for $\nu=1$. In the table, ``\texttt{nobj}" denotes the normalized objective function value, ``$\texttt{out}\#$" denotes the number of outer iterations, ``$\texttt{sink}\#$" denotes the the number of Sinkhorn iterations, and ``--" means that the number of total Sinkhorn iterations reaches $10^5$.}\label{ResTable-nu1}
\centering \tabcolsep 10pt
\small{
\begin{tabular}{lllllllll}
\toprule
&\multicolumn{4}{c}{iBPGM-Ab}&\multicolumn{4}{c}{v-iBPGM-Ab}   \\
\cmidrule(l){1-1} \cmidrule(l){2-5} \cmidrule(l){6-9}
$(\Upsilon,p)$
& \texttt{nobj} & \texttt{out}\# & \texttt{sink}\# & \texttt{time}
& \texttt{nobj} & \texttt{out}\# & \texttt{sink}\# & \texttt{time}  \\
\cmidrule(l){1-1} \cmidrule(l){2-5} \cmidrule(l){6-9}
(10, \,1.1) & 2.14e-04 & 6342 & 12685 & 6.54 & 4.93e-04 & 337 & 674 & 0.35 \\
(10, \,2.1) & 2.14e-04 & 6342 & 12685 & 6.45 & 4.93e-04 & 337 & 674 & 0.35 \\
(10, \,3.1) & 2.86e-03 & 2348 & -- & 20.50 & 4.97e-04 & 334 & 962 & 0.41 \\
(1, \,1.1) & 2.14e-04 & 6342 & 12685 & 6.45 & 4.93e-04 & 337 & 674 & 0.34 \\
(1, \,2.1) & 2.14e-04 & 6342 & 14214 & 6.85 & 4.96e-04 & 328 & 2946 & 0.87 \\
(1, \,3.1) & 9.47e-03 & 1479 & -- & 20.38 & 4.92e-04 & 326 & 60299 & 14.13 \\
(0.1, \,1.1) & 2.14e-04 & 6342 & 12685 & 6.46 & 4.97e-04 & 336 & 691 & 0.35 \\
(0.1, \,2.1) & 2.14e-04 & 6327 & 70412 & 21.72 & 4.92e-04 & 327 & 50474 & 11.78 \\
(0.1, \,3.1) & 1.86e-02 & 1172 & -- & 20.38 & 3.49e-03 & 218 & 99836 & 22.22 \\
(0.01, \,1.1) & 2.14e-04 & 6342 & 12685 & 6.45 & 4.93e-04 & 328 & 2888 & 0.87 \\
(0.01, \,2.1) & 2.34e-03 & 2603 & -- & 20.78 & 2.01e-03 & 247 & 97855 & 22.10 \\
(0.01, \,3.1) & 2.65e-02 & 1079 & -- & 20.40 & 4.26e-03 & 211 & -- & 22.52 \\
\bottomrule  \vspace{-2mm} \\
\toprule
&\multicolumn{4}{c}{iBPGM-Re}&\multicolumn{4}{c}{v-iBPGM-Re}   \\
\cmidrule(l){1-1}\cmidrule(l){2-5}\cmidrule(l){6-9}
$\sigma$
& \texttt{nobj} & \texttt{out}\# & \texttt{sink}\# & \texttt{time}
& \texttt{nobj} & \texttt{out}\# & \texttt{sink}\# & \texttt{time}  \\
\cmidrule(l){1-1}\cmidrule(l){2-5}\cmidrule(l){6-9}
0.999 & 2.14e-04 & 6342 & 12685 & 7.76 & 4.93e-04 & 337 & 674 & 0.41 \\
0.99  & 2.14e-04 & 6342 & 12685 & 7.67 & 4.96e-04 & 335 & 942 & 0.49 \\
0.9   & 2.14e-04 & 6342 & 15784 & 9.12 & 4.94e-04 & 329 & 2860 & 1.07 \\
0.7   & 2.14e-04 & 6342 & 36164 & 17.09 & 4.98e-04 & 327 & 6780 & 2.27 \\
0.5   & 2.14e-04 & 6311 & 65559 & 28.29 & 4.95e-04 & 327 & 11545 & 3.69 \\
0.3   & 2.14e-04 & 5817 & 94225 & 37.16 & 4.93e-04 & 327 & 18909 & 5.89 \\
0.1   & 7.71e-04 & 3762 & -- & 32.54 & 4.92e-04 & 327 & 37350 & 11.36 \\
\bottomrule
\end{tabular}}
\end{table}

\begin{table}[ht]
\setlength{\belowcaptionskip}{7pt}
\captionsetup{width=14cm}
\renewcommand\arraystretch{1.15}
\caption{Same as Table \ref{ResTable-nu1} but for $\nu=0.01$.}\label{ResTable-nu001}
\centering \tabcolsep 10pt
\small{
\begin{tabular}{lllllllll}
\toprule
&\multicolumn{4}{c}{iBPGM-Ab}&\multicolumn{4}{c}{v-iBPGM-Ab}   \\
\cmidrule(l){1-1} \cmidrule(l){2-5} \cmidrule(l){6-9}
$(\Upsilon,p)$
& \texttt{nobj} & \texttt{out}\# & \texttt{sink}\# & \texttt{time}
& \texttt{nobj} & \texttt{out}\# & \texttt{sink}\# & \texttt{time}  \\
\cmidrule(l){1-1}\cmidrule(l){2-5}\cmidrule(l){6-9}
(10, \,1.1) & 6.61e-04 & 1453 & 2945 & 1.18 & 5.33e-04 & 4349 & 80801 & 47.22 \\
(10, \,2.1) & 6.39e-04 & 143 & 1283 & 0.33 & 5.32e-04 & 387 & 15542 & 11.35 \\
(10, \,3.1) & 3.53e-04 & 137 & 14450 & 3.09 & 5.32e-04 & 148 & 13321 & 9.19 \\
(1, \,1.1) & 6.98e-04 & 419 & 1169 & 0.41 & 5.33e-04 & 755 & 15035 & 10.77 \\
(1, \,2.1) & 3.93e-04 & 138 & 3069 & 0.70 & 5.54e-04 & 53 & 4903 & 1.91 \\
(1, \,3.1) & 3.48e-04 & 135 & 56758 & 11.94 & 5.84e-04 & 40 & 23213 & 5.24 \\
(0.1, \,1.1) & 5.39e-04 & 149 & 851 & 0.24 & 5.31e-04 & 84 & 4426 & 2.27 \\
(0.1, \,2.1) & 3.53e-04 & 137 & 15236 & 3.25 & 5.85e-04 & 38 & 38850 & 8.80 \\
(0.1, \,3.1) & 3.93e-04 & 117 & 97516 & 20.37 & 5.79e-04 & 38 & 78862 & 17.83 \\
(0.01, \,1.1) & 3.96e-04 & 137 & 3564 & 0.81 & 5.85e-04 & 41 & 9191 & 2.15 \\
(0.01, \,2.1) & 3.48e-04 & 134 & 60421 & 12.62 & 5.79e-04 & 38 & 78229 & 17.65 \\
(0.01, \,3.1) & 7.01e-04 & 89 & -- & 20.95 & 5.79e-04 & 38 & 90368 & 20.38 \\
\bottomrule  \vspace{-2mm} \\
\toprule
&\multicolumn{4}{c}{iBPGM-Re}&\multicolumn{4}{c}{v-iBPGM-Re}   \\
\cmidrule(l){1-1} \cmidrule(l){2-5} \cmidrule(l){6-9}
$\sigma$
& \texttt{nobj} & \texttt{out}\# & \texttt{sink}\# & \texttt{time}
& \texttt{nobj} & \texttt{out}\# & \texttt{sink}\# & \texttt{time}  \\
\cmidrule(l){1-1}\cmidrule(l){2-5}\cmidrule(l){6-9}
0.999 & 5.21e-04 & 303 & 7791 & 2.27 & 2.07e-02 & 3955 & 45975 & 32.72 \\
0.99 & 4.77e-04 & 142 & 2038 & 0.67 & 7.84e-02 & 1011 & 31514 & 23.82 \\
0.9 & 3.63e-04 & 135 & 3789 & 1.20 & 5.33e-04 & 60 & 8803 & 5.01 \\
0.7 & 3.55e-04 & 137 & 9990 & 3.07 & 5.33e-04 & 47 & 10265 & 4.36 \\
0.5 & 3.51e-04 & 137 & 16961 & 5.12 & 5.83e-04 & 43 & 11491 & 3.61 \\
0.3 & 3.49e-04 & 136 & 29350 & 8.78 & 5.84e-04 & 41 & 16765 & 4.95 \\
0.1 & 3.48e-04 & 134 & 59866 & 17.82 & 5.66e-04 & 40 & 31262 & 9.51 \\
\bottomrule
\end{tabular}}
\end{table}

\section{Concluding remarks}\label{seccon}

In this paper, we develop an inexact Bregman proximal gradient  method  (iBPGM)
based on two types of two-point inexact stopping criteria, and establish the $\mathcal{O}(1/k)$ convergence rate as well as the convergence of the sequence under some proper conditions. To improve the convergence speed, we further develop an inertial variant of our iBPGM (denoted by v-iBPGM) and show that it has the $\mathcal{O}(1/k^{\gamma})$ convergence rate, where $\gamma\geq1$ is a restricted relative smoothness exponent. Thus, when $\gamma>1$, the v-iBPGM can improve the $\mathcal{O}(1/k)$ convergence rate of the iBPGM. We also conduct some preliminary numerical experiments for solving the discrete quadratic regularized optimal transport problem to show the convergence behaviors of the iBPGM and v-iBPGM under different inexactness settings.

Note that the current iterative frameworks of our iBPGM and v-iBPGM require the knowledge of the relative smooth constant $L$ to determine the proximal parameter $\lambda$ in the subproblems \eqref{subpro-iBPGM}. However, this may not always be achievable, or the constant $L$ can be very large, leading to poor practical performance. A possible solution is to incorporate an appropriate backtracking or line-search procedure (see, for example, \cite{bt2009a,hrx2018accelerated,s2017variable,y2024proximal}) into our iterative frameworks. However, integrating such a procedure with our inexact stopping criteria would also require additional effort. We will leave it as a future research topic.

\section*{Acknowledgments}

We thank the editor and referees for their valuable suggestions and comments, which have helped to improve the quality of this paper.

\begin{appendices}

\section{Proofs of supporting lemmas}\label{apd-lemmas}

\subsection{Proof of Lemma \ref{lemseqcon2}}

For any $\varepsilon>0$, we have from $\sum^{\infty}_{k=0}\alpha_k<\infty$ that there exists $\widetilde{N}\geq0$ such that $\sum^{\infty}_{k=\widetilde{N}}\alpha_k\leq\varepsilon/2$. Now, let $N_{\varepsilon}:=\max\left\{\widetilde{N}+1, \,2\varepsilon^{-1}\sum^{\widetilde{N}-1}_{i=0}i\alpha_i\right\}$. Then, for any $k\geq N_{\varepsilon}$, we see that
\begin{equation*}
\frac{1}{k}\sum^{k-1}_{i=0}i\alpha_i
= \frac{1}{k}\sum^{\widetilde{N}-1}_{i=0}i\alpha_i
+ \frac{1}{k}\sum^{k-1}_{i=\widetilde{N}}i\alpha_i
\leq \frac{1}{N_{\varepsilon}}\sum^{\widetilde{N}-1}_{i=0}i\alpha_i
+ \sum^{k-1}_{i=\widetilde{N}}\alpha_i \leq \varepsilon.
\end{equation*}
Since $\varepsilon$ is arbitrary, we can conclude that $\frac{1}{k}\sum^{k-1}_{i=0}i\alpha_i\to0$.

\subsection{Proof of Lemma \ref{supplem}}

First, from condition \eqref{inexcond-gen}, there exists $\bm{d} \in \partial_{\delta}P(\widetilde{\bm{x}}^*)$ such that $\bm{d} + \nabla f(\bm{y}) + \lambda(\nabla \phi(\bm{x}^*)-\nabla \phi(\bm{y})) = \Delta$. Then, for any $\bm{x}\in\mathrm{dom}\,P\cap\mathrm{dom}\,\phi$, we see that
\begin{equation*}
P(\bm{x})
\geq P(\widetilde{\bm{x}}^*) + \langle \bm{d}, \,\bm{x} - \widetilde{\bm{x}}^* \rangle - \delta
= P(\widetilde{\bm{x}}^*) + \langle \Delta - \nabla f(\bm{y}) - \lambda(\nabla \phi(\bm{x}^*) - \nabla \phi(\bm{y})), \,\bm{x} - \widetilde{\bm{x}}^*\rangle - \delta,
\end{equation*}
which implies that
\begin{equation}\label{ineq1-1-gen}
P(\widetilde{\bm{x}}^*)
\leq P(\bm{x}) - \langle \nabla f(\bm{y}), \,\widetilde{\bm{x}}^* - \bm{x} \rangle
+ \lambda\langle \nabla \phi(\bm{x}^*)-\nabla \phi(\bm{y}), \,\bm{x} - \widetilde{\bm{x}}^*\rangle + |\langle \Delta,\,\widetilde{\bm{x}}^* - \bm{x} \rangle| + \delta.
\end{equation}
Note from the four-points identity \eqref{fourId} that
\begin{equation}\label{ineq1-2-gen}
\begin{aligned}
\langle \,\nabla\phi(\bm{x}^*) - \nabla\phi(\bm{y}), \,\bm{x}-\widetilde{\bm{x}}^*\,\rangle
=\mathcal{D}_{\phi}(\bm{x},\,\bm{y})
-\mathcal{D}_{\phi}(\bm{x},\,\bm{x}^*) -\mathcal{D}_{\phi}(\widetilde{\bm{x}}^*,\,\bm{y})
+\mathcal{D}_{\phi}(\widetilde{\bm{x}}^*,\,\bm{x}^*).
\end{aligned}
\end{equation}
Thus, combining \eqref{ineq1-1-gen} and \eqref{ineq1-2-gen}, we obtain that
\begin{equation}\label{ineq1-gen}
\begin{aligned}
P(\widetilde{\bm{x}}^*)
&\leq P(\bm{x}) - \langle \nabla f(\bm{y}), \,\widetilde{\bm{x}}^* - \bm{x} \rangle
+ \lambda\mathcal{D}_{\phi}(\bm{x},\,\bm{y})
- \lambda\mathcal{D}_{\phi}(\bm{x},\,\bm{x}^*)
- \lambda\mathcal{D}_{\phi}(\widetilde{\bm{x}}^*,\,\bm{y}) \\
&\qquad + \lambda\mathcal{D}_{\phi}(\widetilde{\bm{x}}^*,\,\bm{x}^*)
+ |\langle \Delta, \,\widetilde{\bm{x}}^* - \bm{x} \rangle| + \delta.
\end{aligned}
\end{equation}
On the other hand, since $f$ is $L$-smooth relative to $\phi$ restricted on $\mathcal{X}$ (by Assumption \ref{assumA}3) and $f$ is convex with $\mathrm{dom}\,f\supseteq\mathrm{dom}\,\phi$, then
\begin{equation*}
\begin{aligned}
f(\widetilde{\bm{x}}^*) &\leq f(\bm{y}) + \langle \nabla f(\bm{y}), \,\widetilde{\bm{x}}^*-\bm{y}\rangle + L\mathcal{D}_{\phi}(\widetilde{\bm{x}}^*,\,\bm{y}), \\
f(\bm{x}) &\geq f(\bm{y}) + \langle \nabla f(\bm{y}), \,\bm{x}-\bm{y}\rangle.
\end{aligned}
\end{equation*}
Summing the above two inequalities, we obtain that
\begin{equation}\label{ineq2-gen}
f(\widetilde{\bm{x}}^*) \leq f(\bm{x}) + \langle \nabla f(\bm{y}), \,\widetilde{\bm{x}}^*-\bm{x}\rangle + L\mathcal{D}_{\phi}(\widetilde{\bm{x}}^*,\,\bm{y}).
\end{equation}
Thus, summing \eqref{ineq1-gen} and \eqref{ineq2-gen}, one can obtain the desired result.

\subsection{Proof of Lemma \ref{lem-thetasumbd}}

We first show by induction that $\theta_k\geq\frac{1}{k+\gamma}$ for all $k\geq0$. This obviously holds for $k=0$ since $\theta_0=1$. Suppose that $\theta_k\geq\frac{1}{k+\gamma}$ holds for some $k\geq0$. Since $0<\theta_k\leq1$ and $\gamma\geq1$, then $(1-\theta_{k})^{\gamma}\leq1-\theta_{k}$. Using this and \eqref{condtheta}, we see that
\begin{equation*}
\frac{(1-\theta_{k+1})^{\gamma}}{\theta_{k+1}^{\gamma}} \leq \frac{1}{\theta_{k}^{\gamma}}
\quad \Longleftrightarrow \quad
\frac{1-\theta_{k+1}}{\theta_{k+1}} \leq \frac{1}{\theta_{k}},
\end{equation*}
from which, we have that $\theta_{k+1}\geq\frac{\theta_k}{\theta_k+1}=\frac{1}{k+1+\gamma}$. This completes the induction. Using this fact, we further get
\begin{equation*}
\begin{aligned}
&\quad \sum^{k}_{i=0}\frac{1}{\theta_{i-1}^{\gamma}} - \frac{1-\theta_{i}}{\theta_{i}^{\gamma}}
= \frac{1}{\theta_{-1}^{\gamma}} - \frac{1}{\theta_{k}^{\gamma}} + \sum^{k}_{i=0}\frac{1}{\theta_{i}^{\gamma-1}}
\leq 2 + \sum^{k}_{i=1}\frac{1}{\theta_{i}^{\gamma-1}} \\
&\leq 2 + \sum^{k}_{i=1}(i+\gamma)^{\gamma-1}
\leq 2 + \int^{k+1}_{0}(t+\gamma)^{\gamma-1}\mathrm{d}t
\leq 2 + \frac{(k+1+\gamma)^{\gamma}}{\gamma}.
\end{aligned}
\end{equation*}
This completes the proof.

\end{appendices}

\bibliographystyle{plain}
\bibliography{references/Ref_iBPG}

\end{document}